\documentclass[a4paper]{amsart}

\usepackage[tbtags]{mathtools}
\usepackage{amsmath}
\usepackage{amssymb}
\usepackage{mathrsfs}
\usepackage{amsfonts}
\usepackage{eucal}
\usepackage{yfonts}
\usepackage[skip=4pt, indent=10pt]{parskip}

\usepackage{enumerate}
\usepackage{amsthm}
\usepackage[colorlinks=true,
            linkcolor=blue,
            anchorcolor=blue,
            citecolor=blue]{hyperref}
\usepackage{enumitem}
\setlist[itemize]{labelindent=\parindent,leftmargin=*,ref = (\arabic*)}
\setlist[enumerate]{leftmargin=*,label= (\arabic*),ref = (\arabic*)}

\newtheorem{theorem}{Theorem}[section]
\newtheorem{lemma}[theorem]{Lemma}
\newtheorem{corollary}[theorem]{Corollary}
\newtheorem{proposition}[theorem]{Proposition}

\newtheorem{alphthm}{Theorem}			

\newtheorem{alphprop}[alphthm]{Proposition}

\theoremstyle{definition}
\newtheorem{definition}[theorem]{Definition}
\newtheorem{example}[theorem]{Example}

\newtheorem{alphdef}[alphthm]{Definition}			

\theoremstyle{remark}
\newtheorem{remark}[theorem]{Remark}

\numberwithin{equation}{section}

\usepackage[numbers]{natbib}

\usepackage{nicematrix}

\usepackage{tikz}
\usetikzlibrary[cd]

\def\NN{\mathbb{N}}
\def\RR{\mathbb{R}}
\def\CC{\mathbb{C}}

\def\G{\mathcal{G}}
\def\Gz{\mathcal{G}^{(0)}}
\def\A{\mathcal{A}}

\def\L{\mathscr{L}}
\def\B{\mathscr{B}}
\def\K{\mathscr{K}}

\def\s{\mathrm{s}}
\def\r{\mathrm{r}}

\def\supp{\mathrm{supp}}

\def\diam{\mathrm{diam}}
\def\res{\mathrm{Res}}
\def\Ad{\mathrm{Ad}}

\def\Nd{\mathcal{N}}

\begin{document}

\title{Quasi-locality for \'{e}tale groupoids}
\author{Baojie Jiang}
\author{Jiawen Zhang}
\author{Jianguo Zhang}

\address[Baojie Jiang]{College of Mathematics and Statistics, Chongqing University, Chongqing 401331, China.}
\email{jiangbaojie@gmail.com}

\address[Jiawen Zhang]{School of Mathematical Sciences, Fudan University, 220 Handan Road, Shanghai, 200433, China.}
\email{jiawenzhang@fudan.edu.cn}

\address[Jianguo Zhang]{School of Mathematics and Statistics,
    Shaanxi Normal University, Xi’an 710119, China.}
\email{jgzhang@snnu.edu.cn}

\thanks{Baojie Jiang was supported by NSFC12001066 and NSFC12071183. Jiawen Zhang was supported by NSFC11871342. Jianguo Zhang was supported by NSFC12171156, 12271165.}

\begin{abstract}
    Let $\G$ be a locally compact \'{e}tale groupoid and $\L(L^2(\G))$ be the $C^*$-algebra of adjointable operators on the Hilbert $C^*$-module $L^2(\G)$. In this paper, we discover a notion called quasi-locality for operators in $\L(L^2(\G))$, generalising the metric space case introduced by Roe. Our main result shows that when $\G$ is additionally $\sigma$-compact and amenable, an equivariant operator in $\L(L^2(\G))$ belongs to the reduced groupoid $C^*$-algebra $C^*_r(\G)$ if and only if it is quasi-local. This provides a practical approach to describe elements in $C^*_r(\G)$ using coarse geometry. Our main tool is a description for operators in $\L(L^2(\G))$ via their slices with the same philosophy to the computer tomography. As applications, we recover a result by \v{S}pakula and the second-named author in the metric space case,
    and deduce new characterisations for reduced crossed products and uniform Roe algebras for groupoids.
\end{abstract}


\date{\today}

\maketitle


\section{introduction}\label{sec.introduction}

In the last few decades, there has been increased interest in the study of $C^*$-algebras associated to groupoids. The story should be traced back to Renault~\cite{R1980-book-a}, who introduced the notion of groupoid $C^*$-algebras as well as several counterparts for locally compact topological groupoids. These constructions are extremely powerful weapons to produce $C^*$-algebras, and provide a unified approach to a number of classic $C^*$-algebras including group $C^*$-algebras, crossed products, $C^*$-algebras associated to foliations, uniform Roe algebras, \emph{etc}.

The groupoid $C^*$-algebra also provides a useful and concrete model for many classes of $C^*$-algebras, which is helpful to analyse their $C^*$-algebraic structures. Along this way, Renault~\cite{R1980-book-a, zbMATH05573102} introduced the notion of Cartan subalgebras in $C^*$-algebras, motivated by Feldman and Moore's works~\cite{zbMATH03575853} on von Neumann algebras. Since then, the work of Kumjian, Renault, Barlak, Li and others (\emph{e.g.}, \cite{BarlakLi2017,BarlakLi2020,kumjian_1986,Li2020,zbMATH07076096,zbMATH05573102}) showed that a large class of $C^*$-algebras can be realised as twisted groupoid $C^*$-algebras, based on the theory of Cartan subalgebras. Consequently, this allows people to use the groupoid models to study their ideal structure, nuclearity, the Universal Coefficient Theorem (UCT), \emph{etc}.

On the other hand, the groupoid $C^*$-algebra plays a central role in (higher) index theory as a proxy for the algebra of continuous functions and its $K$-theory is the natural locus for generalised indices (see, \emph{e.g.}, \cite{zbMATH01568733}).
For example, the equivariant index of a lifted elliptic differential operator on the universal cover of a closed manifold $M$ belongs to the $K$-theory of the group $C^*$-algebra of $\pi_1(M)$~\cite{zbMATH03855929, zbMATH03834849}; while the index of an elliptic differential operator on an open manifold belongs to the $K$-theory of its Roe algebra \cite{roe1988:index-thm-on-open-mfds}, which is isomorphic to certain crossed product of the groupoid $C^*$-algebra for the associated coarse groupoid \cite{STY-article-2002a}.

These indices provide geometric and topological information of the underlying manifold. Hence the computation for the $K$-theory of groupoid $C^*$-algebras lies at the heart of (higher) index theory, and the famous Baum-Connes conjecture provides an efficient and practical approach (see~\cite{zbMATH01853001} for an excellent introduction, and  \cite{zbMATH07217273} for a more recent survey).
A number of excellent works have been made around this conjecture in the last few decades (see a compresensive list of reference in \cite{zbMATH07217273}), while the general case is still widely open. This suggests that studying the structure of groupoid $C^*$-algebras will help to understand the indices of various operators, and hence to enhance our understanding of geometric properties of the underlying manifolds.

In spite of the importance of groupoid $C^*$-algebras, it is usually hard to determine whether a given element belongs to this algebra. According to the definition (see Section~\ref{ssec.groupoidAlgebra}), we have to construct a sequence of compactly supported functions on the groupoid to approximate the given element in operator norm, while unfortunately there is no routine procedure for the construction.

The aim of this paper is to provide a practical approach to characterise elements in groupoid $C^*$-algebras using coarse geometry. For simplicity, we will only focus on locally compact \'{e}tale groupoids (see Section \ref{ssec.groupoid} for a precise definition).

\subsection{A motivating example: the coarse groupoid}\label{ssec:motivating example}

We start with an illuminating example which is the motivation of this work. Let $(X,d)$ be a discrete metric space with bounded geometry. Regarding operators on $\ell^2(X)$ as $X$-by-$X$ matrices, we say that such an operator has \emph{finite propagation} if the non-zero entries appear only in an \emph{entourage}, \emph{i.e.}, a band of finite width (measured by the metric $d$ on $X$) around the main diagonal (see Section \ref{ssec.motivation} for full details). The finite propagation operators form a $\ast$-subalgebra $\CC_u[X]$ of $\B(\ell^2(X))$, and its closure $C^*_u(X)$ is called the \emph{uniform Roe algebra of $X$}, introduced by Roe in \cite{roe1988:index-thm-on-open-mfds}.

To characterise operators in $C^*_u(X)$, Roe \cite{roe1988:index-thm-on-open-mfds} also introduced a notion of quasi-locality for operators in $\B(\ell^2(X))$. Roughly speaking, an operator is called \emph{quasi-local} if for any $\varepsilon>0$ we can find an entourage such that for any block sitting outside this entourage, its restriction on the block has norm less than $\varepsilon$ (see Section \ref{ssec.motivation}). The collection of all quasi-local operators on $\ell^2(X)$ forms a $C^*$-algebra, called the \emph{uniform quasi-local algebra of $X$} and denoted by $C^*_{uq}(X)$. 

Clearly $C^*_u(X)$ is a subalgebra of $C^*_{uq}(X)$, and it is proved by \v{S}pakula and the second-named author \cite{SpakulaZhang} that they coincide when $X$ has Property A. Quasi-locality is a coarse geometric property in the sense that the algebra $C^*_{uq}(X)$ is preserved under coarse equivalence and moreover, this property is easier to verify than the case of the uniform Roe algebra. Therefore, the quasi-local characterisation is crucial in the work of Engel \cite{engel2015index, EngelRoughIndex} on the index theory of pseudo-differential operators, and also in the work of White and Willett \cite{MR4167028} on classifying Cartan subalgebras of Roe algebras, with some applications on the associated rigidity problem as well.

To build a bridge to the groupoid theory, Skandalis, Tu and Yu \cite{STY-article-2002a} introduced a locally compact \'{e}tale groupoid for $(X,d)$, called the \emph{coarse groupoid} $G(X)$ (see Example \ref{ssec:coarse groupoids pre}). Note that any compact subset in $G(X)$ is contained in the closure of some entourage, and hence it is clear that $C_c(G(X))$ is $\ast$-isomorphic to $\CC_u[X]$. Moreover taking closures, the reduced groupoid $C^*$-algebra $C^*_r(G(X))$ is $C^*$-isomorphic to the uniform Roe algebra $C^*_u(X)$. This suggests us to consider compactly supported operators in the place of finite propagation operators for general groupoids, and hence the reduced groupoid $C^*$-algebra should correspond to certain version of the uniform Roe algebra. Also recall from \cite{STY-article-2002a} that $X$ has Yu's Property A if and only if $G(X)$ is amenable.

Now a natural question is to ask how to recover the notion of quasi-locality for operators in $\B(\ell^2(X))$ using the language of coarse groupoids. This will also be our starting point to introduce quasi-locality for general locally compact \'{e}tale groupoids. Parallel to the situation of uniform Roe algebras, the practical approach we are looking for to describe elements in groupoid $C^*$-algebras will be achieved by the coarse geometric property of quasi-locality.

\subsection{Our strategy for quasi-locality}\label{ssec:strategy}

To introduce the notion of quasi-locality for general groupoids, we would like to mimic the approach for metric spaces. Recall from Section \ref{ssec:motivating example} that for a discrete metric space $(X,d)$ with bounded geometry, we have a natural ambient $C^*$-algebra $\B(\ell^2(X))$. For operators in $\B(\ell^2(X))$, we have the notion of finite propagation and quasi-locality to form the associated $C^*$-algebras. The main difficulty to generalise the notion of quasi-locality to general groupoids is that firstly we have to draw the border of a proper class of operators so that the notion of quasi-locality thereon has the chance to recover the reduced groupoid $C^*$-algebra.


Let us explain our strategy in details. For a locally compact \'{e}tale groupoid $\G$, we follow the steps below:

\textbf{Step I.} We need to choose a suitable candidate for the ambient $C^*$-algebra $\A$. The criterion is that $\A$ should be large enough to include the $\ast$-algebra $C_c(\G)$, while not too large so that a well-defined notion of quasi-locality has the chance to characterise the reduced groupoid $C^*$-algebra $C^*_r(\G)$ when $\G$ is ``well-behaved''. Moreover in the case of coarse groupoid, the ambient $C^*$-algebra $\A$ should be (\emph{almost}\footnote{As we shall see in Example \ref{eg:coarse groupoid for LL2G}, $\A$ is smaller than $\B(\ell^2(X))$. However this does not cause any trouble when considering quasi-local operators (see Example \ref{ex:coarse groupoid again}).}) the same as $\B(\ell^2(X))$.

\textbf{Step II.} We would like to define the notion of compact support and quasi-locality for operators in the ambient $C^*$-algebra $\A$ chosen in \textbf{Step I} so that compactly supported operators are nothing but those from $C_c(\G)$. Moreover in the case of the coarse groupoid, the notion of quasi-locality should recover the uniform quasi-local algebra $C^*_{uq}(X)$ introduced by Roe (see also \cite{li2021quasi}).

\textbf{Step III.} Finally we aim to show that when $\G$ is amenable (which is equivalent to Property A in the case of the coarse groupoid), elements in the reduced groupoid $C^*$-algebra $C^*_r(\G)$ can be characterised by the property of quasi-locality, which fulfils our task.


\subsection{Our main contributions}\label{ssec:intro main contributions}
Following the strategy from Section \ref{ssec:strategy}, now we introduce our main contributions to achieve each step above. Assume that $\G$ is a locally compact \'{e}tale groupoid with unit space $\Gz$. For each $x\in \Gz$, denote $\G_x$ the set consisting of elements in $\G$ with source $x$.

For \textbf{Step I}, there are two natural candidates for the ambient $C^*$-algebras coming from the definition of the reduce groupoid $C^*$-algebra. Recall that the $\ast$-algebra $C_c(\G)$ can be faithfully represented as adjointable operators on the Hilbert $C_0(\Gz)$-module $L^2(\G)$, and the reduced groupoid $C^*$-algebra $C^*_r(\G)$ is the closure of $C_c(\G)$ with respect to the operator norm in $\L(L^2(\G))$. Considering $\G$ as a disjoint union of its source fibres $\{\G_x\}_{x\in \Gz}$, this norm can also be recovered by representing $C_c(\G)$ in $\prod_{x\in \Gz} \B(\ell^2(\G_x))$ (see Section \ref{ssec.groupoidAlgebra} for full details). This suggests that either $\L(L^2(\G))$ or $\prod_{x\in \Gz} \B(\ell^2(\G_x))$ might be a suitable candidate for the ambient $C^*$-algebra.

By checking the case of the coarse groupoid, we realise that $\prod_{x\in \Gz} \B(\ell^2(\G_x))$ is too large due to the lack of control on different coordinates, and $\L(L^2(\G))$ seems more suitable because it contains certain topological information of $\G$. However, we are in a dilemma since elements in $\prod_{x\in \Gz} \B(\ell^2(\G_x))$ are much easier to describe than those in $\L(L^2(\G))$.

To overcome the issue, note that there is a natural embedding $\Phi: \L(L^2(\G)) \to \prod_{x\in \Gz} \B(\ell^2(\G_x))$, called the \emph{slicing map} (see (\ref{eq.PhiMap})). We realise that operators in the image of $\Phi$ satisfy a condition which can be regarded as a vector-wise version of certain quasi-locality (Definition \ref{defn:vectorwise quasiloc}). Hence consulting the notion of operator fibre space introduced by Austin and the second-named author, we achieve the following (see Section \ref{ssec:operators in LL2G} for precise definitions):

\begin{alphthm}[Theorem \ref{cor:char for LL2G}]\label{introthm:char for LL2G}
    Let $\G$ be a locally compact \'{e}tale groupoid with unit space $\Gz$. For $(T_{x})_{x\in\Gz} \in \prod_{x\in\Gz}\B\left(\ell^2(\G_{x})\right)$, the following are equivalent:
    \begin{enumerate}
        \item $(T_{x})_{x\in\Gz}$ belongs to $\Phi\big(\L(L^2(\G))\big)$;
        \item the map $x\mapsto T_x$ is a continuous section of the associated operator fibre space, and $(T_{x})_{x\in\Gz}, (T^*_{x})_{x\in\Gz}$ are vector-wise uniformly quasi-local.
    \end{enumerate}
\end{alphthm}

Theorem \ref{introthm:char for LL2G} provides a characterisation for operators in $\L(L^2(\G))$ via their slices with the same philosophy to the computer tomography, which describes an object using a family of its slices. This will be crucial in the proof of Theorem \ref{introthm:main result} below as well as in the computation of $\L(L^2(\G))$ for concrete examples. We also obtain a simplified version of Theorem \ref{introthm:char for LL2G} when the unit space contains a dense subset (see Proposition \ref{lem.charOfBL2G2-densesubset} and Corollary \ref{cor:char2 for LL2G}), which happens in a number of cases including coarse groupoids.

By checking coarse groupoids again, we realise that the algebra $\L(L^2(\G))$ might still be too large. Note that elements in $C^*_r(\G)$ admit an extra equivariant property: an operator in $\L(L^2(\G))$ is called \emph{$\G$-equivariant} if it commutes with elements in the image of the right regular representation (see Definition \ref{def.GEquivariant}). Denote the subalgebra of $\G$-equivariant operators in $\L(L^2(\G))$ by $\L(L^2(\G))^{\G}$. Thanks to Theorem \ref{introthm:char for LL2G}, we manage to compute $\L(L^2(\G))^{\G}$ for the coarse groupoid (see Example \ref{eg:coarse groupoid for LL2G}). Moreover, we obtain the following realisation for equivariant operators:

\begin{alphprop}[Proposition \ref{prop.leftConvolverAndGEquivariant}]\label{introprop:realisation of equivariance}
    Let $\G$ be a locally compact \'{e}tale groupoid. For any $\G$-equivariant $T \in \L(L^2(\G))$, there exists a unique function $f_T\in C_b(\G)$ such that $T$ is the convolution operator by $f_T$.
\end{alphprop}

The clues above suggest that $\L(L^2(\G))^{\G}$ should be a suitable candidate for the ambient $C^*$-algebra. We remark that Proposition \ref{introprop:realisation of equivariance} will also be crucial to prove Theorem \ref{introthm:main result} below since it transfers operators to functions.

Next, we move to \textbf{Step II}. Recall that we need to define the notion of compact support and quasi-locality for operators in the ambient $C^*$-algebra $\L(L^2(\G))^{\G}$ so that those with compact support recover the algebra $C_c(\G)$. Indeed, the definition we discover can be naturally defined on the larger algebra $\L(L^2(\G))$. Inspired by the case of coarse groupoids, we introduce the following key notion of this paper:

\begin{alphdef}[Definition \ref{defn:separated} and Definition \ref{def.operatorCompactSupportAndQuasiLocal}]\label{introdefn:ql}
    Let $\G$ be a locally compact \'etale groupoid and $T\in \L(L^2(\G))$. We define the following:
    \begin{enumerate}
        \item For a subset $K \subseteq \G$, functions $f,g\in C_{b}(\G)$ are called \emph{$K$-separated} if we have $\big(K\cdot \supp(f)\big)\cap \supp(g)=\emptyset$ and $\supp(f)\cap \big(K\cdot \supp(g)\big)=\emptyset$.
        \item  $T$ is called \emph{compactly supported} if there exists a compact subset $K \subseteq \G$ such that $gTf=0$ for any $K$-separated functions $f,g\in C_{b}(\G)$.
        \item $T$ is called \emph{quasi-local} if for any $\varepsilon>0$, there exists a compact subset $K\subset \G$ such that for any $K$-separated functions $f,g\in C_{b}(\G)$ we have $\|gTf\| < \varepsilon\|g\|_{\infty}\|f\|_{\infty}$.
    \end{enumerate}
\end{alphdef}

Thanks to Theorem \ref{introthm:char for LL2G}, we obtain a characterisation for compactly supported and quasi-local operators in $\L(L^2(\G))$ using their slices (see Proposition \ref{compactfiberwisely} and Corollary \ref{cor:char for quasi-locality}). In particular for a quasi-local operator, each of its slice is an ordinary quasi-local operator on a discrete metric space. Hence our notion of quasi-locality reflects the coarse geometry of the underlying metric family.
Applying this characterisation to the case of the coarse groupoid $G(X)$ for some metric space $X$, we verify that Definition \ref{introdefn:ql} is compatible with the notion of finite propagation and quasi-locality for operators in $\B(\ell^2(X))$ as desired (see Example \ref{ex:coarse groupoid again}).

Moreover under the assumption of equivariance, we reach the following as a consequence of Proposition \ref{introprop:realisation of equivariance}:

\begin{alphprop}[Proposition \ref{Roereduced}]\label{introprop:compatibility}
    Let $\G$ be a locally compact \'{e}tale groupoid. A $\G$-equivariant operator $T \in \L(L^2(\G))$ is compactly supported \emph{if and only if} $T$ belongs to the image of $C_c(\G)$ under the left regular representation. Hence the reduced groupoid $C^*$-algebra $C^*_r(\G)$ can be recovered as the norm closure of all $\G$-equivariant compactly supported operators in $\L(L^2(\G))$.
\end{alphprop}

Hence we fulfil the task of \textbf{Step II}.
For later use, we denote by $\CC_u[\G]$ the $\ast$-subalgebra in $\L(L^2(\G))$ consisting of all compactly supported operators and define the \emph{uniform Roe algebra $C^*_u(\G)$ of $\G$} to be its norm closure in $\L(L^2(\G))$. Also denote by $C^{*}_{uq}(\G)$ the set of all quasi-local operators in $\L(L^2(\G))$, which forms a $C^*$-algebra and is called the \emph{uniform quasi-local algebra of $\G$}. We also consider their equivariant counterparts and denote by $\CC_u[\G]^{\G}$, $C^*_u(\G)^{\G}$ and $C^{*}_{uq}(\G)^{\G}$ the subalgebras consisting of equivariant operators in $\CC_u[\G]$, $C^*_u(\G)$ and $C^{*}_{uq}(\G)$, respectively.

Having established all necessary ingredients, finally we reach \textbf{Step III} and manage to prove the following main result of this paper:

\begin{alphthm}[Theorem \ref{Main-theorem}]\label{introthm:main result}
    Let $\G$ be a locally compact, $\sigma$-compact and \'{e}tale groupoid.
    If $\G$ is amenable, then we have $C^{*}_r(\G)= C^*_u(\G)^{\G} = C^{*}_{uq}(\G)^{\G}$.
\end{alphthm}

Theorem \ref{introthm:main result} shows that for a $\G$-equivariant operator in $\L(L^2(\G))$, it is quasi-local if and only if it belongs to the reduced groupoid $C^*$-algebra $C^*_r(\G)$. Note that quasi-locality reflects the coarse geometry of the underlying object and is easier to verify than the situation of $C^*_r(\G)$, hence Theorem \ref{introthm:main result} provides a practical and coarse geometric approach to characterise elements in the reduced groupoid $C^*$-algebra for amenable groupoids as desired.

As an example when $\G$ is the coarse groupoid associated to a discrete metric space with bounded geometry, Theorem \ref{Main-theorem} recovers the main result of \cite{SpakulaZhang} in the Hilbert space case (see Section \ref{ssec:coarse groupoids}). Moreover when $\G$ is the transformation groupoid associated to a discrete group acting on a compact Hausdorff space, we obtain a new characterisation for elements in the reduced crossed product (see Section \ref{ssec:transformation groupoid}).

The proof of Theorem \ref{introthm:main result} relies heavily on the coarse geometry of groupoids (see~\cite{MaWu} and~\cite{zbMATH07009242}). Roughly speaking, we assign a length function on the groupoid thanks to the $\sigma$-compactness, which induces a metric on each source fibre. Hence we can appeal to the weapon of coarse geometry for a family of metric spaces, and consult the idea of \cite[Theorem 3.3]{SpakulaZhang}. Here we would like to clarify that the proof for Theorem \ref{introthm:main result} is \emph{not} merely a family version of that for \cite[Theorem 3.3]{SpakulaZhang} since we have certain continuity restriction from the topology of the groupoid, and technical analysis is also required especially when the unit space is no longer compact.

\subsection{Beyond equivariance: uniform Roe algebras}

Note that Theorem \ref{introthm:main result} only discusses the situation of equivariant operators, while our notion of quasi-locality (Definition \ref{introdefn:ql}) is designed for general operators in $\L(L^2(\G))$. Hence it is natural to ask whether we can obtain a similar result for the non-equivariant case.

As revealed by the work of Anantharaman-Delaroche \cite{ADExact}, we can transfer the general case to the equivariant one by considering the semi-direct product. More precisely, for a locally compact \'{e}tale groupoid $\G$ we consider the fibrewise Stone-\v{C}ech compactification $\beta_{\r} \G$ and the associated semi-direct product $\beta_{\r} \G \rtimes \G$ (see Section \ref{sssec:groupoid actions}). Anantharaman-Delaroche showed in \cite{ADExact} that the reduced groupoid $C^*$-algebra $C^*_r(\beta_{\r} \G \rtimes \G)$ is $C^*$-isomorphic\footnote{More precisely, Anantharaman-Delaroche \cite{ADExact} defined the uniform Roe algebra $C^*_u(\G)$ in a different way, while we show in Lemma \ref{lem:AD=Cu} that it coincides with our definition.} to the uniform Roe algebra $C^*_u(\G)$.

Furthermore, we construct an embedding from $\L(L^2(\beta_{\r} \G \rtimes \G))^{\beta_{\r} \G \rtimes \G}$ to $\L(L^2(\G))$ in Section \ref{ssec:eq to neq} (see Corollary \ref{cor:Theta conclusion}), and apply Theorem \ref{introthm:char for LL2G} to show that it provides a $C^*$-isomorphism between the equivariant uniform quasi-local algebra of $\beta_{\r} \G \rtimes \G$ and the uniform quasi-local algebra of $\G$. Finally thanks to Theorem \ref{introthm:main result}, we obtain the following concrete description for elements in the uniform Roe algebra $C^*_u(\G)$ (see Section \ref{sssec:groupoid actions} for precise definitions):

\begin{alphthm}[Theorem \ref{thm:general case}]\label{introthm:main result general case}
    Let $\G$ be a locally compact, $\sigma$-compact and \'{e}tale groupoid. Suppose $\G$ is either strongly amenable at infinity, or secondly countable weakly inner amenable and $C^*$-exact. Then we have $C^*_u(\G) = C^*_{uq}(\G)$.
\end{alphthm}

\subsection{Organisation}

The paper is organised as follows. In Section \ref{sec.Preliminaries}, we recall necessary notions in groupoid theory and coarse geometry together with several examples. Section \ref{sec.charForOperator} is devoted to \textbf{Step I} where we provide the required characterisation for operators in the Hilbert $C^*$-module $\L(L^2(\G))$ (Theorem \ref{introthm:char for LL2G}), discuss the notion of equivariance and finally prove Proposition \ref{introprop:realisation of equivariance}. In Section \ref{sec.quasi}, we introduce the key notion of compact support and quasi-locality (Definition \ref{introdefn:ql}) and recover the algebra $C_c(\G)$ (\emph{i.e.}, Proposition \ref{introprop:compatibility}), which fulfils \textbf{Step II}. Section \ref{sec.mainThm} is devoted to \textbf{Step III} where we prove Theorem \ref{Main-theorem}, and as applications we provide several examples in Section \ref{sec.example}. Finally in Section \ref{sec:beyond equivariance}, we consider the general case and prove Theorem \ref{introthm:main result general case}.

\subsection*{Acknowledgement}
We appreciate the anonymous referee for a number of valuable suggestions to improve the paper, especially for suggesting us to use the notion of interior tensor product of Hilbert $C^*$-modules (see Remark \ref{rem:tensor product 1} and \ref{rem:Theta0512}).

\subsection*{Statements and Declarations}
The authors have no relevant financial or non-financial interests to disclose. The authors have no competing interests to declare that are relevant to the content of this article. All authors certify that they have no affiliations with or involvement in any organization or entity with any financial interest or non-financial interest in the subject matter or materials discussed in this manuscript. The authors have no financial or proprietary interests in any material discussed in this article.

\section{Preliminaries}\label{sec.Preliminaries}

\subsection{Standard notation} Here we collect the notation used throughout the paper.

To simplify the terminology, we always assume that locally compact spaces are Hausdorff. Given a locally compact space $X$, we denote by $C(X)$ the set of complex-valued continuous functions on $X$, and by $C_b(X)$ the subset of bounded continuous functions on $X$. Recall that the \textit{support} of a function $f\in C(X)$ is the closure of $\{ x\in X: f(x)\neq 0\}$, written as $\supp(f)$, and denote by $C_c(X)$ the set of continuous functions with compact support. We also denote by $C_0(X)$ the set of continuous functions vanishing at infinity, which is the closure of $C_c(X)$ with respect to the supremum norm $\|f\|_\infty:=\sup\{|f(x)|: x\in X\}$.

When $X$ is discrete, denote $\ell^\infty(X):=C_b(X)$ and $\ell^2(X)$ the Hilbert space of complex-valued square-summable functions on $X$. To be compatible with the setting of Hilbert $C^*$-modules, we always assume that the Hilbert space inner product $\langle \eta, \xi \rangle$ will be taken to be linear in the variable $\xi$ and conjugate-linear in $\eta$. Denote by $\B(\ell^2(X))$ the $C^*$-algebra of all bounded linear operators on $\ell^2(X)$, and $\K(\ell^2(X))$ the $C^*$-subalgebra of all compact operators.

\subsection{Basic notions for Groupoids}\label{ssec.groupoid}

Let us start with some basic notions and terminology on groupoids. For details we refer to \cite{R1980-book-a}.

Recall that a \emph{groupoid} is a small category, in which every morphism is invertible. Roughly speaking, a groupoid consists of a set $\G$, a subset $\Gz$ called the \emph{unit space}, two maps $\s, \r: \G \to \Gz$ called the \emph{source} and \emph{range} maps respectively, a \emph{composition law}:
\[
    \G^{(2)}\coloneqq \{(\gamma_1,\gamma_2) \in \G \times \G: \s(\gamma_1)=\r(\gamma_2)\}\ni(\gamma_1,\gamma_2) \mapsto \gamma_1\gamma_2 \in \G,
\]
and an \emph{inverse} map $\gamma \mapsto \gamma^{-1}$. These operations satisfy a couple of axioms, including the associativity law and the fact that elements in $\Gz$ act as units. For $x \in \Gz$, denote $\G^x:=\r^{-1}(x)$ and $\G_x:=\s^{-1}(x)$.
A subset $Y \subseteq \Gz$ is called \emph{invariant} if $\r^{-1}(Y)=\s^{-1}(Y)$. 
For $A, B \subseteq \G$, we denote
\begin{align*}
    A^{-1} & \coloneqq \{ \gamma\in\G: \gamma^{-1} \in A \},                                                                                                  \\
    AB     & \coloneqq \{ \gamma\in\G: \gamma=\gamma_{1}\gamma_{2}\text{ where }\gamma_{1} \in A,\gamma_{2}\in B \text{ and } \s(\gamma_1) = \r(\gamma_2) \}.
\end{align*}
We say that $A \subseteq \G$ is \emph{symmetric} if $A=A^{-1}$.



A groupoid $\G$ is \textit{principal} if the map $(\s,\r)$ from $\G$ into $\Gz\times\Gz$ is one-to-one, and $\G$ is \textit{transitive} if the map $(\s,\r)$ is onto.

A \textit{locally compact groupoid} is a groupoid $\G$ endowed with a locally compact topology for which composition, inversion, source and range maps are continuous with respect to the induced topologies.

We say that a locally compact groupoid $\G$ is \textit{\'{e}tale} if the range (and hence the source) map is a local homeomorphism, \emph{i.e.}, for any $\gamma \in \G$ there exists a neighbourhood $U$ of $\gamma$ such that $\r(U)$ is open and $\r|_U$ is a homeomorphism. In this case, fibers $\G_x$ and $\G^x$ with the induced topologies are discrete and $\Gz$ is open in $\G$. Throughout the paper, we shall limit ourselves to the \'{e}tale case.


For a locally compact groupoid $\G$, a subset $A \subseteq \G$ is called a \textit{bisection} if the restrictions of $\s, \r$ to $A$ are injective.
It follows from definition that for a locally compact \'{e}tale groupoid $\G$, all open bisections form a basis for the topology of $\G$. As a direct consequence, we obtain that for such a groupoid $\G$, any function $f \in C_c(\G)$ can be written as a linear combination of continuous functions whose supports are contained in open precompact bisections.

\subsection{Groupoid $C^*$-algebras and Hilbert $C^*$-modules}\label{ssec.groupoidAlgebra}

Here we recall the notion of reduced groupoid $C^*$-algebras in the \'{e}tale case. Given a locally compact \'{e}tale groupoid $\G$, note that the space $C_c(\G)$ can be formed into a $*$-algebra with the following operations: for $f,g\in C_c(\G)$,
\begin{align}
    (f*g)(\gamma) & \coloneqq \sum_{\alpha\in \G_{s(\gamma)}}f(\gamma\alpha^{-1})g(\alpha), \label{EQ:convolution} \\
    f^*(\gamma)   & \coloneqq \overline{f(\gamma^{-1})}. \label{EQ:star}
\end{align}
The function $f \ast g \in C_c(\G)$ is called the \emph{convolution} of $f$ and $g$.

Recall that for each $x\in\Gz$ the \emph{left regular representation at $x$}, denoted by $\lambda_x: C_c(\G) \to \B(\ell^2(\G_x))$, is defined as follows:
\begin{equation}\label{EQ:reduced algebra defn}
    \lambda_x(f)(\xi)(\gamma)=\sum_{\alpha\in \G_x}f(\gamma\alpha^{-1})\xi(\alpha), \quad \text{where } f\in C_{c}(\G) \text{ and }\xi\in\ell^{2}(\G_{x}).
\end{equation}
It is routine to check that $\lambda_x$ is a well-defined $\ast$-homomorphism.
The \textit{reduced norm} on $C_{c}(\G)$ is defined by
\[
    \| f \|_{r}\coloneqq \sup_{x\in \Gz} \| \lambda_{x}(f) \|,
\]
and the \textit{reduced groupoid $C^{*}$-algebra} $C^{*}_r(\G)$ is defined to be the completion of the $*$-algebra $C_{c}(\G)$ with respect to the reduced norm $\| \cdot \|_{r}$. It is clear that each left regular representation $\lambda_x$ can be extended automatically to a $C^*$-homomorphism $\lambda_x: C^*_r(\G) \to \B(\ell^2(\G_x))$.

In the following, we would like to provide an alternative approach to define the reduced groupoid $C^*$-algebras using the language of Hilbert $C^*$-modules, which motivates our definition of quasi-locality for groupoids. Here we recall the necessary notions on Hilbert $C^*$-modules, and guide readers to \cite{L1995-book-a} for details.

\begin{definition}[\cite{L1995-book-a}]\label{def.HilbertCModule}
    Let $A$ be a $C^*$-algebra.
    An \textit{inner-product $A$-module} is a linear space $E$ which is a right $A$-module (with compatible scalar multiplication: $\lambda(\xi a) = (\lambda\xi)a = \xi(\lambda a)$ for $\xi\in E$, $a\in A$ and $\lambda\in\CC$), together with a map $(\xi,\eta)\mapsto \langle \xi , \eta \rangle$ from $E\times E$ to $A$ such that
    \begin{enumerate}
        \item $\langle \xi , \alpha \eta +\beta \zeta \rangle = \alpha \langle \xi , \eta \rangle+\beta \langle \xi , \zeta \rangle$ for $\xi,\eta,\zeta\in E$ and $\alpha,\beta\in \CC$;
        \item $\langle \xi , \eta a \rangle = \langle \xi , \eta \rangle a$ for $\xi,\eta\in E$ and $a\in A$;
        \item $\langle \xi , \eta \rangle = \langle \eta , \xi \rangle^*$ for $\xi,\eta\in E$;
        \item $\langle \xi , \xi \rangle \geq 0$ for $\xi\in E$, and $\langle \xi , \xi \rangle=0$ if and only if $\xi=0$.
    \end{enumerate}
\end{definition}

For $\xi\in E$, we define its norm $\| \xi \|_E := \| \langle \xi , \xi \rangle \|^{1/2}$, also denoted by $\|\xi\|$ when there is no ambiguity.
An inner-product $A$-module which is complete with respect to this norm is called a \textit{Hilbert $A$-module}.

Recall that an inner-product $A$-module $E$ admits an $A$-valued norm $|\cdot |$ given by $|\xi| = \langle \xi , \xi \rangle^{1/2}$. Note that $|\langle \xi , \eta \rangle|\leq \| \xi \| \cdot |\eta|$ for $\xi,\eta \in E$ (see \cite[Proposition 1.1]{L1995-book-a}), and hence $|\xi a|\leq \| \xi \|\cdot |a|$ for $\xi\in E$ and $a\in A$.

\begin{definition}\label{defn:adjointable operators}
    Let $A$ be a $C^*$-algebra and $E$ be a Hilbert $A$-module. A map $T: E \rightarrow E$ is called \emph{adjointable} if there exists a map $T^*\colon E\rightarrow E$ such that
    \[
        \langle \eta , T(\xi) \rangle=\langle T^*(\eta), \xi \rangle\text{ for all }\xi,\eta\in E.
    \]
    It is easy to see that adjointable operators are $A$-linear and bounded with respect to the norm on $E$. Denote by $\L(E)$ the set of all adjointable operators on $E$, which forms a $C^*$-algebra.
\end{definition}

The following standard result will be used in the sequel:

\begin{proposition}[{\cite[Proposition 1.2]{L1995-book-a}}]\label{prop.Proposition 1.2}
    Let $E$ be a Hilbert $A$-module.
    For $T\in\L(E)$ and $\xi\in E$, we have $|T(\xi)|^2\leq \| T \|^2 |\xi|^2$ and $|T(\xi)|\leq \| T \| \cdot |\xi|$.
\end{proposition}

Now we consider the Hilbert $C^*$-module associated to a locally compact \'{e}tale groupoid $\G$.
Let $L^2(\G)$ be the Hilbert $C^*$-module over $C_0(\Gz)$ obtained by taking completion of $C_c(\G)$ with respect to the $C_0(\Gz)$-valued inner product:
\[
    \langle \eta,\xi \rangle(x)\coloneqq \sum_{\gamma\in \G_{x}}\overline{\eta(\gamma)}\xi(\gamma),
\]
and the right $C_0(\Gz)$-module structure is given by
\[
    (\xi f)(\gamma)\coloneqq \xi(\gamma) f(\s(\gamma)).
\]
Due to \'{e}taleness, for any $\xi \in C_c(\G)$ we have
\[
    \|\xi\|^2_{L^2(\G)}=\sup_{x\in\Gz}\sum_{\gamma\in\G_x}|\xi(\gamma)|^2 \geq \sup_{\gamma \in \G} |\xi(\gamma)|^2 = \|\xi\|^2_\infty.
\]
Hence we can regard $L^2(\G)$ as a subspace in $C_0(\G)$.

Note that all the regular representations $\lambda_x$ defined in (\ref{EQ:reduced algebra defn}) can be put together to form a single representation $\Lambda: C_c(\G) \to \L(L^2(\G))$ by the formula:
\begin{equation}\label{EQ:Lambda}
    \big(\Lambda (f)\xi\big)(\gamma)\coloneqq \sum_{\alpha \in \G_{\s(\gamma)}} f(\gamma \alpha^{-1})\xi(\alpha)=((\lambda_{\s(\gamma)}f)\xi|_{\G_{\s(\gamma)}})(\gamma).
\end{equation}
And it is easy to check that for $f \in C_c(\G)$, we have $\|f\|_r=\|\Lambda(f)\|_{\L(L^2(\G))}$. Therefore, $\Lambda$ can be extended to a faithful representation $\Lambda: C^*_r(\G) \to \L(L^2(\G))$, which is called the \emph{left regular representation}.

It follows that the convolution operator $\ast: C_c(\G) \times C_c(\G) \to C_c(\G)$ can be extended to an operator (using the same notation) $\ast: C_c(\G) \times L^2(\G) \to L^2(\G)$ by (\ref{EQ:convolution}), and by definition $\Lambda(f) \xi = f \ast \xi$ for $f\in C_c(\G)$ and $\xi \in L^2(\G)$.

There is a special class of operators in $\L(L^2(\G))$, called multiplication operators. For $g\in C_b(\G)$, the \emph{multiplication operator} $M(g): L^2(\G) \to L^2(\G)$ is defined by $(M(g)\xi)(\gamma)=g(\gamma)\xi(\gamma)$ for $\xi \in L^2(\G)$ and $\gamma\in \G$. It is clear that $M(g) \in \L(L^2(\G))$, and we abbreviate $M(g)\xi$ by $g\xi$.

For later use, here we briefly recall the notion of interior tensor products for Hilbert $C^*$-modules. Readers might go to \cite[Chapter 4]{L1995-book-a} for more details.


Let $A, B$ be $C^{*}$-algebras, $E$ be a Hilbert $A$-module and $F$ be a Hilbert $B$-module. Let $\rho: A
\rightarrow\L(F)$ be a $*$-homomorphism, which turns $F$ into a left $A$-module. Denote the algebraic tensor product of $E$ and $F$ over $A$ by $E\otimes_{A}F$, which is a right $B$-module. According to \cite[Proposition 4.5]{L1995-book-a}, $E\otimes_{A}F$ is an inner-product $B$-module under the inner product given on simple tensors by
\[
\langle x_{1}\otimes y_{1},x_{2}\otimes y_{2}\rangle = \langle y_{1},\rho(\langle x_{1}, x_{2}\rangle)(y_{2})\rangle\quad \mbox{where} \quad x_{1},x_{2}\in E \mbox{~and~} y_{1},y_{2}\in F.
\]
The completion of the inner-product $B$-module $E\otimes_{A}F$ is a Hilbert $B$-module, denoted by $E\otimes_{\rho}F$, which is called the \textit{interior (\emph{or} internal) tensor product} of $E$ and $F$ with respect to $\rho$.

Furthermore, given $T\in\L(E)$, we define the tensor product $T\otimes 1$ on the algebraic tensor product $E\otimes_{A}F$ by setting $(T\otimes 1)(x\otimes y) := T(x)\otimes y$ for $x\in E$ and $y\in F$. One can easily check that this is well-defined and can be extended to an operator in $\L(E\otimes_{\rho}F)$, denoted by $T\otimes_{\rho} 1$.


\subsection{Amenable groupoids}\label{ssec.amenable}

Amenable groupoids comprise a large class of groupoids with relatively nice properties. Here we only focus on the case of \'{e}taleness, in which the notion of amenability behave quite well. A standard reference is \cite{ADR00} and another reference for the \'{e}tale case is \cite[Chapter 5.6]{BrownOzawa}.


\begin{definition}[{\cite[Chapter 2.2]{ADR00}}]\label{defn:amenable groupoids}
    A locally compact \'{e}tale groupoid $\G$ is said to be \emph{amenable} if for any $\varepsilon>0$ and compact subset $K \subseteq \G$, there exists a non-negative function $f \in C_c(\G)$ such that for any $\gamma \in K$ we have:
    \begin{itemize}
        \item $\left| 1- \sum_{\beta \in \G_{\r(\gamma)}} f(\beta) \right| < \varepsilon$;
        \item $\sum_{\beta \in \G_{\r(\gamma)}} |f(\beta) - f(\beta \gamma)| < \varepsilon$.
    \end{itemize}
\end{definition}

Using a standard normalisation argument, we have the following:

\begin{lemma}\label{lem:equiv. for amenability}
    A locally compact \'{e}tale groupoid $\G$ is amenable \emph{if and only if} for any $\varepsilon>0$ and compact subset $K \subseteq \G$, there exists a non-negative function $f \in C_c(\G)$ such that $\sum_{\beta \in \G_x} f(\beta) \leq 1$ for any $x\in \Gz$, and for any $\gamma \in K$ we have:
    \begin{itemize}
        \item $\sum_{\beta \in \G_{\r(\gamma)}} f(\beta) =1$;
        \item $\sum_{\beta \in \G_{\r(\gamma)}} |f(\beta) - f(\beta \gamma)| < \varepsilon$.
    \end{itemize}
\end{lemma}

We also need another description for amenability in terms of positive type functions. Recall that for a locally compact \'{e}tale groupoid $\G$, a function $h$ on $\G$ is of \textit{positive type} if for any $x\in\Gz$, $n\in\NN$, $\gamma_{1}, \ldots,\gamma_{n}\in\G_{x}$ and $\lambda_1,\ldots,\lambda_n\in\CC$, we have the following:
\[
    \sum_{i,j=1}^{n}\overline{\lambda_i}\lambda_jh(\gamma_{i}\gamma_{j}^{-1})\geq 0.
\]
We remark that if $h$ is of positive type, then $h(\gamma^{-1}) = \overline{h(\gamma)}$ for $\gamma \in \G$.

Combining Lemma \ref{lem:equiv. for amenability} with \cite[Proposition 2.2.13]{ADR00}, we obtain the following:

\begin{lemma}\label{lem:p.d.functions for amenability}
    A locally compact \'{e}tale groupoid $\G$ is amenable \emph{if and only if} for any $\varepsilon>0$ and compact subset $K \subseteq \G$, there exists a non-negative function $h \in C_c(\G)$ of positive type such that
    \begin{itemize}
        \item $h(x)\leq 1$ for any $x\in \Gz$, and $h(x)=1$ for any $x\in K \cap \Gz$;
        \item $|1- h(\gamma)|<\varepsilon$ for any $\gamma \in K$.
    \end{itemize}
\end{lemma}

\subsection{Notions from coarse geometry}\label{ssec.motivation}

Now we collect necessary notions from coarse geometry. For a discrete metric space $(X,d)$, denote the closed ball by $B_{X}(x,r):=\{y\in X: d(x,y)\leq r\}$ for $x\in X$ and $r\geq 0$. We say that $X$ has \emph{bounded geometry} if for every $r>0$, the number $\sup_{x\in X}\sharp B_{X}(x,r)$ is finite. Here we use the notation $\sharp S$ to denote the cardinality of a set $S$. For a subset $A\subseteq X$, we denote $\chi_{A}$ the characteristic function of $A$. For a point $x\in X$, also denote $\delta_x:=\chi_{\{x\}}$.

Let us start with the notion of Property A, which was introduced by Yu in \cite{Yu2000}. Property A plays a key role in Yu's study on the coarse Baum-Connes conjecture, and it admits a number of equivalent definitions. Here we focus on the one in terms of kernels for later use.

For a set $X$, a \emph{kernel} on $X$ is a function $k\colon X\times X\rightarrow\RR$. We say that a kernel $k$ on $X$ is of \textit{positive type} if for any $n\in\NN$, $x_{1},\ldots,x_{n}\in X$ and any $\lambda_1,\ldots,\lambda_{n}\in\RR$, we have the following:
\[
    \sum_{i,j=1}^{n}\lambda_{i}\lambda_{j}k(x_{i},x_{j})\geq 0.
\]
We say that a kernel $k$ has \textit{finite propagation} if there is $S>0$ such that $k(x,y) = 0$ whenever $d(x,y)>S$, $k$ is \textit{normalised} if $k(x,x) =1$ for every $x\in X$, and $k$ is \emph{symmetric} if $k(x,y)=k(y,x)$ for any $x,y\in X$. For $R>0$ and $\varepsilon>0$, we say that a kernel $k$ has \textit{$(R,\varepsilon)$-variation} if $|1-k(x,y)|<\varepsilon$ whenever $d(x,y)<R$.


Associated to a normalised, symmetric and positive type kernel $k$ on $X$, we have the so-called \textit{Schur multiplier} $m_k: \B(\ell^{2}(X)) \rightarrow \B(\ell^{2}(X))$. First note that every bounded linear operator $T\in\B(\ell^{2}(X))$ can be written as an $X$-by-$X$ matrix $T =(T_{x,y})_{x,y\in X}$ with coefficient $T_{x,y} := \langle \delta_x,T(\delta_y) \rangle \in \CC$. The Schur multiplier $m_k$ is defined by
\[
    m_{k}(T) \coloneqq (k(x,y)T_{x,y})_{x,y\in X}
\]
in the matrix form. It is a standard result that $m_k$ is well-defined linear bounded operator with norm $1$ (see, \emph{e.g.}, \cite[Theorem C.1.4]{BHV2008-book-a} or \cite[Theorem 5.1]{P2001-book-a}), \emph{i.e.}, $\|m_k(T)\| \leq \|T\|$ for $T \in \B(\ell^{2}(X))$.

Now we recall the definition of Property A (see, \emph{e.g.}, \cite[Proposition 3.2(v)]{T-article-2001a} or \cite[Proposition 1.2.4]{WillettPropertyA} for the equivalence to Yu's original definition):

\begin{definition}[\cite{Yu2000}]\label{defn:property A}
    Let $X$ be a discrete metric space with bounded geometry. We say that $X$ has \textit{Property A} if for any $R>0$ and $\varepsilon>0$, there exists a normalised, finite propagation, symmetric kernel $k$ on $X$ of positive type and $(R,\varepsilon)$-variation.
\end{definition}

In the following, we would like to recall $C^*$-algebras associated to a metric space. Let $(X,d)$ be a discrete metric space. Recall that there is a $\ast$-representation $M: \ell^{\infty}(X) \to \B(\ell^2(X))$ defined by point-wise multiplication, \emph{i.e.}, $(M(f)\xi)(x):=f(x)\xi(x)$ for $f\in\ell^{\infty}(X), \xi\in\ell^{2}(X)$ and $x\in X$. To simplify the notation, we will write $f\xi$ instead of $M(f)\xi$ in the sequel.


\begin{definition}\label{defn:finite ppg}
    For a bounded linear operator $T\in \B(\ell^{2}(X))$, we say that $T$ has \textit{finite propagation} if there exists $R>0$ such that for any $f_1,f_2\in\ell^{\infty}(X)$ with $d(\supp(f_{1}),\supp(f_{2}))>R$, then $f_{1}Tf_{2} = 0$. Equivalently, there exists $R>0$ such that $T_{x,y}=0$ for $x,y \in X$ with $d(x,y)>R$.
\end{definition}

\begin{definition}\label{defn:unif. Roe}
    Let $(X,d)$ be a discrete metric space with bounded geometry. The set of all finite propagation operators in $\B(\ell^{2}(X))$ forms a $\ast$-algebra, called the \emph{algebraic uniform Roe algebra} of $X$ and denoted by $\CC_{u}[X]$. The \emph{uniform Roe algebra} of $X$ is defined to be the operator norm closure of $\CC_{u}[X]$ in $\B(\ell^{2}(X))$, which is a $C^*$-algebra and denoted by $C^{*}_{u}(X)$.
\end{definition}

Now we move to the case of quasi-locality, which was introduced by Roe in \cite{roe1988:index-thm-on-open-mfds}.

\begin{definition}\label{defn:quasi-locality}
    For a bounded linear operator $T\in \B(\ell^{2}(X))$ and $\varepsilon>0$, we say that $T$ has \textit{finite $\varepsilon$-propagation} if there exists $R>0$ such that for any $f_{1},f_{2}\in\ell^{\infty}(X)$ with $d(\supp(f_{1}),\supp(f_{2}))>R$, then $\|f_{1}Tf_{2}\|<\varepsilon\|f_{1}\|_{\infty}\|f_{2}\|_{\infty}$.
    We say that $T$ is \textit{quasi-local} if $T$ has finite $\varepsilon$-propagation for any $\varepsilon>0$.
\end{definition}

Similar to the case of Roe algebras, we form the following:

\begin{definition}[\cite{li2021quasi}]\label{defn:unif. quasilocal}
    Let $(X,d)$ be a discrete metric space with bounded geometry. The set of all quasi-local operators in $\B(\ell^{2}(X))$ forms a $C^*$-algebra, called the \emph{uniform quasi-local algebra} of $X$ and denoted by $C^*_{uq}(X)$.
\end{definition}



It is clear that finite propagation operators are quasi-local, hence taking closures, we obtain that $C^*_u(X)$ is a $C^*$-subalgebra in $C^*_{uq}(X)$. For the converse, there have been a number of progresses mostly in the recent years \cite{EngelRoughIndex, LangeRabinovich, li2019quasi, SpakulaTikuisis, SpakulaZhang}. Currently the most general answer is due to \v{S}pakula and the second-named author \cite{SpakulaZhang} as follows, while the general case is still open.

\begin{proposition}[{\cite[Theorem 3.3]{SpakulaZhang}}]\label{prop:SZ}
    Let $(X,d)$ be a metric space with bounded geometry and Property A, then the uniform Roe algebra $C^*_u(X)$ coincides with the uniform quasi-local algebra $C^*_{uq}(X)$.
\end{proposition}

\subsection{Examples}\label{ssec:pre ex}

We collect several classic examples of groupoids.

\subsubsection{Discrete groups}\label{sssec:groups}
Let $\Gamma$ be a discrete group. Then $\Gamma$ is a locally compact \'{e}tale groupoid whose unit space consists of a single point $1_{\Gamma}$, \emph{i.e.}, the unit of $\Gamma$. Obviously, $\Gamma$ is amenable as a groupoid if and only if $\Gamma$ is amenable as a group.

In this case, the $\ast$-algebra $C_c(\Gamma)$ is just the group algebra $\CC[\Gamma]$ and the left regular representation at $1_\Gamma$ defined in (\ref{EQ:reduced algebra defn}) coincides with the left regular representation of the group $\Gamma$. Taking completion, the reduced groupoid $C^*$-algebra of $\Gamma$ is the same as the reduced group $C^*$-algebra $C^*_r(\Gamma)$. We also record that the Hilbert $\CC$-module $L^2(\Gamma)$ is nothing but the Hilbert space $\ell^2(\Gamma)$, and hence $\L(L^2(\Gamma)) = \B(\ell^2(\Gamma))$.

\subsubsection{Pair groupoids}\label{sssec:pair groupoids}
Let $X$ be a set. The \emph{pair groupoid} of $X$ is $X \times X$ as a set, whose unit space is $\{(x,x) \in X \times X: x\in X\}$ and identified with $X$ for simplicity. The source map is the projection onto the second coordinate and the range map is the projection onto the first coordinate. The composition is given by $(x,y) \cdot (y,z)=(x,z)$ for $x,y,z \in X$. When $X$ is a discrete Hausdorff space, then $X \times X$ is a locally compact \'{e}tale groupoid. It can be checked easily that the pair groupoid $X \times X$ is always amenable.

Each $f \in C_c(X \times X)$ can be regarded as an $X$-by-$X$ matrix with finitely many non-zero matrix entries, which induces a map
\begin{equation}\label{EQ:theta}
    \theta: C_c(X \times X) \longrightarrow \B(\ell^2(X))
\end{equation}
by setting the matrix coefficients $\theta(f)_{x,y}:=f(x,y)$ for $x,y \in X$. It is obvious that $\theta$ is a $\ast$-monomorphism, and induces a $C^*$-isomorphism
\begin{equation}\label{EQ:Theta pair}
    \Theta: C^*_r(X \times X) \cong \K(\ell^2(X)).
\end{equation}

\subsubsection{Discrete metric spaces}\label{ssec:coarse groupoids pre}
This example is the motivation for the entire work.

Let $(X,d)$ be a discrete metric space with bounded geometry. The coarse groupoid $G(X)$ on $X$ was introduced by Skandalis, Tu, and Yu in \cite{STY-article-2002a} (see also \cite[Chapter 10]{R2003-book-a}) to relate coarse geometry to the theory of
groupoids. As a topological space,
\[
    G(X):=\bigcup_{r>0}{\overline{E_r}}^{\beta (X \times X)} \subseteq \beta (X \times X),
\]
where $E_r:=\{(x,y) \in X\times X: d(x,y) \leq r\}$ and $\beta (X \times X)$ is the Stone-\v{C}ech compactification of $X \times X$. Note that $X \times X$ is the pair groupoid with source and range maps $\s(x,y)=y$ and $\r(x,y)=x$ (see Section \ref{sssec:pair groupoids}). These maps extend to maps $\beta(X \times X) \to \beta X$, and hence to maps $G(X) \to \beta X$, still denoted by $\r$ and $\s$.

Consider the map $(\r,\s): G(X) \to \beta X \times \beta X$. It was shown in \cite[Lemma 2.7]{STY-article-2002a} that $(\r,\s)$ is injective, and hence $G(X)$ can be endowed with a groupoid structure induced by the pair groupoid $\beta X \times \beta X$, called the \emph{coarse groupoid} of $X$. It was also shown in \cite[Proposition 3.2]{STY-article-2002a} that the coarse groupoid $G(X)$ is \'{e}tale, locally compact and principal. Clearly, the unit space of $G(X)$ is $\beta X$. We denote by $\partial_\beta X:=\beta X \setminus X$ the Stone-\v{C}ech boundary of $X$. It was shown in \cite[Theorem 5.3]{STY-article-2002a} that the coarse groupoid $G(X)$ is amenable if and only if $X$ has Property A.

Each $f \in C_c(G(X))$ is a continuous function supported in $\overline{E_r}$ for some $r>0$; equivalently, we can interpret $f$ as a bounded continuous function on $E_r$.  Define
\[
    \theta: C_c(G(X)) \longrightarrow \B(\ell^2(X))
\]
by setting the matrix coefficients $\theta(f)_{x,y}:=f(x,y)$ for $x,y\in X$, which extends the map defined in (\ref{EQ:theta}). Hence we take the liberty of using the same notation. It is clear that $\theta$ provides a $\ast$-isomorphism from $C_c(G(X))$ to $\mathbb{C}_u[X]$. Moreover, it was shown in \cite[Proposition 10.29]{R2003-book-a} that $\theta$ extends to a $C^*$-isomorphism
\begin{equation}\label{EQ:Theta for coarse groupoid}
    \Theta: C^*_r(G(X)) \longrightarrow C^*_u(X).
\end{equation}
This map extends the isomorphism in (\ref{EQ:Theta pair}), and hence again we abuse the notation.

\subsubsection{Group actions}\label{sssec:transformation}
Let $X$ be a locally compact space and $\Gamma$ be a countable discrete group with an action $\Gamma\curvearrowright X$ by homeomorphisms. The \textit{transformation groupoid} $X\rtimes \Gamma$ is $X \times \Gamma$ as a topological space, whose unit space is $X \times \{1_\Gamma\}$ and identified with $X$ for simplicity. The groupoid structure is given by $\s((x,\gamma))=\gamma^{-1}x$, $\r((x,\gamma))=x$ and $(x,\gamma)(\gamma^{-1}x,\gamma')=(x, \gamma\gamma')$. Then $X\rtimes \Gamma$ is a locally compact \'{e}tale groupoid. It is not hard to check that $X \rtimes \Gamma$ is amenable if and only if the action is amenable (see, \emph{e.g.}, \cite[Chapter 4]{BrownOzawa}).

Denote by $C_c(\Gamma,C_0(X))$ the $\ast$-algebra consisting of all finitely supported functions on $\Gamma$ with values in $C_0(X)$ (see \cite[Section 4.1]{BrownOzawa} for details). The \emph{reduced crossed product} $C_0(X)\rtimes_r \Gamma$ is its norm closure with respect to a regular representation of $C_c(\Gamma,C_0(X))$ (see \cite[Definition 4.1.4]{BrownOzawa}). We have the following $\ast$-isomorphism:
\[
    \theta: C_c(X \rtimes_r \Gamma) \longrightarrow C_c(\Gamma, C_c(X))
\]
given by
\[
    F \mapsto \sum_{\gamma \in \Gamma} f_\gamma \gamma, \quad \mbox{where~} f_\gamma \in C_c(X) \mbox{~is~defined~by~} f_\gamma(x)=F(x,\gamma).
\]
It is known that $\theta$ can be extended to a $C^*$-isomorphism:
\begin{equation}\label{iso reduced product}
    \Theta: C^*_r(X \rtimes \Gamma) \longrightarrow C_0(X) \rtimes_r \Gamma.
\end{equation}
Finally, we remark that $L^2(X \rtimes \Gamma) \cong \ell^2(\Gamma) \otimes C_0(X)$.

\section{Characterisations for operators in $\L(L^2(\G))$ and equivariance}\label{sec.charForOperator}


In this section, we aim to provide a detailed and practical approach to describe operators in $\L(L^2(\G))$ and discuss the notion of equivariance, which will play an important role in the sequel.
We will start with characterisations for vectors in $L^2(\G)$, and then move to operators.

Throughout this section, we assume that $\G$ is a locally compact \'{e}tale groupoid with unit space $\Gz$.

\subsection{Vectors in $L^2(\G)$}\label{ssec.fibreSpace}

Let us start by analysing vectors in the Hilbert $C_0(\Gz)$-module $L^2(\G)$ in terms of their ``slices''.

Recall from Section \ref{ssec.groupoidAlgebra} that vectors in $L^2(\G)$ can be regarded as continuous functions on $\G$ vanishing at infinity. Hence for each $x\in \Gz$, we have the restriction map
\[
    \phi_x: L^2(\G) \longrightarrow \ell^2(\G_x) \quad \text{by} \quad \phi_x(\xi)=\xi|_{\G_x}.
\]

\begin{lemma}\label{lem:surjectivity for phi_x}
    With the notation as above, the map $\phi_x$ is surjective.
\end{lemma}

\begin{proof}
    Given $\xi_x \in C_c(\G_{x})$, we assume that $\supp(\xi_x) = \{ \gamma_{1},\ldots,\gamma_{N} \}$. Since $\G$ is \'{e}tale, there exists an open subset $U$ in $\Gz$ containing $x$ and an open bisection $A_i$ containing $\gamma_i$ for each $i=1,2,\ldots,N$ such that $\s(A_i)=U$.

    Take a function $\rho\in C_c(U) \subseteq C_c(\G)$ with range in $[0,1]$ and value $1$ at the given point $x$. Define a function $\xi: \G \to \CC$ as follows: for $\gamma \in \G$ we set
    \[
        \xi(\gamma) \coloneqq \begin{cases}
            \xi_x(\gamma_{i})\rho\left(\s(\gamma)\right), & \gamma\in A_{i} \text{ for } i=1,2,\ldots,N; \\
            0,                                            & \text{otherwise}.
        \end{cases}
    \]
    It is clear that $\xi \in C_c(\G)$ and $\phi_{x}(\xi) = \xi_x$, which implies that $\phi_{x}(C_{c}(\G)) = C_{c}(\G_{x})$. Since $\phi_x$ is contractive and $C_c(\G_x)$ is dense in $\ell^2(\G_x)$, we conclude the proof.
\end{proof}

Collecting $\phi_x$ together, we obtain the following \emph{slicing map}:
\[
    \phi: L^2(\G) \longrightarrow \prod_{x\in\Gz} \ell^2(\G_{x}) \quad \text{by} \quad \phi(\xi) \coloneqq \left(\phi_{x}(\xi)\right)_{x\in\Gz},
\]
which is clearly an isometric embedding. Here $\prod_{x\in\Gz} \ell^2(\G_{x})$ denotes the Banach space consisting of uniformly norm-bounded families $(\xi_x)_{x\in \Gz}$ with norm $\|(\xi_x)_x\|:=\sup_{x\in \Gz} \|\xi_x\|$.

The following result characterises elements in the image of $\phi$, which essentially comes from \cite{D1977-book-a, DD_1963_article_a, Takahashi1979}. For completeness, here we provide a proof.

\begin{lemma}\label{lem.charOfL2G}
    With the same notation as above, let $(\xi_{x})_{x\in\Gz}\in\prod_{x\in\Gz} \ell^2(\G_{x})$ be a family of vectors.
    Then the following are equivalent:
    \begin{enumerate}
        \item\label{item1.lem.charOfL2G} $\left(\xi_{x}\right)_{x\in\Gz}\in \phi\left(L^2(\G)\right)$;
        \item\label{item2.lem.charOfL2G} the function $x\mapsto \| \xi_x - \eta|_{\G_{x}} \|$ belongs to $C_0(\Gz)$ for all $\eta\in C_{c}(\G)$;
        \item\label{item3.lem.charOfL2G} the function $x\mapsto \left\langle \xi_x , \eta|_{\G_{x}} \right\rangle$ belongs to $C_0(\Gz)$ for all $\eta\in C_{c}(\G)$, and the function $x\mapsto \|\xi_x\|$ belongs to $C_0(\Gz)$.
    \end{enumerate}
\end{lemma}

\begin{proof}
    ``\ref{item1.lem.charOfL2G} $\Rightarrow$ \ref{item2.lem.charOfL2G}'': Let $\xi\in L^2(\G)$ and we consider $(\xi_x)_x = \phi(\xi)$. Note that condition \ref{item2.lem.charOfL2G} clearly holds when $\xi \in C_c(\G)$. Since $C_c(\G)$ is dense in $L^2(\G)$ with respect to the norm $\| \xi \| = \sup_{x\in\Gz}\| \xi|_{\G_{x}} \|$, condition \ref{item2.lem.charOfL2G} holds for general $\xi$ by a standard approximating argument.

    \noindent``\ref{item2.lem.charOfL2G} $\Rightarrow$ \ref{item1.lem.charOfL2G}": Taking $\eta=0$ in condition  \ref{item2.lem.charOfL2G}, we obtain that for any $\varepsilon>0$ there exists a compact subset $K \subseteq \Gz$ such that $\sup_{x\in \Gz \setminus K}\| \xi_{x}\|<\varepsilon$. For $x \in K$, it follows from Lemma \ref{lem:surjectivity for phi_x} that there exists $\eta^x\in C_{c}(\G)$ such that $\|\xi_x - \eta^x|_{\G_x}\| < \varepsilon$. By condition \ref{item2.lem.charOfL2G}, there exists an open neighbourhood $U_x$ of $x$ such that for $y \in U_x$ we have $\|\xi_y - \eta^x|_{\G_y}\| < \varepsilon$.

    Since $K$ is compact, there exist $x_1, \ldots, x_N \in K$ such that $K\subset \bigcup_{i=1}^{N} U_{x_i}$. Consider the open cover $\mathcal{U}\coloneqq \{\Gz \setminus K, U_{x_1}, \ldots, U_{x_N}\}$ of $\Gz$ and take a partition of unity $\{\rho_0, \rho_1, \ldots, \rho_N\}$ subordinate to $\mathcal{U}$ such that $\supp(\rho_0) \subseteq \Gz \setminus K$ and $\supp(\rho_i) \subseteq U_{x_i}$ for $i=1,\ldots, N$. Define
    \[
        \eta\coloneqq \sum_{i=1}^N \rho_i \cdot \eta^{x_i}.
    \]
    It is easy to check that $\eta \in C_c(\G)$ and $\sup_{x\in \Gz}\|\xi_x - \eta|_{\G_x}\|<\varepsilon$. Hence we conclude condition \ref{item1.lem.charOfL2G}.

    Finally, notice that ``\ref{item2.lem.charOfL2G} $\Rightarrow$ \ref{item3.lem.charOfL2G}" follows directly from the polarization identity, and ``\ref{item3.lem.charOfL2G} $\Rightarrow$ \ref{item2.lem.charOfL2G}" holds trivially. Hence we conclude the proof.
\end{proof}

\begin{remark}
    For those who are familiar with the notion of Hilbert bundles, we remark that the image of $\phi$ generates a continuous field of Hilbert spaces $\mathfrak{S}$ over $\Gz$ (see \cite[Proposition 10.2.3]{D1977-book-a}), and Lemma \ref{lem.charOfL2G} is designed to show that $L^2(\G)$ coincides with elements in $\mathfrak{S}$ vanishing at infinity. Furthermore, it follows from \cite[Theorem II.13.18]{FD_1988_book_a} that there is a unique topology on $\bigsqcup_{x\in\Gz} \ell^2(\G_{x})$ which makes it into a Hilbert bundle over $\Gz$ such that $L^2(\G)$ are continuous sections.
\end{remark}

\subsection{Operators in $\L(L^2(\G))$}\label{ssec:operators in LL2G}

Having established the characterisation for vectors in $L^2(\G)$ (see Lemma \ref{lem.charOfL2G}), now we move to the case of operators in $\L(L^2(\G))$.

First note that similar to the slicing map $\phi$, we can construct a slicing map on $\L(L^2(\G))$. More precisely, for $T\in\L(L^{2}(\G))$ and $x\in \Gz$ we define an operator $\Phi_x(T)$ on $\ell^2(\G_x)$ as follows:
\[
    \Phi_x(T)(\xi|_{\G_{x}}) \coloneqq T(\xi)|_{\G_{x}} \quad \text{for} \quad \xi\in L^2(\G).
\]
It follows from Proposition \ref{prop.Proposition 1.2} and Lemma \ref{lem:surjectivity for phi_x} that $\Phi_x(T)$ is well-defined and belongs to $\B(\ell^2(\G_x))$, called the \emph{slice of $T$ at $x$}. Collecting them together, we define the \emph{slicing map}:
\begin{equation}\label{eq.PhiMap}
    \Phi\colon \L(L^2(\G)) \longrightarrow \prod_{x\in\Gz}\B(\ell^2(\G_{x})) \quad \text{by} \quad \Phi(T) \coloneqq (\Phi_x(T))_{x\in\Gz},
\end{equation}
where $\prod_{x\in\Gz}\B(\ell^2(\G_{x}))$ denotes the direct product of $C^*$-algebras. It is easy to see that $\Phi$ is a $C^*$-monomorphism, and hence an isometric embedding.

\begin{remark}\label{rem:tensor product 1}
As pointed out by the anonymous referee, the maps $\phi_{x}$ and $\Phi_{x}$ constructed above can also be defined using the language of interior tensor products of Hilbert $C^*$-modules (recalled in Section \ref{ssec.groupoidAlgebra}). More precisely, for $x\in\Gz$ we consider the evaluation map $\mathrm{ev}_{x}\colon C_{0}(\Gz)\rightarrow\mathbb{C}$ given by $\mathrm{ev}_{x}(f)=f(x)$ for $f\in C_{0}(\Gz)$. Then the interior tensor product $L^{2}(\G)\otimes_{\mathrm{ev}_{x}} \mathbb{C}$ is a Hilbert space and isomorphic to $\ell^{2}(\G_{x})$. It is not hard to verify that for $\xi\in L^{2}(\G)$ and $T \in \L(L^2(\G))$, the vector $\phi_{x}(\xi)$ is nothing but the image of $\xi\otimes 1$ under this isomorphism, and the operator $\Phi_{x}(T)$ can be identified with $T\otimes_{\mathrm{ev}_{x}} 1$. However, we stick to the original definitions for $\phi_{x}$ and $\Phi_{x}$ since the tools of interior tensor products and related theories seem only responsible to simplify the definitions rather than the proofs.
\end{remark}


The rest of this subsection is devoted to establishing characterisations for the image of $\Phi$. Our philosophy is that although operators in $\L(L^2(\G))$ are hard to describe, the family of their slices (\emph{i.e.}, the image of $\Phi$) are usually more achievable and easy to handle. This idea will play a crucial role when we study the quasi-locality for groupoids later.

Let us start with the following elementary result:

\begin{lemma}\label{lem.compactToL2}
    Let $(T_{x})_{x\in\Gz}\in\prod_{x\in\Gz}\mathscr{B}\left(\ell^2(\G_{x})\right)$ be a family of operators.
    Then $\left(T_{x}\right)_{x\in\Gz}$ belongs to the image of $\Phi$ \emph{if and only if} for any $\xi\in C_{c}(\G)$, the families $(T_{x}(\xi|_{\G_{x}}))_{x\in\Gz}$ and $(T_{x}^{*}(\xi|_{\G_{x}}))_{x\in\Gz}$ belong to $\phi\left(L^2(\G)\right)$.
\end{lemma}

\begin{proof}
    The necessity is clear, and hence we only focus on the sufficiency. First define a map $T: C_c(\G) \to L^2(\G)$ by
    \[
        T(\xi) \coloneqq \phi^{-1}\big((T_x(\xi|_{\G_x}))_x\big) \quad \text{for} \quad \xi \in C_c(\G),
    \]
    which is well-defined by assumption. It follows directly that $\|T(\xi)\| \leq \|(T_x)_x\| \cdot \|\xi\|$, which implies that $T$ can be extended to a bounded linear map $L^2(\G) \to L^2(\G)$, still denoted by $T$. Similarly, we have a bounded linear map $S: L^2(\G) \to L^2(\G)$ such that $S(\eta) = \phi^{-1}\big((T^*_x(\eta|_{\G_x}))_x\big)$ for $\eta \in C_c(\G)$. For $\xi, \eta \in C_c(\G)$, it is clear that $\langle\eta,  T(\xi) \rangle = \langle S(\eta),  \xi \rangle$. By continuity, we obtain that $S=T^*$ and hence conclude the proof.
\end{proof}

Combining Lemma \ref{lem.charOfL2G} with Lemma \ref{lem.compactToL2}, we obtain the following:

\begin{corollary}\label{lem.charOfBL2G1}
    For $\left(T_{x}\right)_{x\in\Gz}\in\prod_{x\in\Gz}\mathscr{B}\left(\ell^2(\G_{x})\right)$, the following are equivalent:
    \begin{enumerate}
        \item\label{item1.lem.charOfBL2G1} $\left(T_{x}\right)_{x\in\Gz}\in \Phi\left(\mathscr{L}\left(L^2(\G)\right)\right)$;
        \item\label{item2.lem.charOfBL2G1} the functions $x\mapsto \| T_{x}(\xi|_{\G_{x}}) - \eta|_{\G_{x}} \|$ and $x\mapsto \| T^*_{x}(\xi|_{\G_{x}}) - \eta|_{\G_{x}} \|$ belong to $C_0(\Gz)$ for all $\xi,\eta\in C_{c}(\G)$;
        \item\label{item3.lem.charOfBL2G1} the functions $x\mapsto \left\langle \eta|_{\G_{x}} , T_{x}(\xi|_{\G_{x}}) \right\rangle$, $x\mapsto \| T_{x}(\xi|_{\G_{x}}) \|$ and $x\mapsto \| T^*_{x}(\xi|_{\G_{x}}) \|$ belong to $C_0(\Gz)$ for all $\xi,\eta\in C_{c}(\G)$.
    \end{enumerate}
\end{corollary}

Corollary \ref{lem.charOfBL2G1} does provide a description for operators in $\L(L^2(\G))$. However, either condition (2) or (3) is still not easy to verify. We would like to explore a more delicate characterisation for elements in the image of $\Phi$. To fulfil the task, we would like to consult the machinery of operator fibre spaces introduced by Austin and the second-named author in \cite[Section 3]{AZ_2020_article_a}.

Consider the following disjoint union:
\begin{equation}\label{EQ:E}
    E \coloneqq \bigsqcup_{x\in\Gz} \B\left(\ell^2(\G_{x})\right).
\end{equation}
We write $\sigma_x$ for an element in $\B\left(\ell^2(\G_{x})\right) \subseteq E$ to indicate the fibre in which it lives.
Following \cite[Section 3]{AZ_2020_article_a}, we endow a topology on $E$ as follows: a net $\{ \sigma_{x_i} \}_{i\in I}$ converges to $\sigma_x$ in $E$ \emph{if and only if} $x_i\rightarrow x$ in $\Gz$ and for any $\gamma_i'\rightarrow \gamma'$, $\gamma_i''\rightarrow \gamma''$ in $\G$ with $\s(\gamma_i')=x_i=\s(\gamma_i'')$, we have
\[
    \left\langle \delta_{\gamma_i''} , \sigma_{x_i}(\delta_{\gamma_i'}) \right\rangle \rightarrow \left\langle \delta_{\gamma''} , \sigma_{x}(\delta_{\gamma'}) \right\rangle.
\]

\begin{definition}[{\cite[Definition 3.1]{AZ_2020_article_a}}]\label{defn:operator fibre space}
    For a locally compact \'{e}tale groupoid $\G$, the space $E$ defined in (\ref{EQ:E}) equipped with the above topology is called the \textit{operator fibre space} associated to $\G$.
\end{definition}

A \emph{section} of $E$ is a map $\sigma: \Gz \to E$ such that $\sigma(x) \in \B\left(\ell^2(\G_{x})\right)$ for $x\in \Gz$. Denote by $\Gamma_{b}(E)$ the set of continuous sections $\sigma$ of $E$ with $\sup_{x\in\Gz}\|\sigma(x)\|<\infty$. Equipped with pointwise operations and norm $\|\sigma\|:=\sup_{x\in\Gz}\|\sigma(x)\|$, $\Gamma_{b}(E)$ becomes a $C^*$-algebra. By definition, we have the inclusion map
\begin{equation}\label{EQ:iota}
    \iota\colon \Gamma_{b}(E) \hookrightarrow  \prod_{x\in\Gz}\mathscr{B}\left(\ell^2(\G_{x})\right), \quad \sigma \mapsto (\sigma(x))_{x\in \Gz}.
\end{equation}

The following lemma characterises continuous sections in terms of their slices, which can be deduced easily from the definition, hence we omit the proof.

\begin{lemma}\label{lem.charOfGammaB}
    For $\left(\sigma_{x}\right)_{x\in\Gz}\in \prod_{x\in\Gz}\B\left(\ell^2(\G_{x})\right)$, the following are equivalent:
    \begin{enumerate}
        \item\label{item1.lem.charOfGammaB} $\left(\sigma_{x}\right)_{x\in\Gz}\in \iota\left(\Gamma_{b}(E)\right)$;
        \item\label{item2.lem.charOfGammaB} the function $x\mapsto \left\langle \eta|_{\G_{x}} , \sigma_{x}(\xi|_{\G_{x}}) \right\rangle$ belongs to $C_0(\Gz)$ for all $\xi,\eta\in C_{c}(\G)$;
        \item\label{item3.lem.charOfGammaB} the function $\gamma\mapsto \left(\sigma_{\s(\gamma)}(\xi|_{\G_{\s(\gamma)}}) \right)(\gamma)$ is continuous on $\G$ for all $\xi \in C_{c}(\G)$;
        \item\label{item4.lem.charOfGammaB} the function $\gamma\mapsto \left(\sigma^*_{\s(\gamma)}(\xi|_{\G_{\s(\gamma)}}) \right)(\gamma)$ is continuous on $\G$ for all $\xi \in C_{c}(\G)$.
    \end{enumerate}
\end{lemma}

Combining Corollary~\refeq{lem.charOfBL2G1} with Lemma~\ref{lem.charOfGammaB}, we obtain that $\Phi\left(\L\left(L^2(\G)\right)\right)\subset \iota\left(\Gamma_{b}(E)\right)$. In other words, the homomorphism $\Phi$ factors through $\Gamma_{b}(E)$, \emph{i.e.}, there exists a $C^*$-homomorphism $\widetilde{\Phi}: \L\left(L^2(\G)\right) \to \Gamma_b(E)$ such that the following diagram commutes:
\[
    \begin{tikzcd}[column sep=small]
        \L\left(L^2(\G)\right)  \arrow[rr, "\Phi"]\arrow[rd,"\widetilde{\Phi}"']& & \prod_{x\in\Gz}\mathscr{B}\left(\ell^2(\G_{x})\right).\\
        & \Gamma_b(E) \arrow[ru, hookrightarrow, "\iota"'] &
    \end{tikzcd}
\]

\begin{remark}
    We remark that for a general element $T\in \L(L^2(\G))$, it was pointed out by Exel \cite{E_2014_incollection_a} that the map $x \mapsto \|\Phi_x(T)\|$ on $\Gz$ might not be continuous. However, Corollary~\refeq{lem.charOfBL2G1} and Lemma~\ref{lem.charOfGammaB} indicate that the map $x\mapsto \Phi_x(T)$ is indeed continuous with respect to the topology in Definition \ref{defn:operator fibre space}. That is the reason for us to consult the notion of operator fibre spaces.
\end{remark}

Thanks to Lemma \ref{lem.charOfGammaB}, it remains to determine the image of $\widetilde{\Phi}$. The following result fulfils the task:


\begin{proposition}\label{lem.charOfBL2G2}
    Let $\G$ be a locally compact \'{e}tale groupoid, and $E$ be the associated operator fibre space.
    For $\sigma\in\Gamma_{b}(E)$, the following are equivalent:
    \begin{enumerate}
        \item\label{item1.lem.charOfBL2G2} $\sigma\in \widetilde{\Phi}\left(\L\left(L^2(\G)\right)\right)$;
        \item\label{item2.lem.charOfBL2G2} for any $\varepsilon>0$ and $\xi\in C_c(\G)$, there exists a compact subset $K\subseteq \G$ such that for any $x\in\Gz$ and $A_{x}\subseteq \G_{x}$ with $A_{x}\cap K =\emptyset$, we have
        \begin{equation}\label{eq.lem.charOfBL2G2}
            \| \chi_{A_{x}} \cdot \big(\sigma(x)(\xi|_{\G_{x}})\big) \|<\varepsilon\quad \mbox{and} \quad \|\chi_{A_{x}} \cdot \big(\sigma(x)^*(\xi|_{\G_{x}})\big)\| < \varepsilon;
        \end{equation}
        \item\label{item3.lem.charOfBL2G2} for any $\varepsilon>0$ and $\xi\in C_c(\G)$, there exists a compact subset $K\subseteq \G$ such that for any $x\in\Gz$ and $A_{x}, B_{x}\subseteq \G_{x}$ with $A_{x}\cap (K B_{x}) =\emptyset$, we have
        \[
            \| \big(\chi_{A_{x}}\sigma(x)\chi_{B_{x}}\big) (\xi|_{\G_{x}}) \| < \varepsilon \text{ and }\| \big(\chi_{A_{x}}\sigma(x)^*\chi_{B_{x}}\big) (\xi|_{\G_{x}}) \| < \varepsilon.
        \]
    \end{enumerate}
\end{proposition}

\begin{proof} To save notation, we write $\sigma_x$ instead of $\sigma(x)$ where $x\in \Gz$ for the given section $\sigma \in \Gamma_b(E)$.

    \noindent``\ref{item1.lem.charOfBL2G2} $\Rightarrow $ \ref{item2.lem.charOfBL2G2}'': By assumption, there exists $T \in \L(L^2(\G))$ such that $\sigma = \widetilde{\Phi}(T)$. For any $\xi \in C_c(\G)$, we have $T(\xi) \in L^2(\G)$ and $T^*(\xi) \in L^2(\G)$. Hence for any $\varepsilon>0$, there exist $\eta, \zeta \in C_c(\G)$ such that $\|T(\xi) - \eta\|<\varepsilon$ and $\|T^*(\xi)-\zeta\|<\varepsilon$.

    Take $K\coloneqq \supp(\eta) \cup \supp(\zeta)$. For any $x\in \Gz$ and $A_{x}\subseteq \G_{x}$ with $A_{x}\cap K =\emptyset$, we obtain:
    \[
        \| \chi_{A_{x}} \cdot \big(\sigma_x(\xi|_{\G_{x}})\big) \| = \| \chi_{A_{x}} \cdot \big(T(\xi)|_{\G_x}\big) \| \leq \| \chi_{A_{x}} \cdot \big(\eta|_{\G_x}\big) \| + \varepsilon = \varepsilon.
    \]
    Similarly, we obtain $\|\chi_{A_{x}} \cdot \big(\sigma_x^*(\xi|_{\G_{x}})\big)\| < \varepsilon$.

    \noindent ``\ref{item2.lem.charOfBL2G2} $\Rightarrow $ \ref{item1.lem.charOfBL2G2}'': By assumption, for any $\varepsilon>0$ and $\xi\in C_c(\G)$, there exists a compact subset $K \subseteq \G$ satisfying condition \ref{item2.lem.charOfBL2G2}. Take a function $\rho\in C_{c}(\G)$ such that $\rho|_K \equiv 1$.

    Consider $\zeta: \G \to \CC$ defined by $\zeta(\gamma) \coloneqq \rho(\gamma) \cdot \left(\sigma_{\s(\gamma)}(\xi|_{\G_{\s(\gamma)}}) \right)(\gamma)$ for $\gamma \in \G$. It follows from Lemma \ref{lem.charOfGammaB} that $\zeta \in C_c(\G)$. Moreover, for $x\in \Gz$ we have:
    \[
        \| \sigma_x(\xi|_{\G_x})  - \phi_x(\zeta) \|  = \|(1-\rho)|_{\G_x} \cdot \big(\sigma_x(\xi|_{\G_x})\big)\| = \|(1-\rho)|_{\G_x} \cdot \chi_{\G_x \setminus K} \cdot \big(\sigma_x(\xi|_{\G_x})\big)\|,
    \]
    which is less than $\varepsilon$ by the assumption (\ref{eq.lem.charOfBL2G2}). Letting $\varepsilon \to 0$, we obtain that $\big(\sigma_x(\xi|_{\G_x})\big)_{x\in \Gz}$ belongs to $\phi(L^2(\G))$. Similarly, $\big(\sigma_x^*(\xi|_{\G_x})\big)_{x\in \Gz}$ belongs to $\phi(L^2(\G))$. Therefore, we conclude condition \ref{item1.lem.charOfBL2G2} thanks to Lemma \ref{lem.compactToL2}.

    \noindent ``\ref{item2.lem.charOfBL2G2} $\Rightarrow $ \ref{item3.lem.charOfBL2G2}'': As remarked in the last paragraph of Section \ref{ssec.groupoid}, it suffices to prove condition \ref{item3.lem.charOfBL2G2} for $\xi \in C_c(\G)$ whose support is contained in a compact bisection $K_\xi$. Given $\varepsilon>0$, take a compact subset $K \subseteq \G$ satisfying condition \ref{item2.lem.charOfBL2G2}. We set $\widetilde{K} \coloneqq K K_\xi^{-1}$. Since $\supp(\xi)$ is a bisection, it follows that for any $x\in \Gz$ and $B_x \subseteq \G_x$ we have:
    \[
        \chi_{B_x} \xi|_{\G_x} =
        \begin{cases}
            \xi|_{\G_x}, & \text{ if } K_{\xi}\cap B_{x}\neq\emptyset; \\
            0,           & \text{ if } K_{\xi}\cap B_{x}=\emptyset.
        \end{cases}
    \]

    Now for any $x\in\Gz$ and $A_x, B_x \subseteq \G_x$ with $A_x \cap (\widetilde{K} B_x) = \emptyset$, then $A_x \cap (K K_\xi^{-1} B_x) = \emptyset$. If $K_{\xi}\cap B_{x}=\emptyset$, then:
    \[
        \| \big(\chi_{A_{x}}\sigma_x\chi_{B_{x}}\big) (\xi|_{\G_{x}}) \| = \|\big(\chi_{A_{x}}\sigma_x\big)(\chi_{B_x} \xi|_{\G_x})\|=0
    \]
    and similarly, $\| \big(\chi_{A_{x}}\sigma_x^*\chi_{B_{x}}\big) (\xi|_{\G_{x}}) \|=0$. If $K_{\xi}\cap B_{x}\neq\emptyset$, then
    \[
        K K_\xi^{-1} B_x \supseteq K \cdot \{x\} = K \cap \G_x.
    \]
    Hence $A_x \cap K= A_x \cap (K \cap \G_x) \subseteq A_x \cap (K K_\xi^{-1} B_x) = \emptyset$. Applying condition \ref{item2.lem.charOfBL2G2}, we obtain:
    \[
        \| \big(\chi_{A_{x}}\sigma_x\chi_{B_{x}}\big) (\xi|_{\G_{x}}) \| = \|\big(\chi_{A_{x}}\sigma_x\big)(\chi_{B_x} \xi|_{\G_x})\|=\|\chi_{A_x} \cdot \big(\sigma_x (\xi|_{\G_x})\big)\| < \varepsilon.
    \]
    Similarly, we have $\| \big(\chi_{A_{x}}\sigma_x^*\chi_{B_{x}}\big) (\xi|_{\G_{x}}) \| < \varepsilon$. Therefore, we conclude condition \ref{item3.lem.charOfBL2G2}.

    \noindent ``\ref{item3.lem.charOfBL2G2} $\Rightarrow $ \ref{item2.lem.charOfBL2G2}'': Given $\varepsilon>0$ and $\xi\in C_c(\G)$, let $K_{\xi} \coloneqq \supp(\xi)$. By condition \ref{item3.lem.charOfBL2G2}, there exists a compact subset $\widetilde{K}\subseteq \G$ satisfying the requirement therein. We set $K \coloneqq \widetilde{K} K_{\xi}$ and take $B_{x} = \supp(\xi|_{\G_{x}})$. For any $x\in\Gz$ and $A_x\subseteq \G_x$ with $A_x\cap K=\emptyset$, we have
    \[
        A_x \cap (\widetilde{K} B_x) = A_x \cap \big(\widetilde{K} \cdot  \supp(\xi|_{\G_{x}})\big) \subseteq A_x \cap (\widetilde{K} K_\xi) = A_x \cap K = \emptyset.
    \]
    Hence we obtain:
    \[
        \| \chi_{A_{x}} \cdot \big(\sigma_{x}(\xi|_{\G_{x}})\big)  \| = \| \left(\chi_{A_{x}}\sigma_{x}\chi_{B_{x}}\right)(\xi|_{\G_{x}})\| < \varepsilon
    \]
    and similarly, $\| \chi_{A_{x}} \cdot \big(\sigma_x^*(\xi|_{\G_{x}})\big)  \| < \varepsilon$. Therefore, we conclude the proof.
\end{proof}

\begin{remark}
    Careful readers might already notice from Corollary \refeq{lem.charOfBL2G1}(3) and Lemma \ref{lem.charOfGammaB}(2) that in order to show that a continuous section $\sigma \in \Gamma_b(E)$ belongs to the image of $\widetilde{\Phi}$, it suffices to show that the functions $x \mapsto \|\sigma(x)(\xi|_{\G_x})\|$ and $x \mapsto \|\sigma(x)^*(\xi|_{\G_x})\|$ belong to $C_0(\Gz)$ for any $\xi \in C_c(\G)$. In fact, one can verify these conditions directly using condition \ref{item2.lem.charOfBL2G2} in Proposition \ref{lem.charOfBL2G2}, which suggests an alternative proof for ``\ref{item2.lem.charOfBL2G2} $\Rightarrow$ \ref{item1.lem.charOfBL2G2}''.
\end{remark}

Comparing with the notion of quasi-locality recalled in Definition \ref{defn:quasi-locality}, we note that condition \ref{item3.lem.charOfBL2G2} in Proposition \ref{lem.charOfBL2G2} can be regarded as a vector-wise uniform version of quasi-locality. Hence we introduce the following:

\begin{definition}\label{defn:vectorwise quasiloc}
    Let $\G$ be a locally compact \'{e}tale groupoid with unit space $\Gz$. A family $(T_{x})_{x\in\Gz} \in \prod_{x\in\Gz}\B\left(\ell^2(\G_{x})\right)$ is said to be \emph{vector-wise uniformly quasi-local} if for any $\varepsilon>0$ and $\xi\in C_c(\G)$, there exists a compact subset $K\subseteq \G$ such that for any $x\in\Gz$ and $A_{x}, B_{x}\subseteq \G_{x}$ with $A_{x}\cap (K B_{x}) =\emptyset$, we have
    \[
        \| \big(\chi_{A_{x}}T_x\chi_{B_{x}}\big) (\xi|_{\G_{x}}) \| < \varepsilon.
    \]
\end{definition}

Combining Lemma \ref{lem.charOfGammaB} with Proposition \ref{lem.charOfBL2G2}, we finally reach the following:

\begin{theorem}\label{cor:char for LL2G}
    Let $\G$ be a locally compact \'{e}tale groupoid with unit space $\Gz$. For $(T_{x})_{x\in\Gz} \in \prod_{x\in\Gz}\B\left(\ell^2(\G_{x})\right)$, the following are equivalent:
    \begin{enumerate}
        \item $(T_{x})_{x\in\Gz}$ belongs to $\Phi\big(\L(L^2(\G))\big)$;
        \item the map $x\mapsto T_x$ is a continuous section of $E$, and $(T_{x})_{x\in\Gz}, (T^*_{x})_{x\in\Gz}$ are vector-wise uniformly quasi-local.
    \end{enumerate}
\end{theorem}

\subsection{The case of dense subsets}\label{ssec:dense subset}

A lot of examples including the coarse groupoids (see Section \ref{ssec:coarse groupoids pre}) come with natural dense open subsets in their unit spaces. In these cases, we now show that the characterisations established in Section \ref{ssec:operators in LL2G} can be further simplified.

Let $\G$ be a locally compact  \'{e}tale groupoid with unit space $\Gz$. Assume that $X$ is a dense subset of $\Gz$, and we set $\partial X \coloneqq \Gz\setminus X$. Define
\[
    E_X:=\bigsqcup_{x\in X}\mathscr{B}(\ell^2(\G_x)) \subseteq E
\]
with the induced topology from $E$, and denote by $\Gamma_{b}(E_X)$ the $C^*$-algebra consisting of all continuous norm-bounded sections of $E_X$. Consider the restriction map
\[
    \res: \Gamma_{b}(E) \longrightarrow \Gamma_b(E_X), \quad \sigma \mapsto \sigma|_X.
\]
Since $X$ is dense and $\G$ is \'{e}tale, it is easy to verify that the map $\res$ is injective. To determine its image, we introduce the following:

\begin{definition}\label{defn:extendable section}
    A section $\sigma \in \Gamma_{b}(E_X)$ is called \emph{extendable} if for any $\omega\in \partial X$ and $\gamma', \gamma'' \in \G_\omega$, there exists a constant $c_{\gamma', \gamma''} \in \CC$ such that for any $x_i\in X$ with $x_i \to \omega$ and any $\gamma'_i, \gamma''_i \in \G_{x_i}$ with $\gamma'_i\rightarrow \gamma'$ and $\gamma''_i\rightarrow \gamma''$, then $\langle \delta_{\gamma''_i}, \sigma(x_i)(\delta_{\gamma'_i}) \rangle \to c_{\gamma', \gamma''}$.
\end{definition}

The following lemma explains the terminology:

\begin{lemma}\label{densesubset-extension}
    The map $\res: \Gamma_{b}(E) \rightarrow \Gamma_b(E_X)$ is a $C^*$-monomorphism with image consisting of all extendable sections in $\Gamma_b(E_X)$.
\end{lemma}

\begin{proof}
    It suffices to show that an extendable section $\sigma \in \Gamma_b(E_X)$ can be extended to a section $\tilde{\sigma} \in \Gamma_b(E)$. For $\omega \in \partial X$, define an operator $\sigma_\omega: \ell^2(\G_\omega)\rightarrow \ell^2(\G_\omega)$ by
    \begin{equation*}
        \langle \delta_{\gamma''}, \sigma_\omega(\delta_{\gamma'}) \rangle \coloneqq \lim_i \langle \delta_{\gamma''_i}, \sigma({x_i})(\delta_{\gamma'_i}) \rangle \quad \mbox{for} \quad \gamma', \gamma'' \in \G_\omega,
    \end{equation*}
    where $x_i$ is a net in $X$ converging to $\omega$, and $\gamma'_i, \gamma''_i \in \G_{x_i}$ with $\gamma'_i\rightarrow \gamma'$ and $\gamma''_i\rightarrow \gamma''$. This is well-defined since $\G$ is \'{e}tale and $\sigma$ is extenable. To show that $\sigma_\omega$ is a bounded operator, we note that
    \[
        \|\sigma_\omega\|=\sup\big\{|\langle \eta_\omega, \sigma_\omega(\xi_\omega) \rangle|: \xi_\omega, \eta_\omega \in C_c(\G_\omega) \mbox{~with~} \|\xi_\omega\| = \|\eta_\omega\|=1\big\}.
    \]
    Hence using the \'{e}taleness and $\sup_{x\in X} \|\sigma(x)\|< \infty$, we obtain that $\sigma_\omega$ is bounded (details are left to readers). Finally, we define a map $\tilde{\sigma}: \Gz \to E$ by
    \[
        \tilde{\sigma}(z) \coloneqq
        \begin{cases}
            \sigma(z), & \text{ if } z\in X;          \\
            \sigma_z,  & \text{ if } z\in \partial X,
        \end{cases}
    \]
    which clearly belongs to $\Gamma_b(E)$ and extends $\sigma$. Hence we conclude the proof.
\end{proof}

The following result shows that in some special cases, sections in $\Gamma_b(E_X)$ are always extendable. First recall that for a locally compact Hausdorff space $X$, its \emph{Stone-\v{C}ech compactification} is a compact space $\beta X$ together with a continuous map $i: X\rightarrow \beta X$ such that $i: X\rightarrow i(X)$ is a homeomorphsim with dense image in $\beta X$ and satisfies the universal property: for any $f\in C_b(X)$ there exists a continuous map $f^{\beta}:\beta X\rightarrow \CC$ such that $f^{\beta}\circ i=f$. Note that in this case, $i(X)$ is open in $\beta X$.

\begin{lemma}\label{Stone-Cech-extension}
    Let $\G$ be a locally compact \'{e}tale groupoid with unit space $\Gz$, and $X$ be an open dense subset in $\Gz$ such that $\Gz=\beta X$ is the Stone-\v{C}ech compactification of $X$ (with the inclusion map). Then any section $\sigma \in \Gamma_b(E_X)$ is extendable.
\end{lemma}

\begin{proof}
    Fix a section $\sigma \in \Gamma_b(E_X)$, $\omega\in \partial X$ and $\gamma', \gamma''\in \G_{\omega}$. Using \'{e}taleness of $\G$, we can find two open neighborhoods $U_{\gamma'}$ and $U_{\gamma''}$ of $\gamma'$ and $\gamma''$, respectively, such that $\s|_{U_{\gamma'}}$ and $\s|_{U_{\gamma''}}$ are homeomorphisms. Let $U_{\omega}\coloneqq \s(U_{\gamma'})\cap \s(U_{\gamma''})$, and we take $\rho \in C_b(\mathcal{G}^{(0)})$ such that $\rho(\omega)=1$ and $\rho|_{\Gz\setminus U_{\omega}}=0$. Consider the function $f_{\gamma',\gamma''}: X \rightarrow \mathbb{C}$ defined by
    \[
        f_{\gamma',\gamma''}(x) \coloneqq
        \begin{cases}
            \rho(x) \langle \delta_{\gamma''_x}, \sigma(x)(\delta_{\gamma'_x}) \rangle, & \text{ if } x\in X\cap U_{\omega}; \\
            0,                                                                          & \text{otherwise},
        \end{cases}
    \]
    where $\gamma'_x\in U_{\gamma'}\cap \mathcal{G}_x$ and $\gamma''_x\in U_{\gamma''}\cap \mathcal{G}_x$ are uniquely determined. It is clear that $f_{\gamma',\gamma''}\in C_b(X)$. Thanks to the universal property of $\beta X=\Gz$, $f_{\gamma',\gamma''}$ can be uniquely extended to a function $f^{\beta}_{\gamma',\gamma''} \in C(\Gz)$. Setting $c_{\gamma',\gamma''} \coloneqq f^{\beta}_{\gamma',\gamma''}(\omega)$, we conclude that $\sigma$ is extendable.
\end{proof}

Combining Lemma \ref{densesubset-extension} and \ref{Stone-Cech-extension}, we obtain:

\begin{corollary}\label{cor:beta X case}
    Let $\G$ be a locally compact \'{e}tale groupoid with unit space $\Gz$, and $X$ be an open dense subset in $\Gz$ such that $\Gz=\beta X$ is the Stone-\v{C}ech compactification of $X$ (with the inclusion map). Then the map $\res: \Gamma_{b}(E) \rightarrow \Gamma_b(E_X)$ is a $C^*$-isomorphism.
\end{corollary}

Now we are in the position to simplify Proposition~\ref{lem.charOfBL2G2}:

\begin{proposition}\label{lem.charOfBL2G2-densesubset}
    Let $\G$ be a locally compact \'{e}tale groupoid with unit space $\Gz$, and $X$ be a dense subset of $\Gz$.
    For an extendable section $\sigma\in\Gamma_{b}(E_X)$, the following are equivalent:
    \begin{enumerate}
        \item\label{item1.lem.charOfBL2G2-densesubset} $\sigma\in \res \circ\widetilde{\Phi}\left(\mathscr{L}\left(L^2(\G)\right)\right)$;
        \item\label{item2.lem.charOfBL2G2-densesubset} for any $\varepsilon>0$ and $\xi\in C_c(\G)$, there exists a compact subset $K\subseteq \G$ such that for any $x\in X$ and $A_{x}\subseteq \G_{x}$ with $A_{x}\cap K =\emptyset$, we have
        \[
            \| \chi_{A_{x}} \cdot \big(\sigma(x)(\xi|_{\G_{x}})\big) \|<\varepsilon\quad \mbox{and} \quad \|\chi_{A_{x}} \cdot \big(\sigma(x)^*(\xi|_{\G_{x}})\big)\| < \varepsilon;
        \]
        \item\label{item3.lem.charOfBL2G2-densesubset} for any $\varepsilon>0$ and $\xi\in C_c(\G)$, there exists a compact subset $K\subseteq \G$ such that for any $x\in X$ and $A_{x}, B_{x}\subseteq \G_{x}$ with $A_{x}\cap (K B_{x}) =\emptyset$, we have
        \[
            \| \big(\chi_{A_{x}}\sigma(x)\chi_{B_{x}}\big) (\xi|_{\G_{x}}) \| < \varepsilon \text{ and }\| \big(\chi_{A_{x}}\sigma(x)^*\chi_{B_{x}}\big) (\xi|_{\G_{x}}) \| < \varepsilon.
        \]
    \end{enumerate}
\end{proposition}

\begin{proof}
    From Lemma \ref{densesubset-extension}, there exists $\tilde{\sigma}\in \Gamma_b(E)$ such that $\res(\tilde{\sigma})=\sigma$. Hence it suffices to show that condition (2) and (3) for $\sigma$ are equivalent to those in Proposition \ref{lem.charOfBL2G2-densesubset} for $\tilde{\sigma}$, respectively. Again to save the notation, we write $\tilde{\sigma}_x$ instead of $\tilde{\sigma}(x)$ for $x\in \Gz$.

    We only present the case for condition (2), while the other is similar. Assume that condition (2) holds for $\sigma$, we would like to verify Proposition \ref{lem.charOfBL2G2-densesubset}(2) for $\tilde{\sigma}$. Given $\varepsilon>0$ and $\xi\in C_c(\G)$, let $K\subseteq \G$ be a compact subset satisfying condition (2). For any $y\in \Gz$ and any finite subset $A_{y}\subset\G_{y}$ with $A_{y}\cap K=\emptyset$, there exist $x\in X$ and a finite subset $A_{x}\subset\G_{x}$ such that $A_{x}\cap K=\emptyset$ and
    \[
        |\langle \tilde{\sigma}_{y}(\xi|_{\G_{y}}), \eta|_{A_y} \rangle - \langle \sigma_{x}(\xi|_{\G_{x}}), \eta|_{A_x} \rangle|<\varepsilon
    \]
    for any $\eta\in  C_c(\G)$ with $\|\eta\|\leq 1$, since $\G$ is \'{e}tale and $\tilde{\sigma}$ is the extension of $\sigma$. Note that $\| \chi_{A_{x}} \big(\sigma_x(\xi|_{\G_{x}})\big) \|<\varepsilon$ and
    \[
        \| \chi_{A_{y}} \left(\tilde{\sigma}_{y}(\xi|_{\G_{y}})\right) \|=\sup_{\|\eta|_{\G_{y}}\|=1} |\langle \tilde{\sigma}_{y}(\xi|_{\G_{y}}), \eta|_{A_{y}} \rangle|.
    \]
    Hence we obtain $\| \chi_{A_{y}} \left(\tilde{\sigma}_{y}(\xi|_{\G_{y}})\right) \| < 2\varepsilon$, and similarly $\| \chi_{A_{y}} \left(\tilde{\sigma}_{y}^*(\xi|_{\G_{y}})\right) \| < 2\varepsilon$. Finally using a standard approximating argument, the above estimates hold for arbitrary $A_{y}\subset\G_{y}$ with $A_{y}\cap K=\emptyset$, and therefore we conclude the proof.
\end{proof}

Analogous to Definition \ref{defn:vectorwise quasiloc}, for a dense subset $X$ of $\Gz$ we say that a family $(T_{x})_{x\in X} \in \prod_{x\in X}\B\left(\ell^2(\G_{x})\right)$ is \emph{vector-wise uniformly quasi-local} if for any $\varepsilon>0$ and $\xi\in C_c(\G)$, there exists a compact subset $K\subseteq \G$ such that for any $x\in X$ and $A_{x}, B_{x}\subseteq \G_{x}$ with $A_{x}\cap (K B_{x}) =\emptyset$, we have $\| \big(\chi_{A_{x}}T_x\chi_{B_{x}}\big) (\xi|_{\G_{x}}) \| < \varepsilon$.

Finally, we conclude the following for the case of dense subsets:

\begin{corollary}\label{cor:char2 for LL2G}
    Let $\G$ be a locally compact \'{e}tale groupoid with unit space $\Gz$, and $X$ be a dense subset of $\Gz$. For $(T_{x})_{x\in X} \in \prod_{x\in X}\B\left(\ell^2(\G_{x})\right)$, the following are equivalent:
    \begin{enumerate}
        \item $(T_{x})_{x\in X}$ belongs to $\res\circ \Phi\big(\L(L^2(\G))\big)$;
        \item the map $x\mapsto T_x$ is a continuous extendable section of $E_X$, and $(T_{x})_{x\in X}, (T^*_{x})_{x\in X}$ are vector-wise uniformly quasi-local.
    \end{enumerate}
    Moreover, we can omit the requirement of extendableness in condition (2) when $\Gz$ is the Stone-\v{C}ech compactification of $X$ (with the inclusion map).
\end{corollary}

\subsection{Equivariant operators}\label{ssec.equivariant}
In this subsection, we study an extra condition for operators in $\L(L^2(\G))$ called $\G$-equivariance. It turns out that all operators in the reduced groupoid $C^*$-algebra $C^*_r(\G)$ satisfy this condition, which suggests us to restrict ourselves to the equivariant case when studying quasi-locality in Section \ref{sec.quasi}. This class of operators possess a huge advantage since they can always be represented by certain continuous functions (see Section \ref{ssec:realisation for equiv}), which will play a key role in the proof of our main theorem.

Recall that the left regular representation $\Lambda: C_c(\G) \to \L(L^2(\G))$ is defined in Section \ref{ssec.groupoidAlgebra}. As for groups, we can also consider the right regular representation as follows. Given a function $g \in C_c(\G)$, define the map $\rho(g)\colon C_{c}(\G) \rightarrow C_{c}(\G)$ by
\[
    \big(\rho(g)\xi\big)(\gamma):=(\xi \ast g)(\gamma) = \sum_{\beta\in\G_{\s(\gamma)}}g(\beta)\xi(\gamma\beta^{-1}) =\sum_{\alpha\in \G^{\r(\gamma)}}g(\alpha^{-1}\gamma)\xi(\alpha) \text{  for  }  \xi\in C_{c}(\G).
\]
Note that in general, $\rho(g)$ might \emph{not} be a $C_0(\Gz)$-module homomorphism. However, we have the following:

\begin{lemma}
    For $g\in C_{c}(\G)$, the operator $\rho(g)$ can be extended to a bounded linear operator on $L^{2}(\G)$ (regarded as a Banach space), still denoted by $\rho(g)$.
\end{lemma}

\begin{proof}
    For any $\xi \in C_c(\G)$, we have:
    \begin{equation*}
        \begin{split}
            \|\rho(g)\xi\|^2
            &=  \sup_{x\in \Gz}\sum_{\gamma\in \G_x}|\sum_{\beta\in\G_x}g(\beta)\xi(\gamma\beta^{-1})|^2      \\
            &\leq \sup_{x\in\Gz} \sum_{\gamma\in \G_x}\big(\sum_{\beta\in \G_x}|g(\beta)|\big)\cdot\big(\sum_{\beta\in\G_x}|g(\beta)|\cdot |\xi(\gamma\beta^{-1})|^2\big)    \\
            &= \sup_{x\in\Gz}\big(\sum_{\beta\in \G_x}|g(\beta)|\big)\cdot \sum_{\beta\in\G_x}\big(|g(\beta)|\cdot \sum_{\gamma\in \G_x}|\xi(\gamma\beta^{-1})|^2 \big)\\
            &\leq \big(\sup_{x\in \Gz}\sum_{\beta\in \G_x}|g(\beta)|\big)^2\cdot\|\xi\|^2.
        \end{split}
    \end{equation*}
    Note that $g\in C_c(\G)$, hence the number $\sup_{x\in \Gz}\sum_{\beta\in \G_x}|g(\beta)|$ is finite. Therefore, $\rho(g)$ can be extended to a bounded operator on $L^2(\G)$.
\end{proof}

As a result, the convolution operator can also be extended to an operator (with the same notation) $\ast: L^2(\G) \times C_c(\G) \to L^2(\G)$ by the same formula (\ref{EQ:convolution}), and $\rho(g)\xi=\xi \ast g$ for $g\in C_c(\G)$ and $\xi \in L^2(\G)$.


\begin{definition}\label{def.GEquivariant}
    Let $\G$ be a locally compact \'{e}tale groupoid. We say that $T\in\L(L^2(\G))$ is \textit{$\G$-equivariant}\footnote{Warning: Note that our definition of $\G$-equivariant operators is not the same as the standard one considered in the literature, \emph{e.g.}, \cite[Definition 4.6]{MR1686846}.} if $T\rho(g)=\rho(g)T$ (as bounded operators on $L^{2}(\G)$) for any $g\in C_{c}(\G)$. Denote by ${\L\left(L^{2}(\G)\right)}^{\G}$ the set of all $\G$-equivariant operators in $\L\left(L^2(\G)\right)$, which clearly forms a $C^*$-subalgebra in $\L(L^2(\G))$.
\end{definition}


The following result shows that operators in $C^*_r(\G)$ are $\G$-equivariant. The proof is straightforward, hence omitted.

\begin{lemma}\label{lem.GEquivariantGroupoidAlgebra}
    Given $f\in C_c(\G)$, the operator $\Lambda({f})$ is $\G$-equivariant, where $\Lambda$ is the left regular representation of $\G$.
    Thus, we have $C^{*}_r(\G)\subseteq {\L\left(L^{2}(\G)\right)}^{\G}$.
\end{lemma}

Using the slicing map $\Phi$ defined in (\ref{eq.PhiMap}), we can characterise $\G$-equivariance in terms of their slices. To start, for $\gamma\in\G$ we consider the operator
\[
    V_{\gamma}:\ell^2(\G_{\s(\gamma)}) \rightarrow\ell^2(\G_{\r(\gamma)}) \quad \text{by} \quad V_{\gamma}(\xi)(\gamma') = \xi(\gamma'\gamma)
\]
for $\xi\in\ell^{2}(\G_{\s(\gamma)})$. We have the following:

\begin{lemma}\label{equivariant-fiber}
    Let $\G$ be a locally compact \'{e}tale groupoid. An operator $T\in\L(L^2(\G))$ is $\G$-equivariant \emph{if and only if} $V_{\gamma}\Phi_{\s(\gamma)}(T)=\Phi_{\r(\gamma)}(T) V_{\gamma}$ for all $\gamma\in \G$.
\end{lemma}

\begin{proof}
    For the necessity, we fix $\gamma \in \G$ and denote $x=\s(\gamma), y=\r(\gamma)$. By \'{e}taleness, choose a pre-compact neighbourhood $U$ of $\gamma^{-1}$ such that $\s|_U: U \to \s(U)$ is a homeomorphism. Choose $g\in C_c(\G)$ such that $g(\gamma^{-1})=1$ and $g(\alpha)=0$ for $\alpha\in \G\setminus U$. For any $\eta\in L^2(\G)$ and $\gamma' \in \G_y$, note that
    \[
        (\rho(g)\eta)(\gamma')=\sum_{\alpha\in \G^{\r(\gamma')}}g(\alpha^{-1}\gamma')\eta(\alpha) = \eta(\gamma'\gamma) = V_{\gamma}(\eta|_{\G_x})(\gamma'),
    \]
    which implies that $(\rho(g)\eta)|_{G_y} = V_{\gamma}(\eta|_{\G_x})$. Hence for any $\xi \in L^2(\G)$, we obtain
    \begin{align*}
        \big(V_{\gamma}\Phi_x(T)\big)(\xi|_{\G_x}) & =V_{\gamma} (T(\xi)|_{\G_x}) = \big( \rho(g)T\xi \big)|_{\G_y} = \big( T \rho(g)\xi \big)|_{\G_y}= \Phi_y(T)\big((\rho(g)\xi)|_{\G_y}\big) \\
                                                   & = \big(\Phi_y(T) V_\gamma \big)(\xi|_{\G_x}).
    \end{align*}

    For the sufficiency, we need to show that $\rho(g)T=T\rho(g)$ for any $g\in C_c(\G)$. Without loss of generality, we can assume that $\supp(g)$ is a bisection (see Section \ref{ssec.groupoid}). Given $\xi\in L^2(\G)$ and $\gamma \in \G$ with $\s(\gamma)=x$ and $\r(\gamma)=y$, the intersection $\G_x \cap \supp(g)$ contains at most one point. If $\G_x \cap \supp(g)=\{\hat{\gamma}\}$, then
    \begin{align*}
        \big( \rho(g)T\xi \big)(\gamma) & = \sum_{\beta \in \G_x} g(\beta) \cdot (T\xi)(\gamma \beta^{-1})=g(\hat{\gamma})\cdot (T\xi)(\gamma \hat{\gamma}^{-1})                                                                                                                                 \\
                                        & = g(\hat{\gamma}) \cdot \Big(\Phi_{\r(\hat{\gamma})}(T)(\xi|_{\G_{\r(\hat{\gamma})}})\Big)(\gamma \hat{\gamma}^{-1}) = g(\hat{\gamma}) \cdot \Big(\big(V_{\hat{\gamma}^{-1}}\Phi_{\r(\hat{\gamma})}(T)\big)(\xi|_{\G_{\r(\hat{\gamma})}})\Big)(\gamma) \\
                                        & = g(\hat{\gamma}) \cdot \Big(\big(\Phi_x(T) V_{\hat{\gamma}^{-1}}\big)(\xi|_{\G_{\r(\hat{\gamma})}})\Big)(\gamma),
    \end{align*}
    where we use $V_{\hat{\gamma}^{-1}}\Phi_{\r(\hat{\gamma})}(T) = \Phi_x(T) V_{\hat{\gamma}^{-1}}$ in the last equality. On the other hand, note that for any $\gamma' \in \G_x$ we have
    \[
        ( \rho(g)\xi )(\gamma') = \sum_{\beta \in \G_x} g(\beta) \cdot \xi(\gamma' \beta^{-1}) =g(\hat{\gamma}) \cdot \xi(\gamma' \hat{\gamma}^{-1}) = g(\hat{\gamma}) \cdot V_{\hat{\gamma}^{-1}}(\xi|_{\G_{\r(\hat{\gamma}))}})(\gamma').
    \]
    Hence we obtain
    \[
        \big( T\rho(g)\xi \big)(\gamma) = \Phi_x(T)\big( (\rho(g)\xi)|_{\G_x} \big)(\gamma) = g(\hat{\gamma}) \cdot \Big(\big(\Phi_x(T) V_{\hat{\gamma}^{-1}}\big)(\xi|_{\G_{\r(\hat{\gamma})}})\Big)(\gamma),
    \]
    which shows that $\rho(g)T = T\rho(g)$. The case of $\G_x \cap \supp(g)=\emptyset$ is similar but much easier, hence omitted.
\end{proof}


%

In the case of dense subsets, Lemma \ref{equivariant-fiber} can be reduced to the following. The proof is straightforward, hence omitted.

\begin{corollary}\label{equivariant-dense}
    Let $\G$ be a locally compact \'{e}tale groupoid and $X$ be a dense invariant subset of $\Gz$. Then an operator $T \in \L(L^2(\G))$ is $\G$-equivariant \emph{if and only if} $V_{\gamma}T_{\s(\gamma)}=T_{\r(\gamma)} V_{\gamma}$ for any $\gamma\in \s^{-1}(X)=\r^{-1}(X)$.
\end{corollary}

Now we apply the tools developed above to calculate the $C^*$-algebra $\L(L^2(\G))^{\G}$ for the coarse groupoid:

\begin{example}\label{eg:coarse groupoid for LL2G}
    Following the notation in Section \ref{ssec:coarse groupoids pre}, let $(X,d)$ be a discrete metric space with bounded geometry and $\G=G(X)$ be the associated coarse groupoid. The unit space of $G(X)$ is the Stone-\v{C}ech compactification $\beta X$ of $X$, with an open invariant dense subset $X$. Note that each $\G_x$ is bijective to $X$ via the map $(y,x) \mapsto y$, and hence Corollary \ref{cor:beta X case} implies that the map
    \[
        \res\circ \Phi: \L(L^2(\G))^{\G} \longrightarrow \prod_{x\in X}\B\left(\ell^2(\G_{x})\right) \cong \prod_{x\in X}\B\left(\ell^2(X)\right)
    \]
    is a $C^\ast$-monomorphism. Moreover, it follows from Corollary \ref{cor:char2 for LL2G} and \ref{equivariant-dense} that $(T_x)_{x\in X}$ belongs to the image of $\res\circ \Phi$ if and only if there exists $T \in \B(\ell^2(X))$ such that the constant family $(T)_{x\in X}$ is vector-wise uniformly quasi-local.

    To provide a more concrete picture, we claim that the family $(T)_{x\in X}$ is vector-wise uniformly quasi-local \emph{if and only if} the following holds:
    \begin{equation}\label{EQ:char for LL2G coarse groupoid}
        \lim_{S \to \infty} \sup_{x\in X} \sum_{y\notin B(x,S)} |T_{y,x}|^2 =0
        \quad \text{and} \quad \lim_{S \to \infty} \sup_{x\in X} \sum_{y\notin B(x,S)} |T_{x,y}|^2 =0.
    \end{equation}
    In fact, by definition we know that $(T)_{x\in X}$ is vector-wise uniformly quasi-local if and only if for any $\varepsilon>0$ and $F \in \CC_u[X]$, there exists $S>0$ such that for any $A,B \subseteq X$ with $d(A,B) \geq S$ then $\|\chi_A T \chi_B (F_x)\| \leq \varepsilon$ and $\|\chi_A T^* \chi_B (F_x)\| \leq \varepsilon$, where $F_x\in \ell^2(X)$ is defined by $F_x(y) \coloneqq F(y,x)$. It is easy to see that this is equivalent to the following: for any $\varepsilon>0$ and $R>0$, there exists $S>0$ such that for any $A,B \subseteq X$ with $d(A,B) \geq S$ and any $\xi \in \ell^2(X)$ with $\diam(\supp(\xi)) \leq R$ and $\|\xi\|_\infty \leq 1$ we have $\|\chi_A T \chi_B \xi\| \leq \varepsilon$ and $\|\chi_A T^* \chi_B \xi\| \leq \varepsilon$. Taking $A$ to be $X \setminus B(x,S)$ and using bounded geometry, this is also equivalent to that for any $\varepsilon>0$, there exists $S>0$ such that for any $B \subseteq X$ and any $x\in X$ we have $\|\chi_{X \setminus \Nd_S(B)} T \chi_B (\delta_{x})\| \leq \varepsilon$ and $\|\chi_{X \setminus \Nd_S(B)} T^* \chi_B (\delta_{x})\| \leq \varepsilon$. Finally, note that it suffices to consider the case when $B=\{x\}$, and hence we conclude the claim.

    In conclusion, we obtain the following:
    \begin{lemma}\label{lem:char for LL2G coarse groupoids}
        Let $(X,d)$ be a discrete metric space with bounded geometry and $G(X)$ be the associated coarse groupoid. Then for some fixed $x_0 \in X$, the map
        \begin{equation}\label{EQ:Theta coarse groupoid}
            \Theta: \L(L^2(G(X)))^{G(X)} \longrightarrow \{T \in \B(\ell^2(X)): T \text{ satisfies Equation (\ref{EQ:char for LL2G coarse groupoid})}\}
        \end{equation}
        given by $T \mapsto \Phi_{x_0}(T)$ is a $C^*$-isomorphism. Note that $\Theta$ is independent of the choice of $x_0$.
    \end{lemma}

    As we will see in the next subsection, the map $\Theta$ in (\ref{EQ:Theta coarse groupoid}) extends the isomorphism $\Theta: C^*_r(G(X)) \to C^*_u(X)$ in (\ref{EQ:Theta for coarse groupoid}). Hence we take the liberty of abusing the notation.
\end{example}

\subsection{Realisations for equivariant operators and left convolvers}\label{ssec:realisation for equiv}
This subsection is devoted to providing a concrete description for equivariant operators. Let us start with the case of groups, \emph{i.e.}, the unit space consists of a single point.

Let $\G=G$ be a discrete group, and hence $\L(L^2(G))=\B(\ell^2(G))$. For $f, g \in \ell^2(\G)$, the convolution $f \ast g$ is a well-defined bounded function on $G$ thanks to the Cauchy-Schwarz inequality. Recall that $f\in \ell^2(G)$ is a \emph{left convolver} (see \cite[Section 1.3]{AP2019a}) if $f \ast g \in \ell^2(G)$ whenever $g\in \ell^2(G)$. In this case,  the map $g\mapsto f \ast g$ is a bounded operator on $\ell^2(G)$ (by the closed graph theorem), denoted by $\Lambda(f)$. It is shown in \cite[Theorem 1.3.6]{AP2019a} that an operator $T \in \B(\ell^2(G))$ is $G$-equivariant \emph{if and only if} there exists a left convolver $f_T \in \ell^2(G)$ such that $T=\Lambda(f_T)$.

In the following, we would like to explore an analogous description for equivariant operators on groupoids. Let $\G$ be a locally compact \'{e}tale groupoid with unit space $\Gz$. Similar to the assumption of square-summability in the group case, we consider functions $f\in C_{b}(\G)$ satisfying the following:
\begin{enumerate}[label= (c.\arabic*)]
    \item\label{item.leftConvolver1} $f|_{\G_{x}}\in\ell^{2}(\G_{x})$ and $f^{*}|_{\G_{x}}\in\ell^{2}(\G_{x})$ for every $x\in\Gz$;
    \item\label{item.leftConvolver2} $\sup_{x\in\Gz}\|f|_{\G_{x}}\|_{\ell^{2}(\G_{x})}<\infty$, and $\sup_{x\in\Gz}\|f^{*}|_{\G_{x}}\|_{\ell^{2}(\G_{x})}<\infty$,
\end{enumerate}
where $f^*$ is defined by the same formula (\ref{EQ:star}).
Then for $g\in L^2(\G)$, the Cauchy-Schwarz inequality implies that the convolution $f\ast g$ and $f^* \ast g$ by (\ref{EQ:convolution}) are well-defined, and are bounded functions on $\G$.

\begin{definition}\label{def.left convolver}
    Let $\G$ be a locally compact \'{e}tale groupoid. A function $f\in C_{b}(\G)$ is called a \textit{left convolver} if $f$ satisfies \ref{item.leftConvolver1} and \ref{item.leftConvolver2}, and both $f*g$ and $f^{*}*g$ belong to $L^{2}(\G)$ for every $g\in L^{2}(\G)$.
\end{definition}

For a left convolver $f \in C_b(\G)$, it is routine to check that the map $g \mapsto f \ast g$ is an adjointable operator on $L^2(\G)$, still denoted by $\Lambda(f)$. It is clear that $\Lambda(f)^*=\Lambda(f^*)$. Note that when $f \in C_c(\G)$, this coincides with the left regular representation defined by (\ref{EQ:Lambda}).

\begin{remark}
    Comparing with the case of groups, readers might wonder why we need requirements on both $f$ and $f^*$ in \ref{item.leftConvolver1}, \ref{item.leftConvolver2} and Definition \ref{def.left convolver}. The reason lies in the fact that on Hilbert $C^*$-modules, a bounded module morphism might \emph{not} be adjointable in general. We provide some examples below.
\end{remark}

\begin{example}
    Let $\G = \mathbb{N}\times\mathbb{N}$ be the pair groupoid from Section \ref{sssec:pair groupoids}.
    Consider the function $f \in C_b(\G)$ given by $f((m,n)):=m^{-1}$. Direct calculations show that $f|_{\G_{(n,n)}}\in\ell^{2}(\NN)$ for each $n\in \NN$, and $\sup_{n\in\NN} \|f|_{\G_{(n,n)}}\|_2 <\infty$. On the other hand, note that $f^*$ is given by $f^*((m,n)) = n^{-1}$, and hence $f^*|_{\G_{(n,n)}} \notin \ell^2(\G_{(n,n)})$ for any $n\in \NN$. Moreover, note that for the function $g:\G \to \RR$ defined by $g((m,n)) = (nm)^{-1}$, it is clear that $g\in L^{2}(\G)$ while
    \[
        f*g ((m,n)) = \sum_{m_{1}\in\mathbb{N}}f((m,m_{1}))g((m_{1},n)) = (nm)^{-1}\sum_{m_{1}\in\mathbb{N}} m_{1}^{-1} = \infty.
    \]
    This illustrates that the requirements on $f^*$ in \ref{item.leftConvolver1} and \ref{item.leftConvolver2} \emph{cannot} be deduced from those on $f$, and both of them are used to ensure that the convolutions by $f$ and $f^*$ are well-defined.
\end{example}

The following example shows that even if a function $f \in C_b(\G)$ satisfies \ref{item.leftConvolver1} and \ref{item.leftConvolver2}, it might occur that the convolution by $f$ is a bounded operator on $L^2(\G)$ while \emph{not} adjointbale. This example essentially comes from \cite[Example 2.1.2]{MarkovichTroitsky2005}, and explains that we need to consider both $f$ and $f^*$ in Definition \ref{def.left convolver}.

\begin{example}
    We consider the groupoid $\G = \NN \times \NN \times [0,1]$ with source and range maps given by $\s((m,n,x)) = (n,n,x)$ and $\r((m,n,x)) = (m,m,x)$ for $(m,n,x) \in \G$, respectively. The composition map is given by $(m,n,x) \cdot (n,k,x)=(m,k,x)$, and the inverse map is by $(m,n,x) = (n,m,x)$ for $m,n,k \in \NN$ and $x\in [0,1]$.

    We take the product topology on $\G$, and then identify $C_b(\mathcal{G})$ with
    \[
        \{(f_{m,n})_{m,n\in\NN} : f_{m,n}\in C([0,1]) \text{ and } \sup_{(m,n) \in \NN \times \NN} \|f_{m,n}\|_\infty < \infty\},
    \]
    whose element can be regarded as an $\NN$-by-$\NN$ matrix with entries in $C([0,1])$. Elements in the subspace $C_c(\G)$ can be written as $(f_{m,n})_{m,n\in\NN}$ such that $f_{m,n}\neq 0$ for only finitely many $(m,n) \in \NN \times \NN$.

    Consider the function $f\in C_{b}(\mathcal{G})$ given by
    \[
        f = \begin{pmatrix}
            \varphi_{1} & \varphi_{2} & \cdots \\
            0           & 0           & \cdots \\
            \vdots      & \vdots      & \ddots
        \end{pmatrix}\text{ where }\varphi_{i} = \begin{cases}
            0,             & \text{on } [0,\frac{1}{i+1}]\text{ and }[\frac{1}{i},1];                                        \\
            1,             & \text{at the point } \frac{1}{2}(\frac{1}{i}+\frac{1}{i+1});                                    \\
            \text{linear}, & \text{on } [\frac{1}{i+1}, \frac{2i+1}{2i(i+1)}]\text{ and }[\frac{2i+1}{2i(i+1)},\frac{1}{i}].
        \end{cases}
    \]
    It is easy to see that $f^{*} = \begin{pmatrix}
            \varphi_{1} & 0      & \cdots \\
            \varphi_{2} & 0      & \cdots \\
            \vdots      & \vdots & \ddots
        \end{pmatrix}$ and $f^{*}\notin L^{2}(\mathcal{G})$ since the first coordinate of $\langle f^{*}, f^{*}\rangle$ is $\sum_{n\in\mathbb{N}}\varphi_{n}^{2}$, which is not in $C([0,1])$. It is routine to check that $f$ satisfies \ref{item.leftConvolver1} and \ref{item.leftConvolver2}.

    Given $(\xi_{i,j})\in L^{2}(\mathcal{G})$, we have:
    \[
        f*\xi = \begin{pmatrix}
            \sum_{k\in\mathbb{N}}\varphi_{k}\xi_{k,1} & \sum_{k\in\mathbb{N}}\varphi_{k}\xi_{k,2} & \cdots \\
            0                                         & 0                                         & \cdots \\
            \vdots                                    & \vdots                                    & \ddots
        \end{pmatrix},
    \]
    which belongs to $L^{2}(\G)$ using a standard approximating argument. This implies that the convolution $\Lambda(f)$ is a linear operator on $L^2(\G)$, which is also bounded by the closed graph theorem. On the other hand, consider $\eta = \begin{pmatrix}
            \chi_{[0,1]} & 0      & \cdots \\
            0            & 0      & \cdots \\
            \vdots       & \vdots & \ddots
        \end{pmatrix}\in L^{2}(\mathcal{G})$, and then $f^{*}*\eta =f^{*}\notin L^{2}(\mathcal{G})$.
\end{example}

Now we are in the position to provide the realisation for equivariant operators:

\begin{proposition}\label{prop.leftConvolverAndGEquivariant}
    Let $\G$ be a locally compact \'{e}tale groupoid.
    \begin{enumerate}
        \item For any left convolver $f \in C_b(\G)$, we have $\Lambda(f)\in \L\left(L^{2}(\G)\right)^{\G}$.
        \item Conversely, for any $T \in \L\left(L^{2}(\G)\right)^{\G}$ there exists a unique left convolver $f_T\in C_b(\G)$ such that $T = \Lambda(f_T)$.
    \end{enumerate}
\end{proposition}

\begin{proof}
    (1). It is clear that $\Lambda(f) \in \L(L^2(\G))$ for any left convolver $f \in C_b(\G)$. Moreover, note that for any $g\in C_c(\G)$ and $\xi \in L^2(\G)$, we have $\Lambda(f)\rho(g)\xi = f \ast (\xi \ast g) = (f \ast \xi) \ast g= \rho(g)\Lambda(f)\xi$. Hence $\Lambda(f)$ is $\G$-equivariant.

    (2). Given $T \in \L\left(L^{2}(\G)\right)^{\G}$, we consider the function $f_T:\G \to \CC$ defined by
    \begin{equation}\label{EQ:function fT}
        f_{T}(\gamma) \coloneqq \big(\Phi_{\s(\gamma)}(T)(\delta_{\s(\gamma)})\big)(\gamma) \quad \text{for} \quad \gamma\in\G.
    \end{equation}
    Since $T$ is $\G$-equivariant, then for any $\gamma, \alpha \in \G$ with $\r(\alpha)=\s(\gamma)$ we have
    \begin{equation}\label{EQ:function fT equiv}
        f_T(\gamma) = \big(\Phi_{\s(\gamma)}(T)(\delta_{\s(\gamma)})\big)(\gamma) = \big(\Phi_{\s(\alpha)}(T)(\delta_{\alpha})\big)(\gamma\alpha).
    \end{equation}
    On the other hand, for $\gamma\in \G$ we have
    \begin{align*}
        f_{T^{*}}(\gamma) & = \Phi_{\s(\gamma)}(T^*)(\delta_{\s(\gamma)})(\gamma) = \left(V_{\gamma^{-1}}\Phi_{\r(\gamma)}(T)^*V_{\gamma}\right) (\delta_{\s(\gamma)})(\gamma)                                                             \\
                          & = \left(\Phi_{\r(\gamma)}(T)^*V_{\gamma}\right) (\delta_{\s(\gamma)})(\gamma\gamma^{-1})  = \left\langle \delta_{\r(\gamma)}, \left(\Phi_{\r(\gamma)}(T)^*V_{\gamma}\right) (\delta_{\s(\gamma)})\right\rangle \\
                          & = \langle \Phi_{\r(\gamma)}(T)(\delta_{\r(\gamma)}),  V_{\gamma}(\delta_{\s(\gamma)})\rangle = \langle \Phi_{\r(\gamma)}(T)(\delta_{\r(\gamma)}),  \delta_{\gamma^{-1}}\rangle                                 \\
                          & =\overline{\Phi_{\r(\gamma)}(T)(\delta_{\r(\gamma)})(\gamma^{-1})}=\overline{f_{T}(\gamma^{-1})} = (f_{T})^{*}(\gamma),
    \end{align*}
    where $T$ being $\G$-equivariant is used in the second equality. This shows that $(f_{T})^{*}=f_{T^{*}}$. Hence for any $x\in \Gz$, we have $f_T|_{\G_x} = \Phi_{x}(T)(\delta_x)$ and $(f_T)^*|_{\G_x} = \Phi_{x}(T^*)(\delta_x)$, which implies that \ref{item.leftConvolver1} and \ref{item.leftConvolver2} hold for $f_T$.

    Since $x \mapsto \Phi_x(T)$ is a bounded continuous section in $\Gamma_b(E)$, it follows that $f_T$ is a bounded continuous function on $\G$. For any $\xi \in L^2(\G)$, note that $f_T\ast \xi$ is well-defined since $f_T$ satisfies \ref{item.leftConvolver1} and \ref{item.leftConvolver2}. Moreover, for any $\gamma \in \G$ with $\s(\gamma)=x$ we have
    \begin{align*}
        (f_T \ast \xi)(\gamma) & = \sum_{\alpha \in \G_x} f_T(\gamma \alpha^{-1}) \xi(\alpha) = \sum_{\alpha \in \G_x} \big(\Phi_{\r(\alpha)}(T)(\delta_{\r(\alpha)})\big)(\gamma \alpha^{-1}) \xi(\alpha) \\
                               & =\sum_{\alpha \in \G_x} \big(\Phi_x(T)(\delta_\alpha)\big)(\gamma) \xi(\alpha) = \big(\Phi_x(T)(\xi|_{\G_x})\big)(\gamma) = (T\xi)(\gamma),
    \end{align*}
    where we use (\ref{EQ:function fT equiv}) in the third equality. This implies that $f_T \ast \xi \in L^2(\G)$ and $\Lambda(f_T) = T$. Similarly, we have $(f_T)^* \ast \xi \in L^2(\G)$ and $\Lambda((f_T)^*) = T^*$. Therefore, $f_T$ is a left convolver and we conclude the proof.
\end{proof}

Thanks to Proposition \ref{prop.leftConvolverAndGEquivariant}, we provide an alternative picture for the isomorphism $\Theta$ in (\ref{EQ:Theta coarse groupoid}):

\begin{example}\label{eg:coarse groupoid left convolver}
    Following the notation in Example \ref{eg:coarse groupoid for LL2G}, for $T\in \L(L^2(G(X)))^{G(X)}$ Proposition \ref{prop.leftConvolverAndGEquivariant} implies that there exists $f_{T} \in C_b(G(X))$ such that $T=\Lambda(f_{T})$.  Considering the restriction of $f_T$ on $X \times X$, it is easy to see that the isomorphism $\Theta$ in (\ref{EQ:Theta coarse groupoid}) can be written as follows:
    \[
        \big(\Theta(T)(\xi)\big)(x):=\sum_{y\in X} f_{T}(x,y)\xi(y) \quad \text{for} \quad \xi\in \ell^2(X) \text{ and } x\in X.
    \]
    Therefore as claimed in Example \ref{eg:coarse groupoid for LL2G}, $\Theta$ in (\ref{EQ:Theta coarse groupoid}) extends the isomorphism $\Theta: C^*_r(G(X)) \to C^*_u(X)$ in (\ref{EQ:Theta for coarse groupoid}), which explains the notation.
\end{example}

We end this section with some discussions on multiplication operators.
Recall that for $g\in C_b(\G)$, the multiplication operator $M(g): L^2(\G) \to L^2(\G)$ is defined by $(M(g)\xi)(\gamma)=g(\gamma)\xi(\gamma)$ for $\xi \in L^2(\G)$ and $\gamma\in \G$. 

To simplify the statement, we say that a function $g\in C_b(\G)$ is \emph{$\G$-equivariant} if $g(\beta)=g(\beta \gamma^{-1})$ for any $\beta, \gamma \in \G$ with $\s(\beta) = \s(\gamma)$.


\begin{lemma}\label{lem:equiv function}
    Let $\G$ be a locally compact \'{e}tale groupoid and $g\in C_b(\G)$. Then the multiplication operator $M(g)$ is $\G$-equivariant \emph{if and only if} $g$ is $\G$-equivariant.
\end{lemma}

\begin{proof}
    Assume that $g$ is $\G$-equivariant, then for $h\in C_c(\G)$, $\xi\in L^2(\G)$ and $\gamma\in \G$, we have:
    \begin{align*}
        \big( \rho(h)M(g)\xi \big)(\gamma) & = \sum_{\beta\in\G_{\s(\gamma)}}h(\beta)g(\gamma\beta^{-1})\xi(\gamma\beta^{-1}) = \sum_{\beta\in\G_{\s(\gamma)}}h(\beta)g(\gamma)\xi(\gamma\beta^{-1}) \\
                                           & = g(\gamma) \cdot \sum_{\beta\in\G_{\s(\gamma)}}h(\beta)\xi(\gamma\beta^{-1}) = \big( M(g)\rho(h)\xi \big)(\gamma).
    \end{align*}
    Conversely, it follows from Lemma \ref{equivariant-fiber} that $V_{\gamma}\Phi_{\s(\gamma)}(M(g))=\Phi_{\r(\gamma)}(M(g)) V_{\gamma}$ for all $\gamma\in \G$. Hence for $\beta \in \G_{\s(\gamma)}$, we have
    \[
        g(\beta)\delta_{\beta \gamma^{-1}}= V_\gamma(g(\beta)\delta_\beta) = V_{\gamma}\Phi_{\s(\gamma)}(M(g))(\delta_{\beta})= \Phi_{\r(\gamma)}(M(g)) V_{\gamma} (\delta_{\beta})=g(\beta \gamma^{-1}) \delta_{\beta \gamma^{-1}},
    \]
    which shows that $g$ is $\G$-equivariant.
\end{proof}

On the other hand, the following lemma explores when an equivariant operator can be realised as a multiplication operator:

\begin{lemma}\label{cor.equivariantMultiplier}
    Let $\G$ be a locally compact \'{e}tale groupoid. For $T \in \L(L^2(\G))^{\G}$, let $f_T$ be the associated left convolver from Proposition \ref{prop.leftConvolverAndGEquivariant}. Then $T=M(g)$ for some $g\in C_b(\G)$ \emph{if and only if} $\supp(f_T) \subseteq \Gz$. In this case, $f_T = g|_{\Gz}$ and hence $M(g) = \Lambda(g|_{\Gz})$.
\end{lemma}

\begin{proof}
    If $T=M(g)$ for some $g\in C_b(\G)$, from (\ref{EQ:function fT}) we have:
    \[
        f_T(\gamma)=\big(\Phi_{\s(\gamma)}(M(g))(\delta_{\s(\gamma)})\big)(\gamma) = g(\s(\gamma))\delta_{\s(\gamma)}(\gamma) =
        \begin{cases}
            g(\gamma), & \text{ if } \gamma \in \Gz; \\
            0,         & \text{otherwise}.
        \end{cases}
    \]
    Hence we obtain that $\supp(f_T) \subseteq \Gz$ and $f_T = g|_{\Gz}$. Conversely, assume that $\supp(f_T) \subseteq \Gz$. Then for any $\xi \in L^2(\G)$ and $\gamma \in \G$, we have
    \[
        (f_T\ast \xi)(\gamma) = \sum_{\beta \in \G_{\s(\gamma)}} f_T(\gamma \beta^{-1}) \xi(\beta) = f_T(\r(\gamma)) \xi(\gamma).
    \]
    Setting $g(\gamma) = f_T(\r(\gamma))$ for $\gamma\in \G$, it follows that $g\in C_b(\G)$ is $\G$-equivariant and $\Lambda(f_T)=M(g)$. Hence we conclude the proof.
\end{proof}

\section{Compactly supported and quasi-local operators for groupoids}\label{sec.quasi}

In this section, we introduce the key notion of this paper, \emph{i.e.}, compactly supported and quasi-local operators for groupoids.
Moreover, we show that the reduced groupoid $C^*$-algebra can be recovered by equivariant and compactly supported operators.


\begin{definition}\label{defn:separated}
    Let $\G$ be a locally compact \'{e}tale groupoid and $K\subseteq \G$.
    For $f,g\in C_{b}(\G)$, we say that $f$ and $g$ are \emph{$K$-separated} if $\big(K\cdot \supp(f)\big)\cap \supp(g)=\emptyset$ and $\supp(f)\cap \big(K\cdot \supp(g)\big)=\emptyset$, \emph{i.e.}, $\gamma_{2}\gamma_{1}^{-1}\notin K$ and $\gamma_{1}\gamma_{2}^{-1}\notin K$ for any $\gamma_1 \in \supp(f)$ and $\gamma_2 \in \supp(g)$ with $\s(\gamma_1) = \s(\gamma_2)$.
\end{definition}

Recall that in the case of the coarse groupoids, compact subsets are always contained in closures of entourages. This inspires us to introduce the following:

\begin{definition}\label{def.operatorCompactSupportAndQuasiLocal}
    Let $\G$ be a locally compact \'etale groupoid and $T\in \L(L^2(\G))$. We define the following:
    \begin{enumerate}
        \item $T$ is \emph{supported in $K \subseteq \G$} if $gTf=0$ for any $K$-separated functions $f,g\in C_{b}(\G)$.
        \item  $T$ is \emph{compactly supported} if $T$ is supported in some compact subset of $\G$.
        \item $T$ is \textit{quasi-local} if for any $\varepsilon>0$, there exists a compact subset $K\subset \G$ such that for any $K$-separated functions $f,g\in C_{b}(\G)$ we have $\|gTf\| < \varepsilon\|g\|_{\infty}\|f\|_{\infty}$.
    \end{enumerate}
\end{definition}


It is clear that compactly supported operators are always quasi-local. Also note that $\L(L^2(\G))$ admits approximating units consisting of elements in $C_c(\G)$, hence the following is straightforward:

\begin{lemma}\label{compact-supported-functions}
    For an operator $T\in \L(L^2(\G))$, we have the following:
    \begin{enumerate}
        \item \label{compact-supported-functions-1} $T$ is compactly supported \emph{if and only if} there exists a compact subset $K\subset \G$ such that $gTf=0$ for any $K$-separated functions $f,g\in C_c(\G)$;
        \item \label{compact-supported-functions-2} $T$ is quasi-local \emph{if and only if} for any $\varepsilon>0$, there exists a compact subset $K\subset \G$ such that for any $K$-separated functions $f,g\in C_c(\G)$ we have $\|gTf\| < \varepsilon\|g\|_{\infty}\|f\|_{\infty}$.
    \end{enumerate}
\end{lemma}

%

Using the slicing map $\Phi$ from (\ref{eq.PhiMap}), we provide the following ``fibre-wise'' characterisation for compactly supported and quasi-local operators:

\begin{proposition}\label{compactfiberwisely}
    Let $\G$ be a locally compact \'{e}tale groupoid and $T \in \L\left(L^{2}(\G)\right)$. We have the following:
    \begin{enumerate}
        \item\label{compactfiberwisely-1} $T$ is supported in a compact set $K\subset \G$ \emph{if and only if} $\chi_{A_x}\Phi_x(T)\chi_{B_x}=0$ for any $x\in\Gz$ and subsets $A_x, B_x\subset \G_x$ with $A_{x}\cap (K \cdot B_x)=\emptyset$ and $(K\cdot A_x) \cap B_{x}=\emptyset$;
        \item\label{compactfiberwisely-2} $T$ is a quasi-local operator \emph{if and only if} for any $\varepsilon>0$ there exists a compact subset $K\subset \G$ such that $\|\chi_{A_x}\Phi_x(T)\chi_{B_x}\| < \varepsilon$ for any $x\in\Gz$ and subsets $A_x, B_x\subset \G_x$ with $A_{x}\cap (K \cdot B_x)=\emptyset$ and $(K\cdot A_x) \cap B_{x}=\emptyset$.
    \end{enumerate}
\end{proposition}

\begin{proof}
    Here we only provide the proof for \ref{compactfiberwisely-2}, since the other is similar and simpler.

    \emph{Necessity}: Given $\varepsilon>0$, we take a compact subset $K\subseteq \G$ satisfying the condition in Definition \ref{def.operatorCompactSupportAndQuasiLocal}(3). Fix an $x\in \Gz$. For any finite subsets $A_x, B_x\subset \G_x$ with $A_{x}\cap (K \cdot B_x)=\emptyset$ and $(K\cdot A_x) \cap B_{x}=\emptyset$, take open subsets $U$ and $V$ in $\G$ containing $A_x$ and $B_x$, respectively, such that $UV^{-1}\cap K=\emptyset$ and $VU^{-1}\cap K=\emptyset$. By Urysohn's lemma, there exist $f,g\in C_b(\G)$ with range in $[0,1]$, $\supp(g)\subset U$ and $\supp(f)\subset V$ such that $g|_{\G_x} = \chi_{A_x}$ and $f|_{\G_x} = \chi_{B_x}$. It is then clear that $f$ and $g$ are $K$-separated. Hence we obtain:
    \[
        \|\chi_{A_x}\Phi_x(T)\chi_{B_x}\| = \|\Phi_x(gTf)\| \leq \|gTf\| < \varepsilon.
    \]
    Using an approximating argument, the estimate above holds for any subsets $A_x, B_x$ in $\G_x$ with $A_{x}\cap (K \cdot B_x)=\emptyset$ and $(K\cdot A_x) \cap B_{x}=\emptyset$.

    \emph{Sufficiency}: Given $\varepsilon>0$, we take a compact subset $K\subseteq \G$ satisfying the assumption in (2). For any $K$-separated functions $f,g\in C_b(\G)$, we have
    \[
        \|gTf\| = \sup_{x\in \Gz} \|\Phi_x(gTf)\| = \sup_{x\in \Gz} \|g|_{\G_x} \cdot \Phi_x(T) \cdot f|_{\G_x}\|.
    \]
    For $x\in \Gz$, denote $A_x:=\supp(g|_{\G_x})$ and $B_x:=\supp(f|_{\G_x})$. It is clear that $A_{x}\cap (K \cdot B_x)=\emptyset$ and $(K\cdot A_x) \cap B_{x}=\emptyset$. Therefore, the assumption implies that
    \[
        \|g|_{\G_x} \cdot \Phi_x(T) \cdot f|_{\G_x}\| = \|g|_{\G_x} \cdot \chi_{A_x}\Phi_x(T) \chi_{B_x}\cdot f|_{\G_x}\| < \varepsilon \|g\|_\infty\|f\|_\infty,
    \]
    which shows that $\|gTf\| \leq \varepsilon \|g\|_\infty\|f\|_\infty$.
\end{proof}

Inspired by Proposition \ref{compactfiberwisely}, we introduce the following:

\begin{definition}\label{defn: unif. quasiloc}
    Let $\G$ be a locally compact \'{e}tale groupoid with unit space $\Gz$. A family $(T_{x})_{x\in\Gz} \in \prod_{x\in\Gz}\B\left(\ell^2(\G_{x})\right)$ is said to be \emph{compactly uniformly quasi-local} if for any $\varepsilon>0$ there exists a compact subset $K\subseteq \G$ such that for any $x\in\Gz$ and $A_{x}, B_{x}\subseteq \G_{x}$ with $A_{x}\cap (K \cdot B_x)=\emptyset$ and $(K\cdot A_x) \cap B_{x}=\emptyset$, we have $\| \chi_{A_{x}}T_x\chi_{B_{x}}\| < \varepsilon$.
\end{definition}

It is clear that the notion of compactly uniform quasi-locality implies vector-wise uniform quasi-locality introduced in Definition \ref{defn:vectorwise quasiloc}. Hence combining Theorem \ref{cor:char for LL2G} with Proposition \ref{compactfiberwisely}, we reach the following:

\begin{corollary}\label{cor:char for quasi-locality}
    Let $\G$ be a locally compact \'{e}tale groupoid with unit space $\Gz$. For $(T_{x})_{x\in\Gz} \in \prod_{x\in\Gz}\B\left(\ell^2(\G_{x})\right)$, the following are equivalent:
    \begin{enumerate}
        \item there exists a quasi-local operator $T \in \L(L^2(\G))$ such that $\Phi(T)=(T_x)_{x\in \Gz}$;
        \item the map $x\mapsto T_x$ is a continuous section of $E$ introduced in (\ref{EQ:E}) and $(T_x)_{x\in \Gz}$ is compactly uniformly quasi-local.
    \end{enumerate}
    A similar result also holds for the case of compactly supported operators.
\end{corollary}

As revealed in Section \ref{ssec:dense subset}, we can simplify Proposition \ref{compactfiberwisely} for dense subsets as follows. The proof is similar to that for Proposition \ref{lem.charOfBL2G2-densesubset}, hence omitted.


\begin{lemma}\label{quasi-dense}
    Let $\G$ be a locally compact \'{e}tale groupoid with unit space $\Gz$, $X$ a dense subset of $\Gz$ and $T\in \L(L^2(\G))$. Then we have the following:
    \begin{enumerate}
        \item $T$ is supported in a compact set $K\subset \G$ \emph{if and only if} $\chi_{A_x}\Phi_x(T)\chi_{B_x}=0$ for any $x\in X$ and subsets $A_x, B_x\subset \G_x$ with $A_{x}\cap (K \cdot B_x)=\emptyset$ and $(K\cdot A_x) \cap B_{x}=\emptyset$;
        \item $T$ is quasi-local \emph{if and only if} for any $\varepsilon>0$ there exists a compact subset $K\subset \G$ such that $\|\chi_{A_x}\Phi_x(T)\chi_{B_x}\| < \varepsilon$ for any $x\in X$ and subsets $A_x, B_x\subset \G_x$ with $A_{x}\cap (K \cdot B_x)=\emptyset$ and $(K\cdot A_x) \cap B_{x}=\emptyset$.
    \end{enumerate}
\end{lemma}

Analogous to Definition \ref{defn: unif. quasiloc}, for a dense subset $X$ of $\Gz$ we say that a family $(T_{x})_{x\in X} \in \prod_{x\in X}\B\left(\ell^2(\G_{x})\right)$ is \emph{compactly uniformly quasi-local} if for any $\varepsilon>0$ there exists a compact subset $K\subseteq \G$ such that for any $x\in X$ and $A_{x}, B_{x}\subseteq \G_{x}$ with $A_{x}\cap (K \cdot B_x)=\emptyset$ and $(K\cdot A_x) \cap B_{x}=\emptyset$, we have $\| \chi_{A_{x}}T_x\chi_{B_{x}}\| < \varepsilon$.

It is clear that in the case of dense subsets, the notion of compactly uniform quasi-locality also implies vector-wise uniform quasi-locality. Hence combining Corollary \ref{cor:char2 for LL2G} with Lemma \ref{quasi-dense}, we reach the following:

\begin{corollary}\label{rem:char for dense case}
    Let $\G$ be a locally compact \'{e}tale groupoid with unit space $\Gz$ and $X$ be a dense subset of $\Gz$. For $(T_{x})_{x\in X} \in \prod_{x\in X}\B\left(\ell^2(\G_{x})\right)$, the following are equivalent:
    \begin{enumerate}
        \item there exists a quasi-local operator $T \in \L(L^2(\G))$ such that $\res\circ\Phi(T)=(T_x)_{x\in X}$;
        \item the map $x\mapsto T_x$ is a continuous extendable section of $E_X$ introduced in (\ref{EQ:E}) and $(T_x)_{x\in X}$ is compactly uniformly quasi-local.
    \end{enumerate}
    Moreover, we can omit the requirement of extendableness in condition (2) when $\Gz$ is the Stone-\v{C}ech compactification of $X$ (with the inclusion map).

    A similar result also holds for the case of compactly supported operators.
\end{corollary}


The following example shows that our definition of quasi-locality for groupoids recovers Definition \ref{defn:quasi-locality} in the case of metric spaces:

\begin{example}\label{ex:coarse groupoid again}
    Let $(X,d)$ be a discrete metric space with bounded geometry and $G(X)$ the associated coarse groupoid. Combining Example \ref{eg:coarse groupoid for LL2G} and Corollary \ref{rem:char for dense case}, we obtain that a $G(X)$-equivariant operator $T \in \L(L^2(G(X)))$ is quasi-local \emph{if and only if} there exists a quasi-local operator (in the sense of Definition \ref{defn:quasi-locality}) $T_0 \in \B(\ell^2(X))$ such that $\res\circ \Phi(T) = (T_0)_{x\in X}$.
\end{example}


Analogous to the uniform Roe and quasi-local algebras for metric spaces introduced in Definition \ref{defn:unif. Roe} and \ref{defn:unif. quasilocal}, we define the following algebras for groupoids:

\begin{definition}\label{defn:unif. Roe and quasilocal for groupoids}
    Let $\G$ be a locally compact \'{e}tale groupoid.
    \begin{enumerate}
        \item Denote by $\mathbb{C}_{u}[\G]$ the set of all compactly supported operators in $\L(L^2(\G))$, and define the \emph{uniform Roe algebra $C^*_u(\G)$ of $\G$} to be its norm closure in $\L(L^2(\G))$.
        \item Denote by $C^{*}_{uq}(\G)$ the set of all quasi-local operators in $\L(L^2(\G))$, which forms a $C^*$-algebra and is called the \emph{uniform quasi-local algebra of $\G$}.
    \end{enumerate}
\end{definition}

We also consider the equivariant counterpart. For a locally compact \'{e}tale groupoid $\G$, denote by $\CC_u[\G]^{\G}$ the $\ast$-subalgebra in $\CC_u[\G]$ consisting of $\G$-equivariant operators. Also denote by $C^*_u(\G)^{\G}$ and $C^{*}_{uq}(\G)^{\G}$ the $C^*$-subalgebras in $C^*_u(\G)$ and $C^{*}_{uq}(\G)$, respectively,  consisting of $\G$-equivariant operators.

\begin{remark}\label{rem:coincidence}
    Recall that Anantharaman-Delaroche introduced the notion of uniform $C^*$-algebra of a groupoid in \cite[Definition 6.1]{ADExact} with the same notation $C^*_u(\G)$. We will show in Section \ref{ssec:compact supp case AD} that it coincides with Definition \ref{defn:unif. Roe and quasilocal for groupoids}(1).
\end{remark}

The following proposition connects the algebra $\CC_u[\G]^{\G}$ with $C_c(\G)$, which parallels the result that the equivariant part of the algebraic uniform Roe algebra for a group coincides with the group algebra.

\begin{proposition}\label{Roereduced}
    Let $\G$ be a locally compact \'{e}tale groupoid. Then $\mathbb{C}_u[\G]^{\G}$ is $\ast$-isomorphic to $C_c(\G)$. As a consequence, the norm closure of $\CC_u[\G]^{\G}$ in $\L(L^2(\G))$ is $C^*$-isomorphic to the reduced groupoid $C^{*}$-algebra $C^{*}_r(\G)$.
\end{proposition}

\begin{proof}
    Recall that we have the faithful representation $\Lambda: C_c(\G) \to \L(L^2(\G))$ defined in (\ref{EQ:Lambda}). Hence it suffices to show that the image $\Lambda(C_c(\G))$ coincides with the $\ast$-algebra $\CC_u[\G]^{\G}$.

    Direct calculations show that if $h \in C_c(\G)$ has support in a compact set $K \subseteq \G$, then the operator $\Lambda(h)$ has support in $K$. Together with Proposition \ref{prop.leftConvolverAndGEquivariant}(1), this implies that $\Lambda(C_c(\G)) \subseteq \CC_u[\G]^{\G}$. On the other hand, for $T\in \mathbb{C}_u[\G]^{\G}$ with support in a compact subset $K \subseteq \G$, Proposition \ref{prop.leftConvolverAndGEquivariant}(2) implies that there exists a unique $f_T \in C_b(\G)$ defined in (\ref{EQ:function fT}) such that $T=\Lambda(f_{T})$. Recall that
    \[
        f_{T}(\gamma) = \big(\Phi_{\s(\gamma)}(T)(\delta_{\s(\gamma)})\big)(\gamma)= \langle \delta_{\gamma}, \Phi_{\s(\gamma)}(T) (\delta_{\s(\gamma)}) \rangle \quad \text{for} \quad \gamma\in\G.
    \]
    Hence Proposition \ref{compactfiberwisely} implies that $f_T$ has support in $K\cup K^{-1}$, which concludes that $f_T \in C_c(\G)$.
\end{proof}

\begin{remark}
    Note that $\CC_u[\G]^{\G}$ is contained in $C^*_u(\G)^{\G}$, hence doing completing we obtain that the norm closure of $\CC_u[\G]^{\G}$ in $\L(L^2(\G))$ is contained in $C^*_u(\G)^{\G}$. However in general, it is unclear whether these two algebras are identical. We provide a sufficient condition in the next section.
\end{remark}

To conclude this section, we discuss a relative commutant picture for quasi-local operators on groupoids, which generalises the metric space case proved in \cite[Theorem 2.8 ``(i) $\Leftrightarrow$ (ii)'']{SpakulaTikuisis} (see also~\cite[Proposition 3.3]{BaoChenZhang}). Recall the following notion:


\begin{definition}\label{defn:variation}
    Let $\G$ be a groupoid, $K$ a symmetric compact subset of $\G$ and $\delta>0$. A function $h\in C_b(\G)$ is said to have \textit{$(K,\delta)$--variation}, if for any $\alpha, \beta\in \G$ with $\alpha\beta^{-1}\in K$, we have $|h(\alpha)-h(\beta)|< \delta$.
\end{definition}

\begin{proposition}\label{prop:relative picture}
    Let $\G$ be a locally compact \'{e}tale groupoid and $T\in \L(L^2(\G))$. Then the following are equivalent:
    \begin{enumerate}
        \item $T$ is quasi-local;
        \item for any $\varepsilon>0$, there exist $\delta>0$ and a symmetric
              compact $K \subseteq \G$ such that for any $h\in C_b(\G)$ with
              $\|h\|_\infty=1$ and $(K, \delta)$-variation, we have $\|[T,h]\|<\varepsilon$.
    \end{enumerate}
\end{proposition}

\begin{proof}
    We follow the outline of the proof for~\cite[Theorem 2.8 ``(i) {$\Leftrightarrow$} (ii)'']{SpakulaTikuisis}.

    ``$(2)\Rightarrow (1)$'':
    Given $\varepsilon>0$, choose $n \in \NN$ such that $1/n\leq \delta$ and
    take $K'=K\cup K^2\cup\cdots\cup K^n$. For any $K'$-separated $f, g\in
        C_c(\G)$ with norm $1$, there exists $h\in C_b(\G)$ with $(K,\delta)$-variation and norm $1$ such that $h$ equals $1$ on $\supp(f)$ and $0$ on $\supp(g)$. This can be achieved by Urysohn's lemma. Then we have:
    \[
        \|gTf\|=\|gThf\|=\|g[T,h]f\|<\varepsilon,
    \]
    which implies that $T$ is a quasi-local operator by Lemma~\ref{compact-supported-functions}(2).

    ``$(1)\Rightarrow (2)$'': For simplicity, we assume that $\|T\|=1$. For any $0<\varepsilon<1$, take $\delta \leq \varepsilon/16$ and choose $N \in \NN$ such that $\delta<1/N<\varepsilon/8$. Due to quasi-locality, there exists a symmetric compact subset $K \subseteq \G$ such that for any $K$-separated $f, g\in C_b(\G)$ with norm $1$, we have $\|gTf\|\leq \varepsilon/(2N^2)$.

    For $h\in C_b(\G)$ with norm $1$ and $(K, \delta)$-variation, we claim that $\|[T, h]\|<\varepsilon$. Without loss of generality, we assume that $h\geq 0$. Fixing $x\in \Gz$, denote by $h_x$ the restriction of $h$ on $\G_x$. We set
    \[
        A_1:=h_x^{-1}\big([0,\frac{1}{N}]\big), \quad A_i:=h_x^{-1}\big((\frac{i-1}{N},\frac{i}{N}]\big) \quad \text{for} \quad  i=2,\ldots, N.
    \]
    It is clear that $(A_i A^{-1}_j)\cap K = \emptyset$ for $|i-j|>1$. Applying Proposition \ref{compactfiberwisely}(2), we have
    \[
        \|\chi_{A_i} \Phi_x(T) \chi_{A_j}\|\leq \frac{\varepsilon}{2N^2}.
    \]
    On the other hand, note that
    \[
        \|h_x-\sum^N_{i=1}\frac{i}{N}\chi_{A_i}\|_\infty \leq 1/N.
    \]
    Hence we obtain:
    \begin{align*}
        \|[\Phi_x(T), h_x]\|
         & \leq \frac{2}{N}+\|[\Phi_x(T), \sum^N_{i=1}\frac{i}{N}\chi_{A_i}]\|                                                                                         \\
         & = \frac{2}{N}+\|(\sum^N_{j=1}\chi_{A_j})(\sum^N_{i=1}\frac{i}{N}\Phi_x(T)\chi_{A_i})-(\sum^N_{j=1}\frac{j}{N}\chi_{A_j}\Phi_x(T))(\sum^N_{i=1}\chi_{A_i})\| \\
         & = \frac{2}{N}+\|\sum^N_{i,j=1}(\frac{i-j}{N})\chi_{A_j} \Phi_x(T) \chi_{A_i}\|                                                                              \\
         & \leq \frac{2}{N}+ \sum_{|i-j|>1}\|\chi_{A_j} \Phi_x(T) \chi_{A_i}\|+ \|\sum_{|i-j|\leq 1}(\frac{i-j}{N})\chi_{A_j} \Phi_x(T) \chi_{A_i}\|.
    \end{align*}
    Each item in the first sum is dominated by $\varepsilon/(2N^2)$, so the first item is less than $\varepsilon/2$. The second sum can be divided into two parts for $j=i+1$ and $j=i-1$, each of which is dominated by $1/N$. Hence we have
    \[
        \|[\Phi_x(T), h_x]\| \leq \frac{2}{N}+ \frac{\varepsilon}{2}+ \frac{2}{N} \leq \varepsilon,
    \]
    which implies that
    \[
        \|[T,h]\|=\sup_{x\in \Gz}\|[\Phi_x(T),h_x]\|<\varepsilon.
    \]
    Hence we conclude the proof.
\end{proof}

\section{Main theorem}\label{sec.mainThm}

In the previous section, we introduced the notion of quasi-locality for groupoids and constructed several associated $C^*$-algebras (see Definition \ref{defn:unif. Roe and quasilocal for groupoids}).
Now we are in the position to prove the following main result of this paper, showing that they coincide with each other under the assumption of amenability.

\begin{theorem}\label{Main-theorem}
    Let $\G$ be a locally compact, $\sigma$-compact and \'{e}tale groupoid.
    If $\G$ is amenable, then we have $C^{*}_r(\G)= C^*_u(\G)^{\G} = C^{*}_{uq}(\G)^{\G}$.
\end{theorem}

Note that when $\G$ is the coarse groupoid associated to a discrete metric space with bounded geometry, Theorem \ref{Main-theorem} recovers the main result of \cite{SpakulaZhang} in the Hilbert space case (see Proposition \ref{prop:SZ}). See Section \ref{ssec:coarse groupoids} for more details.

The proof of Theorem \ref{Main-theorem} occupies the rest of this section. As mentioned in Section \ref{ssec:intro main contributions}, it relies heavily on the coarse geometry of groupoids. Roughly speaking, we assign a length function on the groupoid thanks to the $\sigma$-compactness, which induces a metric on each source fibre. Hence we can appeal to the weapon of coarse geometry for a family of metric spaces, and consult the idea of \cite[Theorem 3.3]{SpakulaZhang}. To start, let us recall some notion for metric families (comparing with Section \ref{ssec.motivation}).

A family of metric spaces $\{(X_i,d_{i})\}_{i\in I}$ is said to have \textit{uniformly bounded geometry} if for any $r>0$, we have
\[
    \sup_{i\in I}\sup_{x\in X_{i}}\sharp B_{X_{i}}(x,r)<\infty.
\]
For each $i\in I$, let $k_i$ be a kernel on $X_i$. We say that the family $(k_i)_{i\in I}$ has \emph{uniformly finite propagation} if there exists $S>0$ such that $k_i(x,y)=0$ whenever $d_i(x,y)>S$ for $x,y\in X_i$ and $i\in I$.



\begin{definition}\label{defn:uniform A}
    Let $\{(X_i,d_{i})\}_{i\in I}$ be a family of discrete metric spaces with uniformly bounded geometry.
    We say that $\{(X_i,d_{i})\}_{i\in I}$ has \textit{uniform Property A} if for any $R>0, \varepsilon>0$ and $i\in I$, there exists a normalised and symmetric kernel $k_i$ on $X_i$ of positive type and $(R,\varepsilon)$-variation such that the family $(k_i)_{i\in I}$ has uniformly finite propagation.
\end{definition}


\begin{definition}\label{defn:operators family version}
    Let $\{(X_i,d_{i})\}_{i\in I}$ be a family of metric spaces with uniformly bounded geometry, and $T_i$ be an operator in $\B(\ell^2(X_i))$ for each $i \in I$.
    \begin{enumerate}
        \item Say that $(T_i)_{i\in I}$ has \textit{uniformly finite propagation} if there exists $R>0$ such that for any $f,g\in\ell^{\infty}(X_{i})$ with $d_{i}(\supp(f),\supp(g))>R$, then $fT_{i}g = 0$.
        \item Given $\varepsilon>0$, $(T_i)_{i\in I}$ is said to have \textit{uniformly finite $\varepsilon$-propagation} if there exists $R>0$ such that for any $f,g\in\ell^{\infty}(X_{i})$ with $d_{i}(\supp(f),\supp(g))>R$, then $\|fT_{i}g\|<\varepsilon\|f\|_{\infty}\|g\|_{\infty}$.
        \item Say that $(T_i)_{i\in I}$ is \textit{uniformly quasi-local} if for any $\varepsilon>0$, the family $(T_i)_{i\in I}$ has uniformly finite $\varepsilon$-propagation.
    \end{enumerate}
\end{definition}

Recall that we defined the notion of compactly uniform quasi-locality for operators on source fibres of a given groupoid in Definition \ref{defn: unif. quasiloc}. In fact, this has close relation with Definition \ref{defn:operators family version}(3) since we can endow each source fibre with a metric as follows\footnote{We remark that these metrics have also been considered in \cite{MaWu} to study amenability for certain groupoids, using the language of length functions on groupoids.}. As explained above, this is the key ingredient to attack Theorem \ref{Main-theorem} where the coarse geometry of the underlying groupoid plays an important role.





Throughout the rest of this section, let $\G$ be a locally compact, $\sigma$-compact and \'{e}tale groupoid with unit space $\Gz$. We also fix a sequence of subsets $\{K_{n}\}_{n\in \NN}$ in $\G$ satisfying the following:
\begin{enumerate}[label= (M.\arabic*)]
    \item\label{item.m1} $K_{0}=\Gz$, and $K_{n}$ is compact and symmetric for $n\geq 1$.
    \item\label{item.m2} $K_n\subset K^{\circ}_{n+1}$ for $n\geq 1$, where $K^{\circ}_{n+1}$ denotes the interior of $K_{n+1}$.
    \item\label{item.m3} $K_n K_m\subset K_{n+m}$ for $n,m\geq 1$ and $\G=\bigcup_{n\geq 1} K_n$.
\end{enumerate}
Note that we do not require $K_0 \subset K^{\circ}_1$ in general to deal with the case that $K_{0}=\Gz$ might not be compact. However when $\Gz$ is compact, we further require that $K_0 \subset K^{\circ}_1$ for convenience.
Using $\{K_n\}_n$, we can endow a function $d_x$ on $\G_x \times \G_x$ for each $x\in \Gz$ as follows:
\begin{equation}\label{EQ:fibre metric}
    d_x(\gamma_1,\gamma_2)\coloneqq\inf\{n \in \NN: \gamma_1\gamma^{-1}_2\in K_n\} \quad \text{ for } \quad  \gamma_1,\gamma_2\in\G_{x}.
\end{equation}
It follows from \ref{item.m1}-\ref{item.m3} that $d_x$ is indeed a metric on $\G_x$.


\begin{lemma}\label{lem.uniformlyBoundedGeometr}
    With the same notation as above, the family $\{(\G_x,d_x)\}_{x\in\Gz}$ has uniformly bounded geometry.
\end{lemma}

\begin{proof}
    Since $\G$ is \'{e}tale, for each $n\in \NN$ there exists a constant $C_n>0$ such that $K_n$ can be covered by at most $C_n$-many open bisections, which implies that $\sharp (\G^{y}\cap K_{n})\leq C_{n}$ for each $y\in \Gz$. Note that for each $n\in \NN$, $x\in \Gz$ and $\alpha \in \G_x$, an element $\beta \in \G_x$ belongs to the ball $B_{\G_x}(\alpha, n)$ if and only if $\alpha\beta^{-1}\in K_{n}\cap \G^{\r(\alpha)}$. Hence the ball $B_{\G_x}(\alpha, n)$ contains at most $C_n$ points, which concludes the proof.
\end{proof}



The following result indicates that the amenability of $\G$ affects the coarse geometry of the metric family $\{(\G_x, d_x)\}_{x\in \Gz}$. This will be crucial in the proof of Theorem \ref{Main-theorem}.

\begin{lemma}\label{lem:amen implies unif. A}
    With the same notation as above, the family $\{(\G_x, d_x)\}_{x\in \Gz}$ has uniform Property A provided $\G$ is amenable.
\end{lemma}

\begin{proof}
    Since $\G$ is amenable, Lemma~\ref{lem:p.d.functions for amenability} implies that for any $\varepsilon>0$ and positive $n\in\mathbb{N}$, there exists a non-negative function $h\in C_{c}(\G)$ of positive type such that
    \begin{itemize}
        \item $h(x)\leq 1$ for any $x\in \Gz$, and $h(x)=1$ for any $x\in \r(K_{n})$;
        \item $|1- h(\gamma)|<\varepsilon$ for any $\gamma \in K_{n}$.
    \end{itemize}
    For each $x\in\Gz$, we define a function $k_{x}\colon\G_{x}\times\G_{x}\rightarrow [0,1]$ by
    \[
        k_{x}(\gamma,\beta) = h(\gamma\beta^{-1})\quad \text{for} \quad \gamma,\beta \in\G_{x}.
    \]
    It follows that $k_x$ is a symmetric kernel on $\G_x$ of positive type such that:
    \begin{itemize}
        \item the family $(k_x)_{x\in \Gz}$ has uniformly finite propagation;
        \item for $\gamma \in\G_{x}$ with $\r(\gamma)\in \r(K_{n})$, then $k_{x}(\gamma,\gamma)=1$;
        \item for $\gamma,\beta\in\G_{x}$ with $0<d_{x}(\gamma,\beta)\leq n$, then $|k_{x}(\gamma,\beta) - 1|<\varepsilon$.
    \end{itemize}
    Note that the kernel $k_{x}$ might \emph{not} have $(n,\varepsilon)$-variation, hence we consider another kernel $\widetilde{k}_{x}$ on $\G_{x}$ defined by
    \[
        \widetilde{k}_{x}(\gamma,\beta) = \left\{\begin{array}{ll}
            k_{x}(\gamma,\beta),   & \text{if } \r(\gamma),\r(\beta)\in \r(K_{n}); \\
            \delta_{\gamma,\beta}, & \text{otherwise}.
        \end{array}
        \right.
    \]
    Clearly $\widetilde{k}_{x}$ is a normalised symmetric kernel on $\G_x$ of positive type such that
    \begin{itemize}
        \item the family $(\widetilde{k}_{x})_{x\in \Gz}$ has uniformly finite propagation;
        \item $\widetilde{k}_{x}$ has $(n,\varepsilon)$-variation.
    \end{itemize}
    The last item comes from the fact that for any $\gamma\in\G_{x}$ with $\r(\gamma)\notin \r(K_n)$ and $\beta\in\G_{x}$ with $\beta \neq \gamma$, we have $d_{x}(\gamma,\beta)>n$. Hence we conclude the proof.
\end{proof}

The following result is basic and the proof is straightforward, hence omitted.

\begin{lemma}\label{lem:relation btw top and metric sep}
    Given $\alpha, \beta \in \G_x$ for $x\in \Gz$ and an integer $n>0$, we have that $\alpha \beta^{-1} \notin K_n$ \emph{if and only if} either $d_x(\alpha, \beta) > n$, or $\alpha = \beta$ with $\r(\alpha) \notin K_n$. If additionally $\Gz$ is compact, then $\alpha \beta^{-1} \notin K_n$ \emph{if and only if} $d_x(\alpha, \beta) > n$.
\end{lemma}

By virtue of the metrics from (\ref{EQ:fibre metric}) together with Proposition~\ref{compactfiberwisely} and Lemma \ref{lem:relation btw top and metric sep}, we reach the following:

\begin{lemma}\label{prop.uniformlyFibreWiseQuasiLocal}
    With the same notation as above and given $T \in \L(L^2(\G))$, we have:
    \begin{enumerate}
        \item if $T \in \CC_u[\G]$, then $(\Phi_{x}(T))_{x\in \Gz}$ has uniformly finite propagation.
        \item if $T \in C^{*}_{uq}(\G)$, then $(\Phi_{x}(T))_{x\in \Gz}$ is uniformly quasi-local.
    \end{enumerate}
    If additionally $\Gz$ is compact, then both of the converse hold as well.
\end{lemma}

\begin{remark}\label{20230510} 
When $\Gz$ is compact, the above result shows
  that the notion of compactly uniform quasi-locality from Definition \ref{defn:
    unif. quasiloc} coincides with that of uniform quasi-locality after endowing the
  metrics as above. However, note that this does not hold in general when $\Gz$ is
  non-compact. An easy example is to consider the identity operator, where each
  slice $\Phi_x(T)$ is again the identity. However, it is direct to show that the
  identity does not belong to $C^{*}_{uq}(\G)$ when $\Gz$ is non-compact.
  
To provide a more concrete example, let us consider that $X$ is a locally compact Hausdorff space which is not compact and regard it as a groupoid by setting $\G = \Gz = X$ with source and range maps being the identities on $X$. It is clear that $C^{*}_{uq}(\G)=C^{*}_{uq}(\G)^{\G}=C_{0}(X)$, and the constant function $1$ on $X$ is uniformly quasi-local, but not compactly uniformly quasi-local.
\end{remark}


\begin{remark}
Up till now, we have introduced various figurations of quasi-locality and the terminologies bear a striking resemblance. As suggested by the anonymous referee, we provide the following diagram (Figure 1) to conclude these notions and their relations. The right vertical implication can be checked directly using the properties of the metrics from (\ref{EQ:fibre metric}). Note that the converse to each direction is incorrect. 
\begin{figure}[h!]
  \centering
  \begin{tikzcd}[row sep=10ex, column sep = 10ex]
    \begin{tabular}{c}
      \text{Quasi-locality} \\
      \text{(Definition~\ref{def.operatorCompactSupportAndQuasiLocal})}
    \end{tabular}
    \arrow[d,Rightarrow,sloped, "\text{Corollary~\ref{cor:char for quasi-locality}}"]&
    \begin{tabular}{c}
      \text{Vector-wise uniform quasi-locality} \\
      \text{(Definition~\ref{defn:vectorwise quasiloc})}
    \end{tabular} \\
    \begin{tabular}{c}
      \text{Compactly uniform quasi-locality} \\
      \text{(Definition~\ref{defn: unif. quasiloc})}
    \end{tabular}
    \arrow[r, Rightarrow,"\text{Lemma \ref{prop.uniformlyFibreWiseQuasiLocal}}"]
    &\begin{tabular}{c}
      \text{Uniform quasi-locality} \\
      \text{(Definition~\ref{defn:operators family version})}
    \end{tabular}
    \arrow[u, Rightarrow]
  \end{tikzcd}
 \caption{A summary for different notions of quasi-locality}
\end{figure}
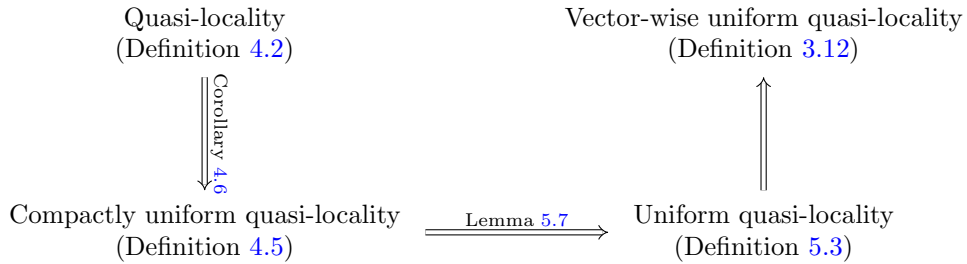
\end{remark}


To prove Theorem \ref{Main-theorem}, we need another description for uniform quasi-locality in terms of relative commutants introduced in \cite{SpakulaTikuisis} (see also Proposition \ref{prop:relative picture}). Let us start with some extra notation. Let $(X,d)$ be a discrete metric space. A function $f\in \ell^{\infty}(X)$ is called \textit{$C$-Lipschitz} for some constant $C>0$ if $|f(x)-f(y)|\leq C d(x,y)$ for any $x,y\in X$.
Given $L>0$ and $\varepsilon>0$, we write $T\in Commut_{X}(L,\varepsilon)$ if $\|[T,f]\|< \varepsilon$ for any $L$-Lipschitz $f\in \ell^{\infty}(X)$ with norm $1$.

Tracking the parameters in the proof of \cite[Lemma 5.2]{SpakulaZhang}, we obtain the following version for a metric family:

\begin{lemma}\label{lem:Lemma 5.2 in SZ20}
    Let $\{(X_i,d_i)\}_{i\in I}$ be a family of discrete metric spaces with uniformly bounded geometry and uniform Property A. For $\varepsilon>0$, $L>0$ and $M>0$, there exists $s>0$ such that for any $i\in I$ and $T_i\in\B(\ell^2(X_i))$ with $T_i\in Commut_{X_i}(L,\frac{\varepsilon}{12})$ and $\|T_i\|\leq 2M$, there is a unit vector $v_i\in\ell^{2}(X_i)$ with $\diam(\supp(v_i))\leq s$ such that
    \[
        \|T_iv_i\|\geq \|T_i\| - \frac{\varepsilon}{2}.
    \]
\end{lemma}

Now we are in the position to prove Theorem \ref{Main-theorem}. Before we dive into the details, let us explain the outline of the proof. The main goal is to approximate a given equivariant quasi-local operator $T$ using elements in $\Lambda(C_c(\G))$. As indicated in the proof of Lemma \ref{lem:amen implies unif. A}, the amenability of $\G$ produces a family of kernels $(k_x)_{x\in \Gz}$ which naturally produce a function $g \in C_c(\G)$. Moreover, these kernels can be modified to provide another family $(\widetilde{k}_x)_{x\in \Gz}$ with nice behaviour. Conducting a ``uniform'' proof of \cite[Theorem 3.3]{SpakulaZhang}, the family of operators determined by $(\widetilde{k}_x)_{x\in \Gz}$ indeed approximate the given operator $T$. Finally thanks to the quasi-locality of $T$, this approximating family of operators is simultaneously uniformly close to the slices of $\Lambda(g)$, which concludes the proof.

\begin{proof}[Proof of Theorem \ref{Main-theorem}]
    It is clear that $C^{*}_r(\G) \subseteq  C^*_u(\G)^{\G} \subseteq C^{*}_{uq}(\G)^{\G}$, hence it suffices to prove $C^{*}_{uq}(\G)^{\G} \subseteq C^{*}_r(\G)$. Fix an operator $T \in C^{*}_{uq}(\G)^{\G}$ and $\varepsilon>0$, we aim to construct an operator $S \in \Lambda(C_c(\G))$ such that $\|T-S\| \leq \varepsilon$.
    Throughout the proof, we also fix a sequence of subsets $\{K_{n}\}_{n \in \NN}$ of $\G$ satisfying \ref{item.m1}-\ref{item.m3}.
    Endow a metric $d_{x}$ on $\G_{x}$ defined by (\ref{EQ:fibre metric}) for each $x\in \Gz$.

    By Proposition \ref{compactfiberwisely}, there exists $n_0 \in \NN$ such that for any $x\in \Gz$ and $A_x, B_x \subseteq \G_x$ with $A_x \cap (K_{n_0} \cdot B_x) = \emptyset$, then $\| \chi_{A_{x}}\Phi_{x}(T)\chi_{B_{x}}\| \leq \frac{\varepsilon}{12}$. Moreover, it follows from Lemma \ref{prop.uniformlyFibreWiseQuasiLocal}(2) that the family $(\Phi_{x}(T))_{x\in \Gz}$ is uniformly quasi-local. Hence following the proof of \cite[Theorem 2.8, ``(ii) $\Rightarrow$ (i)'']{SpakulaTikuisis}, there exists $L>0$ (which only depends on $T$ and $\varepsilon$) such that $\Phi_{x}(T)\in Commut_{\G_{x}}(L,\frac{\varepsilon}{48})$ for any $x\in \Gz$.

    Thanks to Lemma \ref{lem.uniformlyBoundedGeometr} and Lemma \ref{lem:amen implies unif. A}, we can apply Lemma \ref{lem:Lemma 5.2 in SZ20} for $\frac{\varepsilon}{2}$, $L$ and $M \coloneqq \|T\|$ and obtain a constant $s>0$ satisfying the condition therein.
    We set $N \coloneqq \sup_{x\in\Gz}\sup_{\gamma\in\G_{x}}\sharp B_{\G_{x}}(\gamma, s+\frac{1}{L})$ and $\widetilde{\varepsilon} = \min\{\frac{\varepsilon}{8MN},1\}$, and choose an integer $\widetilde{n}>s+\frac{1}{L}+n_0$.


    Since $\G$ is amenable, it follows from Lemma~\ref{lem:p.d.functions for amenability} that there exists a non-negative function $h\in C_{c}(\G)$ of positive type such that
    \begin{itemize}
        \item $h(x)\leq 1$ for any $x\in \Gz$, and $h(x)=1$ for any $x\in \r(K_{\widetilde{n}})$;
        \item $|1- h(\gamma)|<\widetilde{\varepsilon}$ for any $\gamma \in K_{\widetilde{n}}$.
    \end{itemize}
    For each $x\in\Gz$ and $\gamma, \beta \in \G_x$, we define $k_x(\gamma,\beta) = h(\gamma\beta^{-1})$ and
    \[
        \widetilde{k}_{x}(\gamma,\beta) = \left\{\begin{array}{ll}
            h(\gamma\beta^{-1}),   & \text{if } \r(\gamma),\r(\beta)\in \r(K_{\widetilde{n}}); \\
            \delta_{\gamma,\beta}, & \text{otherwise}.
        \end{array}
        \right.
    \]
    The proof of Lemma \ref{lem:amen implies unif. A} shows that $\widetilde{k}_{x}$ is a normalised symmetric kernel on $\G_x$ of positive type and $(\widetilde{n},\widetilde{\varepsilon})$-variation such that the family $(\widetilde{k}_{x})_{x\in \Gz}$ has uniformly finite propagation.

    For each $x\in \Gz$, we consider the operator $T_{x}' = m_{\widetilde{k}_{x}}(\Phi_{x}(T))-\Phi_{x}(T)$ where $m_{\widetilde{k}_{x}}$ is the Schur multiplier.
    It is clear that
    \[
        \sup_{x\in\Gz}\| T_{x}' \|\le 2M \quad \text{and} \quad T_{x}'\in Commut_{\G_{x}}(L, \frac{\varepsilon}{24}).
    \]
    For the family $\{T'_{x}\}_x$ and the parameters above, Lemma \ref{lem:Lemma 5.2 in SZ20} implies that for each $x\in \Gz$ there exists a unit vector $v_{x}\in\ell^{2}(\G_{x})$ with $\diam(\supp(v_{x}))\leq s$ such that
    \begin{equation}\label{ineq1}
        \|T'_{x} v_{x}\|\geq \|T'_{x}\| - \frac{\varepsilon}{4}.
    \end{equation}

    Setting $V_{x} = \{ \gamma\in \G_{x}: d_{x}(\supp(v_{x}),\gamma)<\frac{1}{L}\}$ for each $x\in \Gz$, we have
    \begin{equation}\label{ineq2}
        \|T'_{x} v_{x} \|\leq \|\chi_{V_{x}}(T'_{x} v_{x})\|+\|\chi_{\G_{x}\backslash V_{x}}(T'_{x}v_{x})\|\leq \|\chi_{V_{x}}(T'_{x}v_{x})\|+\frac{\varepsilon}{24},
    \end{equation}
    where the second inequality comes from $T'_{x}\in Commut_{\G_{x}}(L,\frac{\varepsilon}{24})$.
    For $\gamma\in V_{x}$, we have
    \begin{align*}
        |(T'_{x}v_{x})(\gamma)|
         & = \left|\sum_{\alpha\in \supp(v_{x})} \left((\widetilde{k}_{x}(\gamma,\alpha) - 1 )\Phi_{x}(T)_{\gamma,\alpha}\right)v_{x}(\alpha)\right| \\
         & \leq \left(\sum_{\alpha\in \supp(v_{x})}\left|(\widetilde{k}_{x}(\gamma,\alpha) - 1 )\right|^{2}\right)^{1/2}\|\Phi_{x}(T)v_{x}\|
        \le \widetilde{\varepsilon}\sqrt{N} M \leq\frac{\varepsilon}{8\sqrt{N}},
    \end{align*}
    where the second inequality comes from the Cauchy-Schwarz inequality and the third one is due to that $\widetilde{k}_{x}$ has $(\widetilde{n},\widetilde{\varepsilon})$-variation. Here $\Phi_{x}(T)_{\gamma,\alpha} = \langle\delta_\gamma, \Phi_{x}(T)(\delta_\alpha)\rangle$, as defined in Section \ref{ssec.motivation}.
    Hence
    \begin{equation}\label{ineq3}
        \|\chi_{V_{x}}(T'_{x}v_{x})\|^{2}=\sum_{\gamma\in V_{x}}|(T'_{x}v_{x})(\gamma)|^2\leq \frac{\varepsilon^{2}}{64}.
    \end{equation}
    Combining (\ref{ineq1}), (\ref{ineq2}) and (\ref{ineq3}) together, we obtain that for any $x\in \Gz$,
    \begin{equation}\label{EQ:estimate2}
        \|m_{\widetilde{k}_{x}}(\Phi_{x}(T))-\Phi_{x}(T)\| = \|T'_{x}\|\leq \|T'_{x}v_{x}\|+\frac{\varepsilon}{4}\leq \|\chi_{V_{x}}(T'_{x}v_{x})\|+\frac{\varepsilon}{24}+\frac{\varepsilon}{4}< \frac{\varepsilon}{2}.
    \end{equation}

    On the other hand, it follows from Proposition~\ref{prop.leftConvolverAndGEquivariant} that $T=\Lambda(f_{T})$ for some $f_{T}\in C_{b}(\G)$.
    Hence for $x\in \Gz$ and $\gamma,\alpha\in\G_{x}$ we have
    \[
        (m_{k_{x}}(\Phi_{x}(T)))_{\gamma,\alpha}  = k_{x}(\gamma,\alpha)f_{T}(\gamma\alpha^{-1})
        = h(\gamma\alpha^{-1})f_{T}(\gamma\alpha^{-1}),
    \]
    which implies that
    \begin{equation}\label{EQ:realisation for kx}
        m_{k_{x}}(\Phi_{x}(T)) = \Phi_{x}(\Lambda(h\cdot f_{T})).
    \end{equation}
    Set $C_{x}\coloneqq\{ \gamma\in\G_{x}: \r(\gamma)\in \r(K_{\widetilde{n}})\}$ and $D_{x}\coloneqq\G_{x}\backslash C_{x}$, and decompose the operator $m_{k_{x}}(\Phi_{x}(T)) - m_{\widetilde{k}_{x}}(\Phi_{x}(T))$ as follows:
    \begin{align*}
         & m_{k_{x}}(\Phi_{x}(T)) - m_{\widetilde{k}_{x}}(\Phi_{x}(T))                                                                                                                       = (\chi_{C_{x}} +\chi_{D_{x}})\left(m_{k_{x}-\widetilde{k}_{x}}(\Phi_{x}(T))\right)(\chi_{C_{x}} +\chi_{D_{x}}) \\
         & = m_{k_{x}-\widetilde{k}_{x}}\left(\chi_{C_{x}}\Phi_{x}(T)\chi_{C_{x}}+\chi_{D_{x}}\Phi_{x}(T)\chi_{D_{x}}+\chi_{C_{x}}\Phi_{x}(T)\chi_{D_{x}}+\chi_{D_{x}}\Phi_{x}(T)\chi_{C_{x}}\right)                                                                                                         \\
         & = m_{k_{x}-\widetilde{k}_{x}}\left(\chi_{D_{x}}\Phi_{x}(T)\chi_{D_{x}}+\chi_{C_{x}}\Phi_{x}(T)\chi_{D_{x}}+\chi_{D_{x}}\Phi_{x}(T)\chi_{C_{x}}\right),
    \end{align*}
    where we use the fact that $k_x$ and $\widetilde{k}_x$ coincide on $C_x$ in the last equality.

    Note that for $\gamma,\beta\in D_{x}$, we have $\gamma\beta^{-1}\notin K_{\widetilde{n}}$ (otherwise, we would have $\r(\gamma) = \r(\gamma\beta^{-1}) \in \r(K_{\widetilde{n}})$, which contradicts to that $\gamma \in D_x$).
    Since $\widetilde{n}>n_0$, we have $\|\chi_{D_{x}}\Phi_{x}(T)\chi_{D_{x}}\|<\frac{\varepsilon}{12}$. Similarly for $\gamma\in D_{x}$ and $\beta\in C_{x}$, we have $\gamma\beta^{-1}\notin K_{\widetilde{n}}$. This implies that $\|\chi_{D_{x}}\Phi_{x}(T)\chi_{C_{x}}\|<\frac{\varepsilon}{12}$ and $\|\chi_{C_{x}}\Phi_{x}(T)\chi_{D_{x}}\|<\frac{\varepsilon}{12}$. Hence
    \[
        \|\chi_{D_{x}}\Phi_{x}(T)\chi_{D_{x}}+\chi_{C_{x}}\Phi_{x}(T)\chi_{D_{x}}+\chi_{D_{x}}\Phi_{x}(T)\chi_{C_{x}}\| < \frac{\varepsilon}{4},
    \]
    which implies that
    \begin{equation}\label{EQ:estimate3}
        \|m_{k_{x}}(\Phi_{x}(T)) - m_{\widetilde{k}_{x}}(\Phi_{x}(T))                                                                                                   \| < \frac{\varepsilon}{2}.
    \end{equation}

    Finally combining (\ref{EQ:estimate2}), (\ref{EQ:realisation for kx}) and (\ref{EQ:estimate3}), we obtain
    \[
        \|\Lambda(h \cdot f_T) - T\| = \sup_{x\in \Gz} \|m_{k_{x}}(\Phi_{x}(T)) - \Phi_{x}(T)\| < \varepsilon.
    \]
    Note that $h \cdot f_T \in C_c(\G)$, hence we conclude that $T \in C^*_r(\G)$.
\end{proof}

\section{Examples}\label{sec.example}

In this section, we apply Theorem \ref{Main-theorem} to the examples mentioned in Section \ref{ssec:pre ex}.

\subsection{Discrete groups}

Let $\Gamma$ be a countable discrete group with unit $1_\Gamma$. As mentioned in Section \ref{sssec:groups}, the reduced groupoid $C^*$-algebra of $\Gamma$ is the same as the reduced group $C^*$-algebra $C^*_r(\Gamma)$, and $\L(L^2(\Gamma)) = \B(\ell^2(\Gamma))$.

Writing in the matrix form, obviously an operator $T\in \B(\ell^2(\Gamma))$ is $\Gamma$-equivariant \emph{if and only if} $T_{\alpha \gamma, \beta\gamma} = T_{\alpha, \beta}$ for any $\alpha,\beta,\gamma\in \Gamma$.
Equipping $\Gamma$ with an arbitrary proper word length metric, it is clear that $T$ is compactly supported \emph{if and only if} $T$ has finite propagation in the sense of Definition \ref{defn:finite ppg}.
Moreover, the notion of quasi-locality for $T$ from Definition \ref{defn:quasi-locality} coincides with that from Definition \ref{def.operatorCompactSupportAndQuasiLocal}(3). Hence in this case, Theorem \ref{Main-theorem} recovers the following known result:
\begin{theorem}
    Let $\Gamma$ be a countable discrete amenable group. Then for an operator $T\in \B(\ell^2(\Gamma))$, the following are equivalent:
    \begin{enumerate}
        \item $T$ belongs to $C^{*}_r(\Gamma)$;
        \item $T$ is $\Gamma$-invariant and belongs to $C^*_u(\Gamma)$;
        \item $T$ is $\Gamma$-invariant and quasi-local.
    \end{enumerate}
\end{theorem}

\begin{remark}
    The above theorem also holds for groups with the \emph{approximation property} from \cite{MR1220905} (see also \cite{BrownOzawa}). In fact, it was shown in \cite{Zacharias2006} that $C^*_u(\Gamma)^{\Gamma}=C^{*}_r(\Gamma)$ if $\Gamma$ has the approximation property. Moreover, it follows from \cite{MR1220905, MR1763912} that amenable groups have the approximation property, and the approximation property implies Property A. Therefore combining with \cite[Theorem 3.3]{SpakulaZhang}, we obtain that if $\Gamma$ has the approximation property, then $C^{*}_r(\Gamma)=C^*_u(\Gamma)^{\Gamma}=C^*_{uq}(\Gamma)^{\Gamma}$.
\end{remark}

\subsection{Pair groupoids}

Let $X$ be a set and we consider the pair groupoid $X \times X$ from Section \ref{sssec:pair groupoids}. By Theorem \ref{cor:char for LL2G}, we know that the map
\[
    \Phi: \L(L^2(X \times X))^{X \times X} \longrightarrow \prod_{x\in X} \B(\ell^2(X))
\]
is a $C^*$-monomorphism with image consisting of constant families $(T')_{x\in X}$ which are vector-wise uniformly quasi-local. It is not hard (similar to, but much easier than, the analysis in Example \ref{eg:coarse groupoid for LL2G}) to see that $(T')_{x\in X}$ is vector-wise uniformly quasi-local for any $T' \in \B(\ell^2(X))$. Hence similar to Lemma \ref{lem:char for LL2G coarse groupoids}, we obtain the following:

\begin{lemma}
    Let $X$ be a set and $X \times X$ be the pair groupoid. Then for some fixed $x_0 \in X$, the map
    \begin{equation}\label{EQ:Theta pair again}
        \Theta: \L(L^2(X \times X))^{X \times X} \longrightarrow \B(\ell^2(X)), \quad T \mapsto \Phi_{x_0}(T)
    \end{equation}
    is a $C^*$-isomorphism. Note that $\Theta$ is independent of the choice of $x_0$.
\end{lemma}

Analogous to the discussion in Example \ref{eg:coarse groupoid left convolver}, the map $\Theta$ in (\ref{EQ:Theta pair again}) extends the map $\Theta: C^*_r(X \times X) \cong \K(\ell^2(X))$ in (\ref{EQ:Theta pair}). Hence we abuse the notation.

Direct calculations show that an operator $T \in \L(L^2(X \times X))^{X \times X}$ is quasi-local in the sense of Definition \ref{def.operatorCompactSupportAndQuasiLocal}(3) if and only if $\Phi_0(T) \in \K(\ell^2(X))$, which implies that $C^*_{uq}(X \times X)^{X \times X} = C^*_r(X \times X)$. Hence Theorem \ref{Main-theorem} holds trivially for the case of pair groupoids.

\subsection{Coarse groupoids}\label{ssec:coarse groupoids}

This example is our motivation of the whole work, and details have already been spread in Section \ref{sec.charForOperator} and Section \ref{sec.quasi}. Here we recall the whole picture again.

Let $(X,d)$ be a discrete metric space with bounded geometry, and $G(X)$ be the associated coarse groupoid. Recall from Lemma \ref{lem:char for LL2G coarse groupoids} that for a fixed $x_0 \in X$, the map
\[
    \Theta: \L(L^2(G(X)))^{G(X)} \longrightarrow \{T \in \B(\ell^2(X)): T \text{ satisfies Equation (\ref{EQ:char for LL2G coarse groupoid})}\}
\]
given by $T \mapsto \Phi_{x_0}(T)$ is a $C^*$-isomorphism. Example \ref{eg:coarse groupoid left convolver} implies that $\Theta$ extends the isomorphism
\[
    C^*_r(G(X)) \cong C^*_u(X)
\]
from (\ref{EQ:Theta for coarse groupoid}). Furthermore, Example \ref{ex:coarse groupoid again} implies that the restriction
\[
    \Theta: C^*_{uq}(G(X))^{G(X)} \longrightarrow C^*_{uq}(X)
\]
is also an isomorphism. Therefore, our main result recovers \cite[Theorem 3.3]{SpakulaZhang} in the Hilbert space case.

\subsection{Transformation groupoids}\label{ssec:transformation groupoid}
Let $X$ be a locally compact $\sigma$-compact space, and $\Gamma$ be a countable discrete group acting on $X$. We consider the transformation groupoid $X\rtimes \Gamma$ in Section \ref{sssec:transformation}.


First note that $(X \rtimes \Gamma)_x$ is identified with $\Gamma$ by $(\alpha x, \alpha) \mapsto \alpha$ for $x\in X$ and $\alpha \in \Gamma$. Under these identifications and equipping $\Gamma$ with a proper word length metric, we have the following. The proof is straightforward, hence omitted.

\begin{lemma}\label{lem:trans char}
    For a family $(T_x)_{x\in X}\in \prod_{x\in X}\B(\ell^2(\Gamma))$, we have:
    \begin{enumerate}
        \item Writing in the matrix form, $(T_x)_{x\in X}$ is $(X \rtimes \Gamma)$-equivariant \emph{if and only if} the following holds:
              \begin{equation}\label{EQ:equiv for trans}
                  (T_{\gamma x})_{\alpha, \beta} = (T_x)_{\alpha \gamma, \beta \gamma} \quad \text{for any }  x \in X \text{ and }\alpha, \beta, \gamma \in \Gamma.
              \end{equation}
        \item $(T_x)_{x\in X}$ belongs to the image of $\iota$ from (\ref{EQ:iota}) \emph{if and only if} the map $X \to \B(\ell^2(\Gamma)), x \mapsto T_x$ is WOT-continuous.
        \item When $X$ is compact, $(T_x)_{x\in X}$ is compactly uniformly quasi-local \emph{if and only if} $(T_x)_{x\in X}$ is uniformly quasi-local in the sense of Definition \ref{defn:operators family version}(3).
    \end{enumerate}
\end{lemma}

Note that condition (3) above also holds for certain non-compact $X$, while for simplicity we only focus on the compact case in the sequel. Combining Corollary \ref{cor:char for quasi-locality} with Lemma \ref{lem:trans char}, we reach the following:

\begin{corollary}
    When $X$ is compact, we have a $C^*$-monomorphism
    \[
        \Phi: C^*_{uq}(X \rtimes \Gamma)^{X \rtimes \Gamma} \longrightarrow \prod_{x\in X}\B(\ell^2(\Gamma)), \quad T \mapsto (\Phi_x(T))_{x\in X}
    \]
    with range consisting of elements $(T_x)_{x\in X}$ which is uniformly quasi-local and satisfies (\ref{EQ:equiv for trans}) such that the map $x\mapsto T_x$ is WOT-continuous.
\end{corollary}

Hence applying Theorem \ref{Main-theorem}, we reach the following characterisation for the reduced crossed product $C(X) \rtimes_r \Gamma$:

\begin{theorem}
    Let $X$ be a compact space and $\Gamma$ be a countable discrete group with an amenable action on $X$. Then $C(X) \rtimes_r \Gamma$ is $C^*$-isomorphic to the $C^*$-subalgebra in $\prod_{x\in X}\B(\ell^2(\Gamma))$ consisting of elements $(T_x)_{x\in X}$ which is uniformly quasi-local and satisfies (\ref{EQ:equiv for trans}) such that the map $x\mapsto T_x$ is WOT-continuous.
\end{theorem}

\section{Beyond equivariance}\label{sec:beyond equivariance}

In Section \ref{sec.mainThm}, we study the equivariant parts in uniform Roe and quasi-local algebras for groupoids, and our main result (Theorem \ref{Main-theorem}) provides a sufficient condition to ensure that they are the same. In this section, we will turn to these $C^*$-algebras themselves without the restriction of equivariance. Our approach is to consider the semi-direct product groupoids and transfer the general case to the equivariant one, inspired by the discussions in \cite[Section 6]{ADExact}.

\subsection{Semi-direct product groupoids}\label{sssec:groupoid actions}
Here we collect some prelimenary knowledge on groupoid actions and semi-direct product groupoids, and guide readers to \cite{ADExact} for more details.

Let $X$ be a locally compact space. A \emph{fibre space over $X$} is a pair $(Y,p)$, where $Y$ is a locally compact space and $p\colon Y\rightarrow X$ is a continuous surjective map. For two fibre spaces $(Y_1,p_1)$ and $(Y_2,p_2)$ over $X$, we form their \emph{fibred product} to be
\[
    Y_1\prescript{}{p_1}{\ast}_{p_2}Y_2 \coloneqq \{(y_1, y_2) \in Y_1 \times Y_2: p_1(y_1) = p_2(y_2)\},
\]
equipped with the topology induced by the product topology.


\begin{definition}[{\cite[Definition 1.2]{ADExact}}]
    Let $\G$ be a locally compact groupoid.
    A \textit{left $\G$-space} is a fibre space $(Y,p)$ over $\Gz$, equipped with a continuous map $(\gamma,y)\mapsto \gamma y$ from $\G\prescript{}{\s}{*}_{p}Y$ to $Y$ (called a \emph{ $\G$-action on $(Y,p)$}) which satisfies the following:
    \begin{itemize}
        \item $p(\gamma y) = \r(\gamma)$ for $(\gamma,y)\in \G\prescript{}{\s}{*}_{p}Y$, and $p(y)y = y$;
        \item $\gamma_{2}(\gamma_{1}y) = (\gamma_{2}\gamma_{1})y$ for $(\gamma_{1},y)\in \G\prescript{}{\s}{*}_{p}Y$ and $\s(\gamma_{2}) = \r(\gamma_{1})$.
    \end{itemize}
\end{definition}

Given a $\G$-space $(Y, p)$, the associated \textit{semi-direct product groupoid} $Y\rtimes\G$ is defined to be $Y\prescript{}{p}{*}_{\r}\G$ as a topological space. For $(y,\gamma) \in Y\rtimes\G$, define its range to be $(y,\r(\gamma))$ and its source to be $(\gamma^{-1}y,\s(\gamma))$. The product and inverse are given by
\[
    (y,\gamma)(\gamma^{-1}y,\gamma')=(y,\gamma\gamma') \quad \text{and} \quad (y,\gamma)^{-1} = (\gamma^{-1}y,\gamma^{-1}).
\]
Clearly $(y,p(y))\mapsto y$ is a homeomorphism from the unit space of $Y\rtimes\G$ onto $Y$. Hence from now on, we regard Y as the unit space of the groupoid $Y\rtimes\G$. As shown in \cite[Proposition 1.4 and Proposition 1.6]{ADExact}, the range map $\r : Y\rtimes\G \to Y$ is always open, and $Y\rtimes\G$ is \'{e}tale when the groupoid $\G$ itself is \'{e}tale.


The class of groupoid actions we are interested in come from certain compactification of the given groupoid. For completeness, let us recall the following:

\begin{definition}\label{defn:fibrewise cpt}
    Let $(Y,p)$ be a fibre space over a locally compact space $X$.
    \begin{enumerate}
        \item We say that $(Y,p)$ is \textit{fibrewise compact} if $p$ is proper, \emph{i.e.}, $p^{-1}(K)$ is compact for any compact $K \subseteq X$.
        \item A \textit{fibrewise compactification} of $(Y,p)$ is a fibrewise compact fibre space $(Z,q)$ together with a continuous map $\varphi\colon Y\rightarrow Z$, which is a homeomorphism onto an open dense subset of $Z$ and satisfies $q \circ \varphi = p$. Usually we regard $Y$ as a subset of $Z$ in this case, and the morphism $\varphi$ is just the inclusion.
    \end{enumerate}
\end{definition}

To introduce concrete fibrewise compactifications, we need refer to the Gelfand spectra. Let $(Y,p)$ be a fibre space over  a locally compact space $X$. Denote
\[
    p^*(C_0(X)) \coloneqq \{ f\circ p:  f\in C_{0}(X) \}
\]
and $C_{0}(Y,p)$ to be the closure of the following set in $C_{b}(Y)$:
\[
    C_{c}(Y,p) = \{ g\in C_{b}(Y): \text{ there is a compact } K\subset X,\,\supp(g)\subset p^{-1}( K )\}.
\]
It was shown in \cite[Proposition A.2]{ADExact} that the fibrewise compactifications of $(Y,p)$ are in one-to-one correspondence with the $C^{*}$-subalgebras of $C_{0}(Y,p)$ containing $p^{*}\left(C_{0}(X)\right) + C_{0}(Y)$. The fibrewise compactification associated to $C_{0}(Y,p)$ is called the \textit{fibrewise Stone-\v{C}ech compactification of $(Y,p)$}, denoted by $(\beta_{p}Y,p_{\beta})$. 

As a special case, we consider the following: Let $\G$ be a locally compact \'{e}tale groupoid. Then the groupoid multiplication of $\G$ provides a $\G$-action on the fibre space $(\G,\r)$, \emph{i.e.}, $\G\prescript{}{\s}{*}_{\r}\G \to \G$ defined by $(\gamma_{1},\gamma_{2})\mapsto \gamma_{1}\gamma_{2}$. Moreover, it was shown in \cite[Proposition 2.5]{ADExact} that this $\G$-action extends in a unique way to a $\G$-action on the fibrewise Stone-\v{C}ech compactification $\left(\beta_{\r} \G, \r_{\beta}\right)$. Hence we can consider the semi-direct product groupoid $\beta_{\r} \G\rtimes\G$.

Following the terminology in \cite{ADExact}, a locally compact \'{e}tale groupoid $\G$ is said to be \emph{strongly amenable at infinity} if $\beta_{\r} \G\rtimes\G$ is amenable. As shown in \cite{ADExact}, this has a close relation with the notion of exactness for groupoids. By definition, $\G$ is called \emph{$C^*$-exact} if the reduced groupoid $C^*$-algebra $C^*_r(\G)$ is exact. Also recall that $\G$ is said to be \emph{weakly inner amenable} if for any compact subset $K \subseteq \G$ and any $\varepsilon>0$, there exists a continuous bounded positive definite function $f$ on the product groupoid $\G \times \G$ which is properly supported such that $|f(\gamma,\gamma)-1|<\varepsilon$ for all $\gamma \in K$.


\begin{proposition}[{\cite[Corollary 7.4 and Theorem 8.6]{ADExact}}]\label{exact prop 1}
    Let $\G$ be a locally compact, second countable and \'{e}tale groupoid. If $\G$ is strongly amenable at infinity, then $\G$ is $C^*$-exact. Conversely, if $\G$ is weakly inner amenable and $C^*$-exact, then $\G$ is strongly amenable at infinity.
\end{proposition}

Throughout the rest of this section, let $\G$ be a locally compact \'{e}tale groupoid and $(\beta_\r \G, \r_\beta)$ be the fibrewise Stone-\v{C}ech compactification of $(\G,\r)$. Consider the $\G$-action on $(\beta_\r \G, \r_\beta)$ induced by the groupoid multiplication in $\G$, and the associated semi-direct product groupoid $\beta_\r \G \rtimes \G$. In the next few subsections, we will transfer the uniform Roe and quasi-local algebras for $\G$ to the equivariant part of the corresponding algebras for $\beta_\r \G \rtimes \G$.

\subsection{Compactly supported case}\label{ssec:compact supp case AD}
Here we deal with the case of compactly supported operators and the uniform Roe algebra, which was originally studied in \cite{ADExact}.

First we show that our notion of the uniform Roe algebra $C^*_u(\G)$ from Definition \ref{defn:unif. Roe and quasilocal for groupoids}(1) coincides with that of the uniform $C^*$-algebra from \cite[Definition 6.1]{ADExact}. To start, let us recall an extra algebra $C_t(\G \ast_{\s} \G)$ from \cite{ADExact} (see also \cite{AZ_2020_article_a}). By definition,
\[
    \G \ast_{\s} \G \coloneqq \G \prescript{}{\s}{*}_{\s} \G = \{(\gamma, \alpha) \in \G \times \G: \s(\gamma) = \s(\alpha)\}.
\]
A \emph{tube} is a subset of $\G \ast_{\s} \G$ whose image by the map $(\gamma, \alpha) \mapsto \gamma\alpha^{-1}$ is relatively compact in $\G$. Denote by $C_t(\G \ast_{\s} \G)$ the space of continuous bounded functions on $\G \ast_{\s} \G$ with support in a tube. Recall from \cite[Section 6]{ADExact} that there is a $\ast$-monomorphism $\mathcal{T}: C_t(\G \ast_{\s} \G) \rightarrow \L(L^2(\G))$ given by\footnote{Note that the formula given in \cite{ADExact} is slightly different from the above, where the set $\G \ast_{\r} \G$ is considered instead of $\G \ast_{\s} \G$. Here we follow the formula from \cite[Section 6.5]{AZ_2020_article_a}, which is more convenient and compatible to our setting.}:
\[
    ((\mathcal{T} f)\xi)(\gamma):=\sum_{\alpha\in \G_{\s(\gamma)}} f(\gamma, \alpha) \xi(\alpha) \quad \mbox{where} \quad \xi \in L^2(\G) \mbox{~and~}\gamma \in \G,
\]
and the \emph{uniform $C^*$-algebra of $\G$} is defined to be the norm closure of $\mathcal{T}(C_t(\G \ast_{\s} \G))$ in $\L(L^2(\G))$.

The following lemma relates the algebra $C_t(\G \ast_{\s} \G)$ to our notion of compact supported operators:

\begin{lemma}\label{lem:AD=Cu}
    The image $\mathcal{T}(C_t(\G \ast_{\s} \G))$ consists of all compactly supported operators in the sense of Definition \ref{def.operatorCompactSupportAndQuasiLocal}(2). Hence the uniform $C^*$-algebra of $\G$ from \cite[Definition 6.1]{ADExact} coincides with the uniform Roe algebra of $\G$ from Definition \ref{defn:unif. Roe and quasilocal for groupoids}(1).
\end{lemma}

\begin{proof}
    First it is clear that given $f \in C_t(\G \ast_{\s} \G)$ with support in $\{(\gamma, \alpha) \in \G \ast_{\s} \G: \gamma\alpha^{-1} \in K\}$ for some compact $K \subseteq \G$, then the operator $\mathcal{T}f$ has support in $K$. Conversely, for an operator $T \in \L(L^2(\G))$ with support in some compact $K \subseteq \G$, we define a function $f: \G \ast_{\s} \G \to \CC$ as follows: for $(\gamma, \alpha) \in \G \ast_{\s} \G$ with $\s(\gamma) = \s(\alpha) =x$, set $f(\gamma, \alpha)\coloneqq \langle \delta_\gamma, \Phi_x(T)\delta_\alpha \rangle_{\ell^2(\G_x)}$. Thanks to Corollary \ref{cor:char for quasi-locality}, it is easy to show that $f \in C_t(\G \ast_{\s} \G)$ and $\mathcal{T}(f)=T$. Hence we conclude the proof.
\end{proof}

Next, we recall from \cite[Lemma 6.3]{ADExact} that there is a $\ast$-algebraic isomorphism $\vartheta: C_c(\beta_{\r} \G \rtimes \G) \to C_t(\G \ast_{\s} \G)$ given by
\[
    \vartheta(f)(\gamma,\alpha):=f(\gamma,\gamma\alpha^{-1}) \quad \mbox{for} \quad f \in C_c(\beta_{\r} \G \rtimes \G).
\]
Hence we obtain a $\ast$-monomorphism
\[
    \theta\coloneqq \mathcal{T} \circ \vartheta: C_c(\beta_{\r} \G \rtimes \G) \to \L(L^2(\G)).
\]
Direct calculations show that
\begin{equation}\label{EQ:theta2}
    \big(\theta(f)\xi\big)(\gamma) = \sum_{\alpha\in\G_{\s(\gamma)}} f(\gamma,\gamma\alpha^{-1})\xi(\alpha)  \text{ for } f \in C_c(\beta_{\r} \G \rtimes \G) \text{ and } \xi\in L^{2}(\G).
\end{equation}
It is proved in \cite[Theorem 6.4]{ADExact} that $\theta$ extends to a $C^*$-isomorphism
\begin{equation*}
    \Theta: C^*_r(\beta_{\r} \G \rtimes \G) \stackrel{\cong}{\longrightarrow} C^*_u(\G).
\end{equation*}
Combining with Proposition \ref{Roereduced}, we obtain a $\ast$-isomorphism (still denoted by $\Theta$):
\begin{equation}\label{EQ:Theta}
    \Theta: \overline{\CC_u[\beta_{\r} \G \rtimes \G]^{\beta_{\r} \G \rtimes \G}}^{\L(L^2(\beta_{\r} \G \rtimes \G))} \stackrel{\cong}{\longrightarrow} C^*_u(\G)
\end{equation}
such that $\Theta(\Lambda(f)) = \theta(f)$ for $f\in C_c(\beta_{\r} \G \rtimes \G)$.


\subsection{From equivariance to non-equivariance}\label{ssec:eq to neq}

Now we aim to extend the map $\Theta$ from (\ref{EQ:Theta}) to all $(\beta_{\r} \G \rtimes \G)$-equivariant operators in $\L(L^2(\beta_{\r} \G \rtimes \G))$. The main observation is that the operator $\theta(f)$ for $f \in C_c(\beta_{\r} \G \rtimes \G)$ defined in (\ref{EQ:theta2}) can be written as a composition of three maps, which will be explained in details.

First recall that the unit space of $\beta_{\r} \G \rtimes \G$ is $\beta_{\r} \G$ and note that the source fibre at $z\in\beta_{\r}\G$ is
\[
    (\beta_{\r}\G\rtimes\G)_{z} = \{(\gamma z, \gamma): \gamma \in \G \text{ such that }\s(\gamma) = \r_{\beta}(z)\},
\]
which is bijectively mapped onto $\G_{\r_{\beta}(z)}$ via the map
\[
    p_z: (\beta_{\r}\G\rtimes\G)_{z} \longrightarrow \G_{\r_{\beta}(z)}, \quad (\gamma z, \gamma)\mapsto \gamma.
\]
Hence in the following, we will write $(\gamma z, \gamma) \in \beta_{\r}\G\times\G$ such that $\s(\gamma) = \r_{\beta}(z)$ for a general element in $\beta_{\r} \G \rtimes \G$. Furthermore, $p_z$ induces a unitary
\[
    P_z: \ell^2((\beta_{\r}\G\rtimes\G)_{z}) \longrightarrow \ell^2(\G_{\r_{\beta}(z)}) \quad \text{by} \quad P_z(\delta_{(\gamma z, \gamma)}) = \delta_\gamma,
\]
which further induces a $\ast$-isomorphism
\begin{equation}\label{EQ:Ad}
    \Ad_{P_z}: \B(\ell^2((\beta_{\r}\G\rtimes\G)_{z})) \longrightarrow \B(\ell^2(\G_{\r_{\beta}(z)})), \quad T \mapsto P_z T P_z^*.
\end{equation}

Let us consider the map $\iota: C_c(\G) \to C_c(\beta_{\r} \G \rtimes \G)$ defined by
\[
    (\iota \xi)(\gamma z, \gamma) \coloneqq \xi(\gamma) \quad \text{for} \quad (\gamma z, \gamma) \in \beta_{\r} \G \rtimes \G \quad \text{and} \quad \xi \in C_c(\G).
\]
This map is well-defined since for $\xi \in C_c(\G)$ with support in a compact set $K \subseteq \G$, then $\iota \xi$ has support in $\r_\beta^{-1}(\r(K)) \times K$, which is compact in $\beta_{\r} \G \rtimes \G$. Moreover for $\xi\in C_c(\G)$, we have
\[
    \|\iota \xi\|^2_{L^2(\beta_{\r} \G \rtimes \G)} = \sup_{z\in \beta_\r \G} \|(\iota \xi)|_{(\beta_{\r} \G \rtimes \G)_z}\|^2 = \sup_{z\in \beta_{\r} \G} \sum_{\gamma \in \G_{\r_{\beta}(z)}} |\xi(\gamma)|^2 = \|\xi\|^2_{L^2(\G)}.
\]
Hence $\iota$ can be extended to an isometry (with the same notation):
\[
    \iota: L^2(\G) \longrightarrow L^2(\beta_{\r} \G \rtimes \G).
\]

On the other hand, we define another map
\[
    \kappa: C_c(\beta_{\r} \G \rtimes \G) \longrightarrow C_c(\G) \quad \text{by} \quad (\kappa \eta)(\gamma) \coloneqq \eta(\gamma, \gamma)
\]
for $\eta \in C_c(\beta_{\r} \G \rtimes \G)$ and $\gamma \in \G$. It is clear that $\kappa$ is well-defined. Moreover, we have:
\begin{align*}
    \|\kappa \eta\|^2_{L^2(\G)} & = \sup_{x\in \Gz} \sum_{\gamma \in \G_x} |(\kappa \eta)(\gamma)|^2 = \sup_{x\in \Gz} \sum_{\gamma \in \G_x} |\eta(\gamma,\gamma)|^2         \\
                                & \leq \sup_{z\in \beta_{\r} \G} \sum_{\gamma \in \G_{\r_{\beta}(z)}} |\eta(\gamma z,\gamma)|^2 = \|\eta\|^2_{L^2(\beta_{\r} \G \rtimes \G)}.
\end{align*}
Hence $\kappa$ can be extended to a contractive map (with the same notation):
\[
    \kappa: L^2(\beta_{\r} \G \rtimes \G) \longrightarrow L^2(\G).
\]

Now we are in the position to extend the map $\Theta$ from (\ref{EQ:Theta}) to all $(\beta_{\r} \G \rtimes \G)$-equivariant operators in $\L(L^2(\beta_{\r} \G \rtimes \G))$. More precisely, we define
\begin{equation*}
    \Theta: \L(L^2(\beta_\r \G \rtimes \G))^{\beta_\r \G \rtimes \G} \to \L(L^2(\G))
\end{equation*}
by
\begin{equation}\label{EQ:formula for Theta(T)}
    \Theta(T):=\kappa \circ T \circ \iota \quad \text{for} \quad T \in \L(L^2(\beta_\r \G \rtimes \G))^{\beta_\r \G \rtimes \G}.
\end{equation}
In other words, the operator $\Theta(T)$ makes the following diagram commutes:
\[
    \begin{tikzcd}
        L^{2}(\G)\arrow[d,"\Theta(T)"'] \arrow[r,"\iota"] &
        L^{2}(\beta_{\mathrm{r}}\G\rtimes\G)\arrow[d,"T"] \\
        L^{2}(\G)                                        &
        L^{2}(\beta_{\mathrm{r}}\G\rtimes\G).\arrow[l,"\kappa"']
    \end{tikzcd}
\]
We will show in Lemma \ref{lem:Theta being a homomorphism} below that the image of $\Theta$ is indeed contained in $\L(L^2(\G))$. First, note that Proposition \ref{prop.leftConvolverAndGEquivariant} implies that $T \in \L(L^2(\beta_\r \G \rtimes \G))^{\beta_\r \G \rtimes \G}$ can be written as $T=\Lambda(f_T)$ where $f_T$ is the associated left convolver. Direct calculations show that $\Theta(T)$ satisfies the following:
\begin{equation}\label{EQ:Theta(T)}
    \big(\Theta(T)\xi\big)(\gamma) = \sum_{\alpha\in\G_{\s(\gamma)}} f_T(\gamma,\gamma\alpha^{-1})\xi(\alpha)  \quad \text{ for } \quad \xi\in L^{2}(\G).
\end{equation}
Comparing (\ref{EQ:theta2}) with (\ref{EQ:Theta(T)}), it is clear that $\Theta$ in (\ref{EQ:formula for Theta(T)}) indeed extends the isomorphism $\Theta: C^*_r(\beta_{\r} \G \rtimes \G) \stackrel{\cong}{\longrightarrow} C^*_u(\G)$ in (\ref{EQ:Theta}). Hence again we use the same notation.

\begin{remark}\label{rem:Theta0512}
It is pointed out by the anonymous referee that the map $\Theta$ from (\ref{EQ:formula for Theta(T)}) can be defined in a more tidy way using interior tensor products of Hilbert $C^*$-modules as follows. Let $\tau: \Gz \to \beta_{\r}\G$ be the inclusion map, which induces a $*$-homomorphism $\tau^*:C_{0}(\beta_{\r}\G)\rightarrow C_{0}(\Gz)$. It is easy to see that $L^{2}(\beta_{\r}\G\rtimes\G)$ is isomorphic to the interior tensor product $L^2(\G) \otimes_{\r_\beta^*} C_0(\beta_\r \G)$ as Hilbert $C_0(\beta_\r \G)$-modules. Hence the interior tensor product $L^{2}(\beta_{\r}\G\rtimes\G)\otimes_{\tau^*}C_{0}(\Gz)$ is isomorphic to $L^{2}(\G)$ as Hilbert $C_{0}(\Gz)$-modules due to that the composition $\Gz \stackrel{\tau}{\to} \beta_{\r}\G \stackrel{\r_\beta}{\to} \Gz$ is the identity. Moreover, direct calculations show that $\Theta(T)$ in (\ref{EQ:formula for Theta(T)}) is identified with $T\otimes_{\tau^*}1$ under the isomorphism above. 
%
\end{remark}

On the other hand, using (\ref{EQ:Theta(T)}) it is straightforward to prove:

\begin{lemma}\label{lem:Theta being a homomorphism}
    With the same notation as above, we have:
    \begin{enumerate}
        \item for $T \in \L(L^2(\beta_\r \G \rtimes \G))^{\beta_\r \G \rtimes \G}$, the operator $\Theta(T)$ is an adjointable operator on $L^2(\G)$ with adjoint $\Theta(T^*)$;
        \item for $T,S \in \L(L^2(\beta_\r \G \rtimes \G))^{\beta_\r \G \rtimes \G}$, we have $\Theta(TS) = \Theta(T)\Theta(S)$;
        \item $\Theta$ is injective.
    \end{enumerate}
\end{lemma}

Note that item (1) and (2) can be deduced directly from the alternative viewpoint of $\Theta(T)$ given in Remark \ref{rem:Theta0512}. However, we still need (\ref{EQ:Theta(T)}) to conclude item (3).

Combining the analysis above, we reach the following:

\begin{corollary}\label{cor:Theta conclusion}
    The map $\Theta: \L(L^2(\beta_\r \G \rtimes \G))^{\beta_\r \G \rtimes \G} \to \L(L^2(\G))$ given by (\ref{EQ:formula for Theta(T)}) is a $C^*$-monomorphism, which extends the isomorphism in (\ref{EQ:Theta}).
\end{corollary}

\begin{remark}
    A careful reader might already notice that the map $\Theta$ can be further extended to $\L(L^2(\beta_\r \G \rtimes \G))$ using the same formula (\ref{EQ:formula for Theta(T)}). However in this case, it is unclear whether Lemma \ref{lem:Theta being a homomorphism} still holds since (\ref{EQ:Theta(T)}) is no longer available. Due to the same reason, it is unclear whether $\Theta(T)$ coincides with $T\otimes_{\tau^*}1$ given in Remark \ref{rem:Theta0512} for a general $T \in \L(L^2(\beta_\r \G \rtimes \G))$.
\end{remark}

Finally, we provide an extra viewpoint on $\Theta$ in terms of the slicing maps from Section \ref{ssec:operators in LL2G}, which will be used later to characterise the quasi-local algebra. Recall from (\ref{eq.PhiMap}) that we have the slicing maps:
\begin{equation}\label{EQ:Phi cross}
    \Phi^{\rtimes} \colon \L(L^2(\beta_\r \G \rtimes \G))^{\beta_\r \G \rtimes \G} \longrightarrow \prod_{z\in \beta_\r \G}\B(\ell^2((\beta_\r \G \rtimes \G)_z)), \quad T \mapsto (\Phi^{\rtimes}_z(T))_{z\in \beta_\r \G}
\end{equation}
and
\begin{equation}\label{EQ:Phi general}
    \Phi \colon \L(L^2(\G)) \longrightarrow \prod_{x\in \Gz}\B(\ell^2(\G_x)), \quad T' \mapsto (\Phi_x(T'))_{x\in \Gz}.
\end{equation}
For $T \in \L(L^2(\beta_\r \G \rtimes \G))^{\beta_\r \G \rtimes \G}$ and $x\in \Gz$, we consider the operator $\Ad_{P_x}(\Phi^{\rtimes}_{x}(T)) \in \B(\ell^2(\G_x))$ where $\Ad_{P_x}$ is from (\ref{EQ:Ad}).
Using (\ref{EQ:Theta(T)}), it is straightforward to prove:

\begin{lemma}\label{lem:slice char for Theta}
    With the same notation as above, for $T \in \L(L^2(\beta_\r \G \rtimes \G))^{\beta_\r \G \rtimes \G}$ and $x\in \Gz$ we have
    \[
        \Phi_x(\Theta(T)) = \Ad_{P_x}(\Phi^{\rtimes}_{x}(T)).
    \]
\end{lemma}

\begin{remark}\label{rem:Theta}
    In fact, the map $\Theta$ can be alternatively defined using the slicing maps above. More precisely, consider the map
    \begin{equation}\label{EQ:projection}
        \mathfrak{p}: \prod_{z\in \beta_\r \G}\B(\ell^2((\beta_\r \G \rtimes \G)_z)) \longrightarrow \prod_{x\in \Gz}\B(\ell^2(\G_x)), \quad (T_z)_{z\in \beta_\r \G} \mapsto (\Ad_{P_x}(T_x))_{x\in \Gz}.
    \end{equation}
    Given $T \in \L(L^2(\beta_\r \G \rtimes \G))^{\beta_\r \G \rtimes \G}$, we consider the family $\mathfrak{p} \circ \Phi^\rtimes (T)$. Then we can apply Theorem \ref{cor:char for LL2G} to verify that $\mathfrak{p} \circ \Phi^\rtimes (T)$ belongs to $\Phi(\L(L^2(\G)))$, and hence it provides an operator $\Phi^{-1}(\mathfrak{p} \circ \Phi^\rtimes (T))$, which is defined to be $\Theta(T)$. More details will be provided in the next subsection, where Lemma \ref{lem:slice char for Theta} shows that these two definitions are the same.
\end{remark}

\subsection{Quasi-local case}

In this subsection, we consider the quasi-local algebras and the main result is the following:

\begin{proposition}\label{prop:Theta quasi-local}
    The restriction of the map $\Theta$ given by (\ref{EQ:formula for Theta(T)}) provides a $C^*$-isomorphism:
    \[
        \Theta: C^*_{uq}(\beta_\r \G \rtimes \G)^{\beta_\r \G \rtimes \G} \stackrel{\cong}{\longrightarrow} C^*_{uq}(\G).
    \]
\end{proposition}

The proof is divided into several parts, making use of the slicing maps at the end of Section \ref{ssec:eq to neq}.
First recall from Section \ref{ssec:operators in LL2G} and \ref{ssec:dense subset}, the operator fibre spaces are given by
\[
    E^{\rtimes}\coloneqq \bigsqcup_{z\in \beta_\r \G} \B(\ell^2((\beta_\r \G \rtimes \G)_z)), \quad E^{\rtimes}_{\G}\coloneqq \bigsqcup_{z\in \G} \B(\ell^2((\beta_\r \G \rtimes \G)_z))
\]
and
\[
    E\coloneqq \bigsqcup_{x\in \Gz} \B(\ell^2(\G_x)),
\]
equipped with the topology from Definition \ref{defn:operator fibre space} (note that $\G$ is a dense subset in $\beta_\r \G$). Also denote by $\Gamma_b(E^{\rtimes}), \Gamma_b(E^{\rtimes}_{\G})$ and $\Gamma_b(E)$ the associated algebras of continuous bounded sections. Similar to Lemma \ref{Stone-Cech-extension}, we have the following:

\begin{lemma}\label{lem:fibre SC extension}
    Every section $\sigma \in \Gamma_b(E^{\rtimes}_{\G})$ is extendable.
\end{lemma}

The proof is almost the same as the one for Lemma \ref{Stone-Cech-extension}. In fact, borrowing the notation therein, we can shrink the neighbourhoods $U_{\gamma'}$ and $U_{\gamma''}$ to ensure that the function $f_{\gamma', \gamma''} \in C_0(\G,r)$ (see Section \ref{sssec:groupoid actions} to recall the definition). Hence we omit the details.

On the other hand, note that for the slicing maps (\ref{EQ:Phi cross}) and (\ref{EQ:Phi general}), given $T \in \L(L^2(\beta_\r \G \rtimes \G))^{\beta_\r \G \rtimes \G}$ the property of equivariance means the following:
\begin{equation*}
    V_{(\gamma z, \gamma)} \Phi^{\rtimes}_z(T) = \Phi^{\rtimes}_{\gamma z}(T) V_{(\gamma z, \gamma)} \quad \text{for} \quad (\gamma z, \gamma) \in \beta_\r \G \rtimes \G,
\end{equation*}
where $V_{(\gamma z, \gamma)}: \ell^2((\beta_\r \G \rtimes \G)_z) \to \ell^2((\beta_\r \G \rtimes \G)_{\gamma z})$ is given by $V_{(\gamma z, \gamma)}(\delta_{(\alpha z,\alpha)}) = \delta_{(\alpha z, \alpha \gamma^{-1})}$ (see Section \ref{ssec.equivariant}). In the following, we will consider families $(T_z)_{z\in \G}$ satisfying:
\begin{equation}\label{EQ:equiv for Phi}
    V_{(\gamma z, \gamma)} T_z = T_{\gamma z} V_{(\gamma z, \gamma)} \quad \text{for} \quad (\gamma z, \gamma) \in \G \rtimes \G.
\end{equation}

Applying Corollary \ref{rem:char for dense case} together with Lemma \ref{lem:fibre SC extension}, we obtain the following characterisations for the equivariant quasi-local algebras:

\begin{corollary}\label{cor:char for quasi-local cross}
    The map
    \[
        \res \circ \Phi^{\rtimes}: C^*_{uq}(\beta_\r \G \rtimes \G)^{\beta_\r \G \rtimes \G} \longrightarrow \prod_{z\in \G}\B(\ell^2((\beta_\r \G \rtimes \G)_z)), \quad T \mapsto (\Phi^{\rtimes}_z(T))_{z \in \G}
    \]
    is a $C^*$-isomorphism with image consisting of $(T_z)_{z\in \G}$ such that $z \mapsto T_z$ is a continuous section of $E_\G^{\rtimes}$, $(T_z)_{z\in \G}$ is compactly uniformly quasi-local and satisfies (\ref{EQ:equiv for Phi}).
\end{corollary}

\begin{corollary}\label{cor:char for quasi-local}
    The map
    \[
        \Phi: C^*_{uq}(\G) \longrightarrow \prod_{x\in \Gz}\B(\ell^2(\G_x)), \quad T \mapsto (\Phi_x(T))_{x \in \Gz}
    \]
    is a $C^*$-isomorphism with image consisting of $(T_x)_{x\in \Gz}$ such that $x \mapsto T_x$ is a continuous section of $E$ and $(T_x)_{x\in \Gz}$ is compactly uniformly quasi-local.
\end{corollary}

Recall from Lemma \ref{lem:slice char for Theta} and Remark~\ref{rem:Theta} that we have the following commutative diagram:
\[
    \begin{tikzcd}[row sep = huge, column sep = huge]
        \L(L^2(\beta_\r \G \rtimes \G))^{\beta_\r \G \rtimes \G}\arrow[d,"\Theta"'] \arrow[r,"\res \circ \Phi^{\rtimes}"] &
        \prod_{z\in \G}\B(\ell^2((\beta_\r \G \rtimes \G)_z))  \arrow[d,"\mathfrak{p}"] \\
        \L(L^{2}(\G))     \arrow[r,"\Phi"]                                    &
        \prod_{x\in \Gz}\B(\ell^2(\G_x))
    \end{tikzcd}
\]
Hence to prove Proposition \ref{prop:Theta quasi-local}, it suffices to verify whether the map $\mathfrak{p}$ provides a bijection between the images of quasi-local algebras characterised by Corollary \ref{cor:char for quasi-local cross} and Corollary \ref{cor:char for quasi-local}. First note that the equivariance condition (\ref{EQ:equiv for Phi}) directly implies the following:

\begin{lemma}\label{lem:p is bijective}
    The following restriction of $\mathfrak{p}$ is a bijection:
    \[
        \mathfrak{p}: \{(T_z)_{z\in \G} \in \prod_{z\in \G}\B(\ell^2((\beta_\r \G \rtimes \G)_z)) \text{ satisfying } (\ref{EQ:equiv for Phi})\} \longrightarrow \prod_{x\in \Gz}\B(\ell^2(\G_x)).
    \]
\end{lemma}

Now we turn to the condition of continuity for sections:

\begin{lemma}\label{lem:continuity of sections}
    For a family $(T_z)_{z\in \G} \in \prod_{z\in \G}\B(\ell^2((\beta_\r \G \rtimes \G)_z))$ satisfying (\ref{EQ:equiv for Phi}), the map $z\mapsto T_z$ is a continuous section of $E^{\rtimes}_{\G}$ \emph{if and only if} the map $x\mapsto \Ad_{P_x}(T_x)$ is a continuous section of $E$.
\end{lemma}

\begin{proof}
    By definition, the map $z\mapsto T_z$ is continuous if and only if for any $z_i \to z$, $\gamma_i' \to \gamma'$ and $\gamma''_i \to \gamma''$ in $\G$ with $\r(z_i) = \s(\gamma'_i) = \s(\gamma''_i)$ and $\r(z) = \s(\gamma') = \s(\gamma'')$, then
    \begin{equation}\label{EQ:equiv convg}
        \langle \delta_{(\gamma''_i z_i, \gamma''_i)}, T_{z_i}(\delta_{(\gamma'_i z_i, \gamma'_i)}) \rangle \to \langle \delta_{(\gamma'' z, \gamma'')}, T_{z}(\delta_{(\gamma' z, \gamma')}) \rangle
    \end{equation}
    Setting $\s(z_i)= x_i$ and $\s(z)=x$, then (\ref{EQ:equiv for Phi}) implies that
    \[
        V_{(z_i,z_i)} T_{x_i} V_{(z_i,z_i)}^* = T_{z_i} \quad \text{and} \quad V_{(z,z)} T_{x} V_{(z,z)}^* = T_{z}.
    \]
    Hence (\ref{EQ:equiv convg}) can be rewritten as:
    \[
        \langle \delta_{(\gamma''_i z_i, \gamma''_i z_i)}, T_{x_i}(\delta_{(\gamma'_i z_i, \gamma'_i z_i)}) \rangle \to \langle \delta_{(\gamma'' z, \gamma'' z)}, T_{x}(\delta_{(\gamma' z, \gamma'z)}) \rangle.
    \]
    Applying the isomorphisms $\Ad_{P_x}$, the above is also equivalent to:
    \[
        \langle \delta_{\gamma''_i z_i}, \Ad_{P_{x_i}}(T_{x_i})(\delta_{\gamma'_i z_i}) \rangle \to \langle \delta_{\gamma'' z}, \Ad_{P_x}(T_{x})(\delta_{\gamma' z}) \rangle,
    \]
    which is (by definition) nothing but the continuity of the map $x\mapsto \Ad_{P_x}(T_x)$.
\end{proof}

Finally, we consider the condition of compactly uniform quasi-locality.

\begin{lemma}\label{lem:cuql of sections}
    A family $(T_z)_{z\in \G} \in \prod_{z\in \G}\B(\ell^2((\beta_\r \G \rtimes \G)_z))$ satisfying (\ref{EQ:equiv for Phi}) is compactly uniformly quasi-local \emph{if and only if} the family $(\Ad_{P_x}(T_x))_{x\in \Gz}$ is compactly uniformly quasi-local.
\end{lemma}

\begin{proof}
    \emph{Necessity}: By definition, given $\varepsilon>0$ there exists a compact subset $\tilde{K}\subseteq \beta_\r \G \rtimes \G$ such that for any $z\in \G$ and $\tilde{A}_{z}, \tilde{B}_{z}\subseteq (\beta_\r \G \rtimes \G)_z$ with $\tilde{A}_{z}\cap (\tilde{K} \cdot \tilde{B}_z)=\emptyset$ and $(\tilde{K}\cdot \tilde{A}_z) \cap \tilde{B}_{z}=\emptyset$, we have $\| \chi_{\tilde{A}_{z}}T_z\chi_{\tilde{B}_{z}}\| < \varepsilon$. Take
    \[
        K\coloneqq \{\gamma\in \G: \text{ there exists } \gamma' \in \G \text{ such that }(\gamma', \gamma) \in \tilde{K}\},
    \]
    which is compact. Then for any $x\in \Gz$ and $A_x, B_x \subseteq \G_x$ with $A_x \cap (K \cdot B_x) = \emptyset$ and $(K\cdot A_x) \cap B_x=\emptyset$, then we have $p_x^{-1}(A_x) \cap (\tilde{K} \cdot p_x^{-1}(B_x)) = \emptyset$ and $(\tilde{K} \cdot p_x^{-1}(A_x)) \cap p_x^{-1}(B_x) = \emptyset$. Hence
    \[
        \varepsilon \geq \|\chi_{p_x^{-1}(A_x)} T_{x} \chi_{p_x^{-1}(B_x)}\| = \|\Ad_{P_x}(\chi_{p_x^{-1}(A_x)} T_{x} \chi_{p_x^{-1}(B_x)})\| = \|\chi_{A_x} \Ad_{P_x}(T_{x}) \chi_{B_x}\|,
    \]
    which shows that the family $(\Ad_{P_x}(T_x))_{x\in \Gz}$ is compactly uniformly quasi-local.

    \emph{Sufficiency}: Given $\varepsilon>0$ there exists a compact subset $K\subseteq \G$ such that for any $x\in \Gz$ and $A_{x}, B_{x}\subseteq \G_x$ with $A_{x}\cap (K \cdot B_x)=\emptyset$ and $(K\cdot A_x) \cap B_{x}=\emptyset$, we have $\| \chi_{A_{x}}\Ad_{P_x}(T_x)\chi_{B_{x}}\| < \varepsilon$. Take
    \[
        \tilde{K} \coloneqq \{(\gamma', \gamma) \in \beta_\r \G \rtimes \G: \gamma \in K\},
    \]
    which is compact. Then for any $z\in \G$ and $\tilde{A}_{z}, \tilde{B}_{z}\subseteq (\beta_\r \G \rtimes \G)_z$ with $\tilde{A}_{z}\cap (\tilde{K} \cdot \tilde{B}_z)=\emptyset$ and $(\tilde{K}\cdot \tilde{A}_z) \cap \tilde{B}_{z}=\emptyset$, we set $x=\s(z)$ and note that $V_{(z,z)}T_xV^*_{(z,z)} = T_{z}$. Hence we have
    \begin{align}\label{EQ:calculation ql}
      \|\chi_{\tilde{A}_{z}} T_z \chi_{\tilde{B}_{z}}\| & = \|(V^*_{(z,z)} \chi_{\tilde{A}_{z}} V_{(z,z)}) \cdot T_x \cdot (V^*_{(z,z)} \chi_{\tilde{B}_{z}} V_{(z,z)})\| \nonumber \\
                                                        & = \|\chi_{\tilde{A}_{z} \cdot (z,z)} \cdot T_x \cdot \chi_{\tilde{B}_{z} \cdot (z,z)}\| \nonumber                         \\
                                                        & = \|\Ad_{P_x}(\chi_{\tilde{A}_{z} \cdot (z,z)}) \cdot \Ad_{P_x}(T_x) \cdot \Ad_{P_x}(\chi_{\tilde{B}_{z} \cdot (z,z)})\|.
    \end{align}

    It is straightforward to check that $\Ad_{P_x}(\chi_{\tilde{A}_{z} \cdot (z,z)}) = \chi_{A_x}$ and $\Ad_{P_x}(\chi_{\tilde{B}_{z} \cdot (z,z)}) = \chi_{B_x}$, where
    \[
        A_x \coloneqq \{\alpha \in \G_x: (\alpha, \alpha) \in \tilde{A}_{z} \cdot (z,z)\} = \{\alpha \in \G_x: (\alpha, \alpha z^{-1}) \in \tilde{A}_{z}\}
    \]
    and similarly, $B_x\coloneqq \{\beta \in \G_x: (\beta, \beta z^{-1}) \in \tilde{A}_{z}\}$. If $A_x \cap (K \cdot B_x)  \neq \emptyset$, then there exists $\alpha \in A_x$, $\beta \in B_x$ and $\gamma \in K$ such that $\alpha = \gamma \beta$. Then we have
    \[
        (\alpha, \alpha z^{-1}) = (\alpha, \gamma) \cdot (\beta, \beta z^{-1}),
    \]
    which is a contradiction to that $\tilde{A}_{z}\cap (\tilde{K} \cdot \tilde{B}_z)=\emptyset$. Similarly, we have $(K\cdot A_x) \cap B_{x}=\emptyset$. Hence combining with (\ref{EQ:calculation ql}), we obtain:
\begin{eqnarray*}
        \|\chi_{\tilde{A}_{z}} T_z \chi_{\tilde{B}_{z}}\| &=& \|\Ad_{P_x}(\chi_{\tilde{A}_{z} \cdot (z,z)}) \cdot \Ad_{P_x}(T_x) \cdot \Ad_{P_x}(\chi_{\tilde{B}_{z} \cdot (z,z)})\| \\
        &=& \|\chi_{A_x} \Ad_{P_x}(T_x) \chi_{B_x}\| \\
         &\leq & \varepsilon,
\end{eqnarray*}
    which concludes the proof.
\end{proof}

\begin{proof}[Proof of Proposition \ref{prop:Theta quasi-local}]
    Combining Corollary \ref{cor:char for quasi-local cross}, Corollary \ref{cor:char for quasi-local}, Lemma \ref{lem:p is bijective}, Lemma \ref{lem:continuity of sections} and Lemma \ref{lem:cuql of sections}, we conclude the proof.
\end{proof}

\begin{remark}
    In fact, the above procedure is also available for the case of compactly supported operators. Hence combining with Lemma \ref{lem:AD=Cu}, we obtain an alternative proof for \cite[Theorem 6.4]{ADExact}.
\end{remark}

\begin{remark}
    Readers might wonder whether Lemma \ref{lem:cuql of sections} still holds for vector-wise uniform quasi-locality. If fact using a similar idea in the proof of Lemma \ref{lem:cuql of sections}, we can show that if a family $(T_z)_{z\in \G} \in \prod_{z\in \G}\B(\ell^2((\beta_\r \G \rtimes \G)_z))$ satisfying (\ref{EQ:equiv for Phi}) is vector-wise uniformly quasi-local, then the family $(\Ad_{P_x}(T_x))_{x\in \Gz}$ is vector-wise uniformly quasi-local. Hence applying Theorem \ref{cor:char for LL2G} together with Lemma \ref{lem:continuity of sections}, we obtain that $\mathfrak{p} \circ \Phi^\rtimes (T)$ belongs to $\Phi(\L(L^2(\G)))$ for $T \in \L(L^2(\beta_\r \G \rtimes \G))^{\beta_\r \G \rtimes \G}$. As already mentioned in Remark \ref{rem:Theta}, this provides an alternative approach to define the map $\Theta$.

    However, it is unclear whether the opposite direction still holds for vector-wise uniform quasi-locality. If it is true, then $\Theta$ will provide a $C^*$-isomorphism between $\L(L^2(\beta_\r \G \rtimes \G))^{\beta_\r \G \rtimes \G}$ and $\L(L^2(\G))$.
\end{remark}

\subsection{Conclusion}

Combining the discussions in the previous subsections, we obtain that the $C^*$-monomorphism
\[
    \Theta: \L(L^2(\beta_\r \G \rtimes \G))^{\beta_\r \G \rtimes \G} \to \L(L^2(\G))
\]
provides $C^*$-isomorphisms
\[
    \overline{\CC_u[\beta_{\r} \G \rtimes \G]^{\beta_{\r} \G \rtimes \G}}^{\L(L^2(\beta_{\r} \G \rtimes \G))} \big(\cong C^*_r(\beta_{\r} \G \rtimes \G) \big) \cong C^*_u(\G)
\]
and
\[
    C^*_{uq}(\beta_{\r} \G \rtimes \G)^{\beta_{\r} \G \rtimes \G} \cong C^*_{uq}(\G).
\]
Hence applying Theorem~\ref{Main-theorem}, we reach the following quasi-local characterisation for the uniform Roe algebra $C^*_u(\G)$:

\begin{theorem}\label{thm:general case}
    Let $\G$ be a locally compact, $\sigma$-compact and \'{e}tale groupoid. Suppose $\G$ is either strongly amenable at infinity, or secondly countable weakly inner amenable and $C^*$-exact. Then we have $C^*_u(\G) = C^*_{uq}(\G)$.
\end{theorem}

\bibliographystyle{plain}

\bibliography{QLGC}


\end{document}